\documentclass[letterpaper,11pt,oneside,reqno]{amsart}
\usepackage{amsfonts,amsmath, amssymb,amsthm,amscd}
\usepackage[usenames,dvipsnames]{color}
\usepackage{bm}
\usepackage{mathrsfs}
\usepackage{yfonts}
\usepackage{verbatim}
\usepackage{hyperref}
\usepackage{graphicx}
\usepackage[utf8]{inputenc}
\usepackage{latexsym}
\usepackage{lscape}
\usepackage{epsfig}
\usepackage{subfigure}
\usepackage[left=1in, right=1in, top=1in, bottom=1in]{geometry}
\usepackage{dsfont}
\usepackage{tikz}
\usetikzlibrary{shapes.multipart}
\usetikzlibrary{patterns}
\usepackage{setspace}
\usepackage{mathtools}
\begin{document}
\newcommand{\cn}[1]{\overline{#1}}
\newcommand{\e}[0]{\epsilon}
\newcommand{\EE}{\ensuremath{\mathbb{E}}}
\newcommand{\qq}[1]{(q;q)_{#1}}
\newcommand{\A}{\ensuremath{\mathcal{A}}}
\newcommand{\GT}{\ensuremath{\mathbb{GT}}}
\newcommand{\PP}{\ensuremath{\mathbb{P}}}
\newcommand{\frakP}{\ensuremath{\mathfrak{P}}}
\newcommand{\frakQ}{\ensuremath{\mathfrak{Q}}}
\newcommand{\frakq}{\ensuremath{\mathfrak{q}}}
\newcommand{\R}{\ensuremath{\mathbb{R}}}
\newcommand{\Rplus}{\ensuremath{\mathbb{R}_{+}}}
\newcommand{\C}{\ensuremath{\mathbb{C}}}
\newcommand{\Z}{\ensuremath{\mathbb{Z}}}
\newcommand{\Weyl}[1]{\ensuremath{\mathbb{W}}^{#1}}
\newcommand{\Zgzero}{\ensuremath{\mathbb{Z}_{>0}}}
\newcommand{\Zgeqzero}{\ensuremath{\mathbb{Z}_{\geq 0}}}
\newcommand{\Zleqzero}{\ensuremath{\mathbb{Z}_{\leq 0}}}
\newcommand{\Q}{\ensuremath{\mathbb{Q}}}
\newcommand{\T}{\ensuremath{\mathbb{T}}}
\newcommand{\Y}{\ensuremath{\mathbb{Y}}}
\newcommand{\M}{\ensuremath{\mathbf{M}}}
\newcommand{\MM}{\ensuremath{\mathbf{MM}}}
\newcommand{\W}[1]{\ensuremath{\mathbf{W}}_{(#1)}}

\newcommand{\Real}{\ensuremath{\mathrm{Re}}}
\newcommand{\Imag}{\ensuremath{\mathrm{Im}}}
\newcommand{\re}{\ensuremath{\mathrm{Re}}}

\def \Ai {{\rm Ai}}
\def \sgn {{\rm sgn}}

\newcommand{\var}{{\rm var}}

\newtheorem{theorem}{Theorem}[section]
\newtheorem{partialtheorem}{Partial Theorem}[section]
\newtheorem{conj}[theorem]{Conjecture}
\newtheorem{lemma}[theorem]{Lemma}
\newtheorem{proposition}[theorem]{Proposition}
\newtheorem{corollary}[theorem]{Corollary}
\newtheorem{claim}[theorem]{Claim}
\newtheorem{KPZclass}[theorem]{KPZ class Conjecture}

\def\todo#1{\marginpar{\raggedright\footnotesize #1}}
\def\change#1{{\color{green}\todo{change}#1}}
\def\note#1{\textup{\textsf{\color{blue}(#1)}}}

\theoremstyle{definition}
\newtheorem{remark}[theorem]{Remark}
\theoremstyle{definition}
\newtheorem{example}[theorem]{Example}

\theoremstyle{definition}
\newtheorem{definition}[theorem]{Definition}

\theoremstyle{definition}
\newtheorem{definitions}[theorem]{Definitions}

\begin{abstract}
We introduce new integrable exclusion and zero-range processes on the one-dimensional lattice that generalize the $q$-Hahn TASEP and the $q$-Hahn Boson (zero-range) process introduced in \cite{povolotsky2013integrability} and further studied in \cite{corwin2014q}, by allowing jumps in both directions. Owing to a Markov duality, we prove moment formulas for the locations of particles in the exclusion process. This leads to a  Fredholm determinant formula that characterizes the distribution of the location of any particle. We show that the model-dependent constants that arise in the limit theorems predicted by the KPZ scaling theory are recovered by a steepest descent analysis of the Fredholm determinant.
For some choice of the parameters, our model specializes to the multi-particle-asymmetric diffusion model introduced in \cite{sasamoto1998one}. In that case, we make a precise asymptotic analysis that confirms KPZ universality predictions. Surprisingly, we also prove that in the partially asymmetric case, the location of the first particle also enjoys cube-root fluctuations which follow Tracy-Widom GUE statistics.
\end{abstract}

\title{The $q$-Hahn asymmetric exclusion process}

\author[G. Barraquand]{Guillaume Barraquand}
\address{G. Barraquand,
Columbia University,
Department of Mathematics,
2990 Broadway,
New York, NY 10027, USA}
\email{barraquand@math.columbia.edu}

\author[I. Corwin]{Ivan Corwin}
\address{I. Corwin, Columbia University,
Department of Mathematics,
2990 Broadway,
New York, NY 10027, USA,
and Clay Mathematics Institute, 10 Memorial Blvd. Suite 902, Providence, RI 02903, USA,
and Institut Henri Poincar\'e,
11 Rue Pierre et Marie Curie, 75005 Paris, France}
\email{ivan.corwin@gmail.com}

\maketitle
\setcounter{tocdepth}{1}
\tableofcontents
\hypersetup{linktocpage}

\section{Introduction}

The purpose of this paper is to introduce a new family of Bethe ansatz solvable exclusion and zero-range processes on the one-dimensional lattice $\mathbb{Z}$. Our construction generalizes the $q$-Hahn Boson (zero-range) process introduced in \cite{povolotsky2013integrability} and the $q$-Hahn TASEP further studied in \cite{corwin2014q}, by allowing jumps in both directions. Under mild assumptions on the microscopic dynamics, a wide class of interacting  particle systems are expected to lie in the KPZ universality class (see e.g. \cite{corwin2012kardar}). In particular, when started from step initial data, the positions of particles in the bulk of the rarefaction fan are expected to have cube-root scale fluctuations distributed according to Tracy-Widom type statistics, up to scaling constants depending on microscopic dynamics. Presently, these universality predictions can be confirmed only for a small number of exactly solvable models. Discovering a greater variety of analysable models, with more and more degrees of freedom, has a threefold interest:
\begin{enumerate}
\item To better understand the range of applicability of exact solvability,
\item To check the conjectural KPZ scaling theory on various integrable models, and expand the scope of the universality class,
\item To shed light on new phenomena beyond universality.
\end{enumerate}

 In this paper, we discover a new type of phenomena in presence of a jump discontinuity (anti-shock) of the system's hydrodynamic profile. For one particular exactly-solvable model that we call the \emph{MADM exclusion process}, we prove that fluctuations of the jump discontinuity (as measured by the location of the first particle in the system) are of order $t^{1/3}$ with limiting GUE Tracy-Widom statistics as $t$ goes to infinity. In other words, the first particle behaves exactly like particles deep in the rarefaction fan. We believe it is an interesting question to investigate how universal this scaling and limiting statistic is among systems which develop such jump discontinuity.

\subsection{MADM exclusion process}

 \label{subsec:motivations} The MADM exclusion process is a continuous-time Markov process on configurations of particles
$$ +\infty=x_0(t) >x_1(t) >x_2(t) > \dots >x_n(t) > \dots \ ;\ x_i\in \Z.$$
Fix $q\in (0,1)$ and $R>L>0$ such that $R+L=1$. The n$^{th}$ particle,  located at $x_n(t)$, jumps right to the location $x_n(t)+j$ at rate (i.e. according to independent exponentially distributed waiting times with rate) $R/[j]_{q^{-1}}$ for all $j\in \lbrace 1, \dots, x_{n-1}(t) - x_n(t) - 1\rbrace$,  and jumps left to the location $x_n(t)-j'$ at rate $L/[j']_{q}$ for all $j'\in \lbrace 1, \dots, x_{n}(t) - x_{n+1}(t) - 1\rbrace$. Here the $q$-deformed integers $[j]_{q^{-1}}$ and  $[j]_{q}$ are defined as
$$ [j]_{q^{-1}}  = 1+q^{-1}+ \dots + q^{1-j}, \ \ \ [j]_{q}  = 1+ q +\dots + q^{j-1}.$$
An example of some possible jumps is shown in Figure \ref{fig:defMADM}.
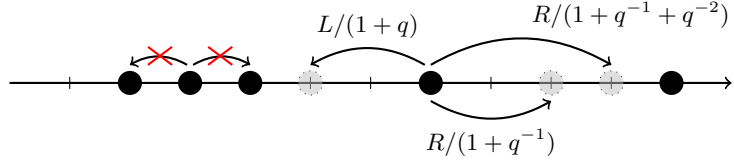
\begin{figure}[t]
\begin{center}
\begin{tikzpicture}[scale=0.8]
\draw[thick, ->] (-5,0) -- (7,0);
\foreach \k in {-4, ..., 6}
{\draw (\k, -0.1) -- (\k,0.1);}
\fill (-3,0) circle(0.2);
\fill (-2,0) circle(0.2);
\fill (-1,0) circle(0.2);
\fill (2,0) circle(0.2);
\draw[dotted] (5,0) circle(0.2);
\fill[black!25, opacity=0.5] (5,0) circle(0.2);
\draw[dotted] (4,0) circle(0.2);
\fill[black!25, opacity=0.5] (4,0) circle(0.2);
\draw[dotted] (0,0) circle(0.2);
\fill[black!25, opacity=0.5] (0,0) circle(0.2);
\fill (6,0) circle(0.2);
\draw[thick, ->]  (-1.95,0.3) to[bend left]  (-1,0.3) ;
\draw[thick, red] (-1.7,0.3)  -- (-1.3,0.6);
\draw[thick, red] (-1.7,0.6)  -- (-1.3,0.3);
\draw[thick, ->] node{} (-2.05,0.3) to[bend right] node{} (-3,0.3);
\draw[thick, red] (-2.7,0.3)  -- (-2.3,0.6);
\draw[thick, red] (-2.7,0.6)  -- (-2.3,0.3);
\draw[thick, ->] node{} (2,0.3) to[bend left] node[above right]{\footnotesize{$R/(1+q^{-1}+q^{-2})$}} (5,0.3);
\draw[thick, ->] node{} (2,-0.3) to[bend right] node[below]{\footnotesize{$R/(1+q^{-1})$}} (4,-0.3);
\draw[thick, ->]  (1.9,0.3) to[bend right] node[above]{\footnotesize{$L/(1+q)$}} (0,0.3) ;
\end{tikzpicture}
\end{center}
\caption{Rates of a few admissible jumps in the exclusion process corresponding to the multi-particle asymmetric diffusion model (MADM exclusion process). }
\label{fig:defMADM}
\end{figure}
The gaps of the system evolve according to the multi-particle asymmetric diffusion model (MADM), introduced by Sasamoto and Wadati \cite{sasamoto1998one} and studied therein in the context of Bethe ansatz diagonalizability.

Let us briefly review the hydrodynamic theory for the MADM exclusion process (see Section \ref{sec:KPZscaling} for more details).  The Bernoulli product measure with probability $\rho$ of having a particle at a site is stationary for the MADM exclusion process. Furthermore, one computes that the average steady-state current (or flux) $j(\rho)$ as a function of density $\rho$ is given by
$$j(\rho)  =  \rho \frac{1-q}{\log(q)^2} \left(\frac R q\ \Psi_q'\big(1+\log_q(1-\rho)\big) - L \ \Psi_q'\big(\log_q(1-\rho)\big) \right),
$$
where $\Psi_q'$ is the derivative of the $q$-digamma function (see Section \ref{sec:preliminaries}).
The function $\rho \mapsto j(\rho)$ is plotted in Figure \ref{fig:plotj}. For small densities, particles have a net drift to the left, whereas for larger densities particles have a net drift to the right.
\begin{figure}
\includegraphics[scale=0.4]{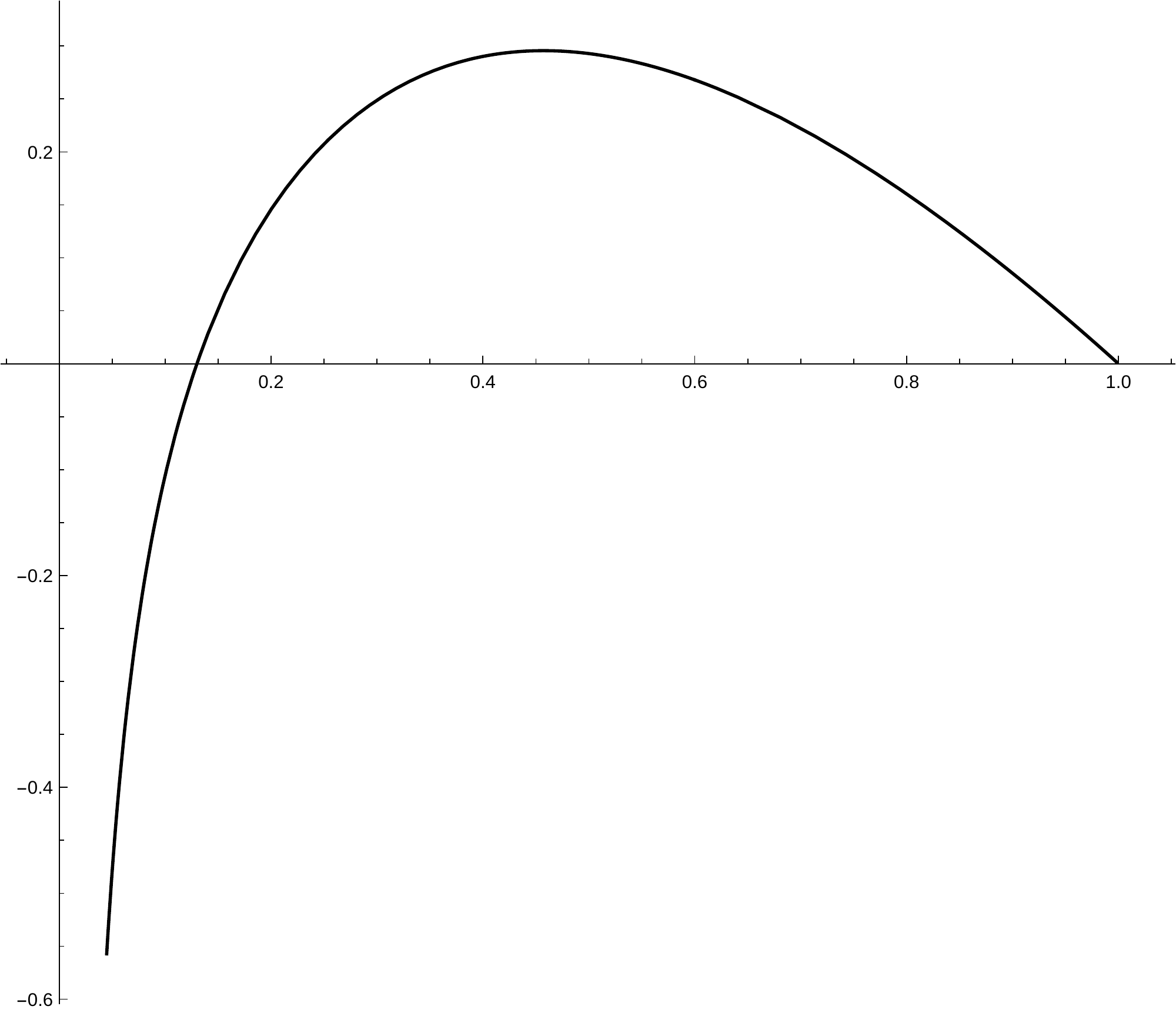}
\caption{Plot of the function $\rho \mapsto j(\rho)$ for $q=0.4$ and asymmetry parameters
$R=0.95=1-L$.}
\label{fig:plotj}
\end{figure}

When the system is started from the step initial condition, that is $x_n(0)=-n$, the locations of particles satisfy a law of large numbers.
Let $\theta >0$ parametrize the position we consider in the rarefaction fan (see Section \ref{sec:KPZscaling}), then we have that
\begin{equation}
\frac{x_{\lfloor \kappa(\theta)\rfloor}(t)}{t} \xrightarrow[t\to \infty]{} \pi(\theta)
\label{eq:llnintro}
\end{equation}
where $\kappa(\theta)$ and $\pi(\theta)$ are functions of $\theta$ defined by
\begin{equation}
\pi(\theta) = \frac{1-q}{\log(q)^2}
\left[\frac R q \left(\Psi_q'(\theta+1) - \frac{1-q^{\theta}}{q^{\theta}\log(q)}\Psi_q''(\theta+1) \right) - L \left(\Psi_q'(\theta) - \Psi_q''(\theta)\frac{1-q^{\theta}}{q^{\theta}\log(q)} \right) \right],
\label{eq:expressionforpiintro}
\end{equation}
and
\begin{equation}
\kappa(\theta) = \frac{1-q}{\log(q)^3} \frac{
(1-q^{\theta})^2}{q^{\theta}} \left(\frac{R}{q}\  \Psi_q''(\theta+1) - L\ \Psi_q''(\theta)\right).
\label{eq:expressoinforkappaintro}
\end{equation}
This hydrodynamic behaviour can also be phrased in terms of the limiting density profile.
Denoting by  $\rho(x)$ the local density of particles at time $t$ around site $xt$ for very large $t$, the law of large numbers \eqref{eq:llnintro} translates into the density profile shown in Figure \ref{fig:densityprofileintro}.
\begin{figure}
\begin{center}
\includegraphics[scale=0.5]{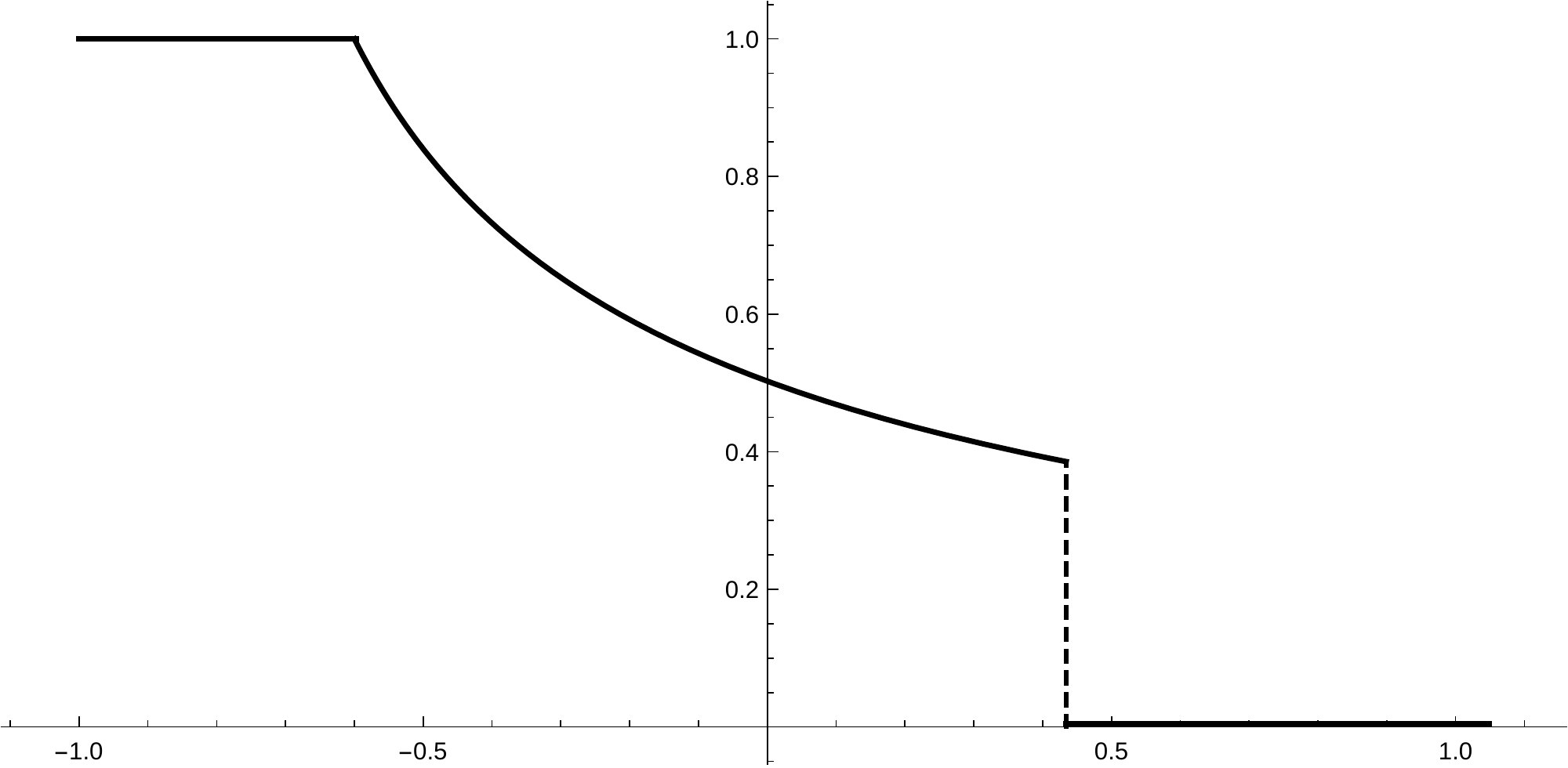}
\end{center}
\caption{Density profile $x\mapsto \rho(x)$ for a $q$-Hahn AEP with $q=\nu=0.6$, and asymmetry parameters $R=0.8$ and $L=0.2$, starting from step initial data.
It is such that $\rho(\pi(\theta)) = 1-q^{\theta}$.
}
\label{fig:densityprofileintro}
\end{figure}
The density profile in the partially asymmetric case (that is when $R>L>0$) is discontinuous on its right edge.  A simple argument explains why this discontinuity is present. Consider the behaviour of the first particle $x_1(t)$. The rate at which it jumps anywhere to the right is
$$ \sum_{j=1}^{\infty} \frac{R}{[j]_{q^{-1}}} <\infty .$$
whereas the rate at which it jumps anywhere to the left is
$$ \sum_{j=1}^m \frac{L}{[j]_q} \xrightarrow[m\to+\infty]{} +\infty, $$
where $m=x_1(t)-x_2(t)-1$.
Thus, even though particles want to generally move right (because $R>L$), the first particle stays with high probability at a bounded distance from the second particle, and hence the density around the first particles remains  strictly positive. In terms of the flux, this explains why $j(\rho)$ is negative for small $\rho$.

The property of the flux function $j(\rho)$ which is responsible for the occurrence and location of the  discontinuity in the density profile is the fact that the drift, that is $j(\rho)/\rho$, is not monotone as a function of $\rho$. Since we are starting with step initial data, the hydrodynamic limit $\rho(x)$ will be decreasing in $x$. As long as the drift $j\big(\rho(x)\big)/\rho(x)$ increases with $x$, the profile will fan out, but once the drift start decreasing, a jam will occur and the discontinuity will form at that $x$. Our limit theorems stated below confirm the result of this reasoning.


We consider the fluctuation behaviour in the rarefaction fan as well as the right edge jump behaviour. For particles in the bulk of the rarefaction fan, we prove that the limit behaviour matches the predictions for models in the KPZ universality class. In Section \ref{subsec:hydro} and \ref{subsec:sigma}, we also explain how  the model-dependent constants in this limit theorem are consistent with the physics KPZ scaling theory \cite{krug1992amplitude, spohn2012kpz}.
\begin{theorem}
\label{thm:fluctuationsintro2}
Consider the MADM exclusion process started from step initial condition, with $q\in (0,1)$ and asymmetry parameters $R$ and $L=1-R$ such that $R>L\geqslant 0$.
Assume that $\theta \in(0, +\infty)$ is such that $q^{\theta}>2q/(1+q)$, then there exists a constant $\sigma(\theta)>0$ such that for $n=\lfloor \kappa(\theta ) t\rfloor$,
$$ \lim_{t\to \infty} \mathbb{P}\left(\frac{x_n(t) - \pi(\theta)t}{\sigma(\theta) t^{1/3}} \geqslant x\right) = F_{\mathrm{GUE}}(-x).$$
The expressions of the model-dependent constants  $\kappa\left(\theta\right), \pi(\theta)$ and $\sigma(\theta)$ as functions of $\theta$  are given in \eqref{eq:expressoinforkappaintro}, \eqref{eq:expressionforpiintro} and \eqref{eq:expressionforsigma} and  $F_{\mathrm{GUE}}$ is the GUE Tracy-Widom distribution (see definition \ref{def:TW}).
\end{theorem}
Theorem \ref{thm:fluctuationsintro2} is proved  as Theorem \ref{thm:fluctuationsrarefactionfan} in Section \ref{sec:MADM}, and it implies the convergence
\eqref{eq:llnintro} in probability. The condition on $\theta$ should just be technical (though as it is, it restricts us to a right section of the rarefaction fan).
\begin{remark}
Lee recently posted a preprint on arXiv \cite{lee2014fredholm} where a similar asymptotic result is proposed for an infinite volume MADM which is different from the one discussed in the present paper. Although Theorem \ref{thm:fluctuationsintro2} is not in contradiction with \cite{lee2014fredholm}, the present authors pointed out fundamental issues in the proof. In particular, the weak law of large numbers implied by the limit theorem \cite[Theorem 1.3]{lee2014fredholm} does not agree with the particle dynamics considered. At the time of posting of the present article, no revision remedying these issues have been made.
\end{remark}

Turning to the right edge behaviour, let us first recall some of what is known for systems without jump discontinuities.
For TASEP (which is a special case of the MADM exclusion process when $R=1$, $L=0$, $q=0$) from step initial condition,
an application of the law of large numbers and the classical central limit theorem shows that as $t\to\infty$,
$$ \frac{x_1(t) -  t }{ \sqrt{t}} \longrightarrow \mathcal{N}(0,1).$$
For ASEP,  Theorem 2 in \cite{tracy2009asymptotics} shows that the position of the first particle still fluctuates on a $\sqrt{t}$ scale, but the limiting law is not Gaussian. Both TASEP and ASEP have no jump in their density profile $\rho(x)$ when started from step initial data. The $t^{1/2}$ scaling seems robust but the limit law not.


Turning back to the MADM exclusion process, we see that the occurrence of a jump discontinuity seems to radically change the first particles fluctuations.
\begin{theorem}
Consider the MADM exclusion process started from step initial condition with $q\in (0,1)$ and asymmetry parameters $R$ and $L=1-R$ such that $R_{min}(q)<R<1$, where $R_{min}(q)$ is an explicit bound depending on the parameter $q$ (see Theorem \ref{thm:fluctuationsparticle1} and Remark \ref{rem:restrictivecondition} for a precise expression).
Then,
$$ \lim_{t\to \infty} \mathbb{P}\left(\frac{x_1(t) - \pi t}{\sigma t^{1/3}} \geqslant x\right) = F_{\mathrm{GUE}}(-x),$$
where $\pi$ and $\sigma>0$ are explicit constants depending on $R$ and $q$, and $F_{\mathrm{GUE}}(x)$ is the GUE Tracy-Widom distribution (see Definition \ref{def:TW}).
\label{thm:fluctuationsintro1}
\end{theorem}
Theorem \ref{thm:fluctuationsintro1} is proved in Section \ref{sec:MADM} as Theorem \ref{thm:fluctuationsparticle1}. This shows that the first particles fluctuates in the same manner as those in the rarefaction fan.

It is tempting to ask whether this behaviour ($t^{1/3}$ scaling and $F_{\mathrm{GUE}}$ limit law) is universal in presence of a discontinuous density profile. We leave that question for further study \cite{barraquand2015preparation}.

\subsection{Duality and Bethe anzatz solvability}

The results announced  in Section \ref{subsec:motivations} are actually special cases of results we prove for a model that we introduce here and call the \emph{$q$-Hahn asymmetric exclusion process} (abbreviated  $q$-Hahn AEP). This process depends on two parameters $q\in (0,1), \nu \in [0,1)$ and asymmetry parameters $R, L \geqslant 0$. The $q$-Hahn AEP degenerates to many known exactly-solvable processes. For instance one recovers the MADM exclusion process when $\nu=q$. Figure \ref{fig:chart} summarizes these degenerations (see Section \ref{subsec:degenerations}).

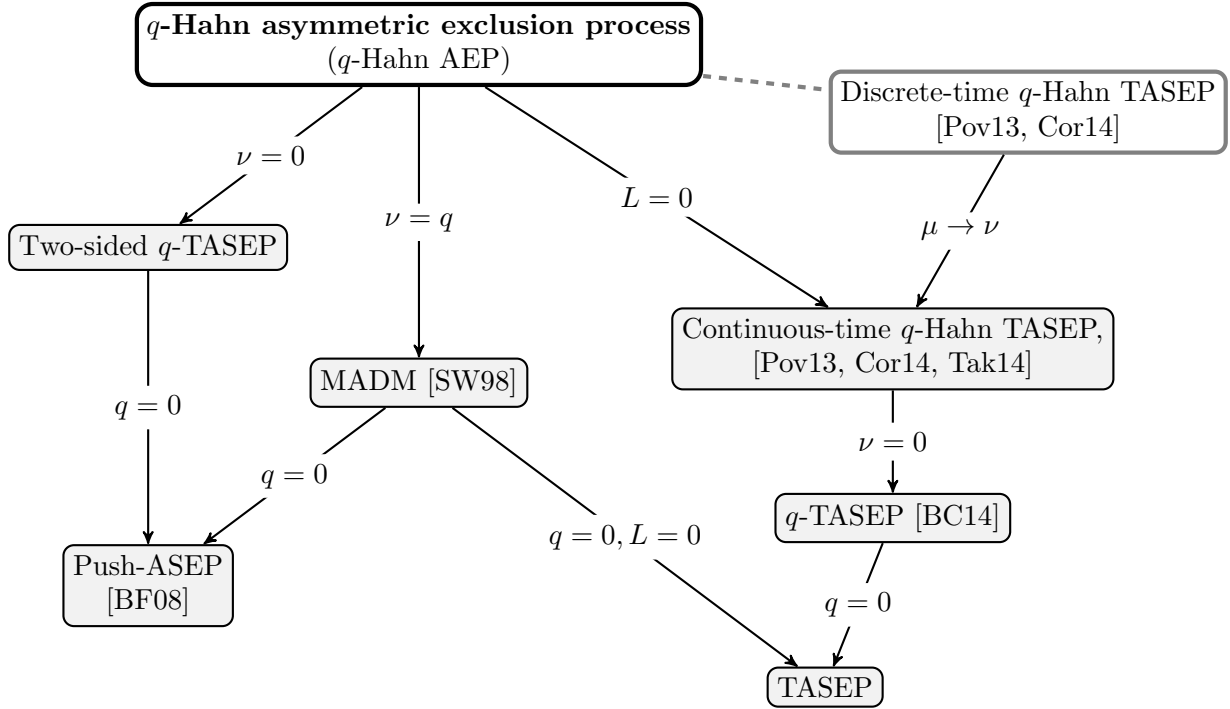
\begin{figure}
\begin{tikzpicture}[scale=0.9, every text node part/.style={align=center}]
\usetikzlibrary{arrows}
\usetikzlibrary{shapes}
\usetikzlibrary{shapes.multipart}
\tikzstyle{process}=[rectangle ,draw,thick ,rounded corners=4pt,fill=black!5 ]
\tikzstyle{AEP}=[rectangle,draw,ultra thick, rounded corners=6pt ]
\tikzstyle{qHahn}=[rectangle,draw, ultra thick, rounded corners=4pt, gray, text=black]
\tikzstyle{limit}=[->,>=stealth',thick,rounded corners=4pt]
\tikzstyle{relation}=[ ultra thick, dashed, rounded corners=4pt, gray]

\node[AEP] (qHahnAEP) at (-1,1) {\textbf{$q$-Hahn asymmetric exclusion process}\\ ($q$-Hahn AEP)};
\node[process]  (cqHahn) at (6,-3.5) {Continuous-time $q$-Hahn TASEP, \\
\cite{povolotsky2013integrability, corwin2014q, takeyama2014deformation}};
\node[process]  (twosidedqTASEP) at (-5,-2) {Two-sided  $q$-TASEP};
\node[process]  (qTASEP) at (6,-6) {$q$-TASEP \cite{borodin2014macdonald}};
\node[process]  (TASEP) at (5,-8.5) {TASEP};
\node[process]  (MADM) at (-1,-4) {MADM  \cite{sasamoto1998one}};
\node[qHahn]  (qHahn) at (8,0) {Discrete-time $q$-Hahn TASEP\\ \cite{povolotsky2013integrability, corwin2014q}};
\node[process]  (PushASEP) at (-5,-7) {Push-ASEP \\ \cite{borodin2008large}} ;

\draw[limit] (qHahnAEP) -- (cqHahn) node[midway,fill=white]{$L=0$};
\draw[limit] (qHahnAEP) -- (twosidedqTASEP) node[midway,fill=white]{$\nu=0$};
\draw[limit] (qHahnAEP) -- (MADM) node[midway,fill=white]{$\nu=q$};
\draw[limit] (twosidedqTASEP) -- (PushASEP) node[midway,fill=white]{$q=0$};
\draw[relation] (qHahnAEP) -- (qHahn);
\draw[limit] (qHahn) -- (cqHahn) node[midway,fill=white]{$\mu\to\nu $};
\draw[limit] (cqHahn) -- (qTASEP) node[midway,fill=white]{$\nu=0$};
\draw[limit] (qTASEP) -- (TASEP) node[midway,fill=white]{$q=0$};
\draw[limit] (MADM) -- (PushASEP) node[midway,fill=white]{$q=0$};
\draw[limit] (MADM) -- (TASEP) node[midway,fill=white]{$q=0, L=0$};
\end{tikzpicture}
\caption{The various degenerations and limits of the $q$-Hahn AEP. All systems except the discrete-time $q$-Hahn TASEP are in continuous time.}
\label{fig:chart}
\end{figure}
%

For $q\in (0,1)$,  $\nu\in [0,1)$ and asymmetry parameters $R, L\geqslant 0$, assuming without loss of generality that $R+L=1$,
we define the $q$-Hahn AEP as a continuous-time Markov process on configurations of particles
$$ +\infty=x_0(t) >x_1(t) >x_2(t) > \dots >x_n(t) > \dots \ ;\ x_i\in \Z.$$
The n$^{th}$ particle,  located at $x_n(t)$, jumps right to the location $x_n(t)+j$ at rate (i.e. according to independent exponentially distributed waiting times with rate) $\phi^R_{q, \nu}(j\vert x_{n-1}(t) - x_n(t) - 1)$ for all $j\in \lbrace 1, \dots, x_{n-1}(t) - x_n(t) - 1\rbrace$, and jumps left to the location $x_n(t)-j'$ at rate $\phi^L_{q, \nu}(j'\vert x_{n}(t) - x_{n+1}(t) - 1)$ for all $j'\in \lbrace 1, \dots, x_{n}(t) - x_{n+1}(t) - 1\rbrace$. Figure \ref{fig:intro} shows two possible jumps for $x_n(t)$.
\begin{figure}[t]
\begin{center}
\begin{tikzpicture}[scale=0.8]
\draw[thick, ->] (-5,0) -- (7,0);
\foreach \k in {-4, ..., 6}
{\draw (\k, -0.1) -- (\k,0.1);}
\fill (-3,0) circle(0.2);
\fill (-2,0) circle(0.2);
\fill (-1,0) circle(0.2);
\fill (2,0) circle(0.2);
\draw[dotted] (4,0) circle(0.2);
\fill[black!25, opacity=0.5] (4,0) circle(0.2);
\draw[dotted] (1,0) circle(0.2);
\fill[black!25, opacity=0.5] (1,0) circle(0.2);
\fill (6,0) circle(0.2);
\draw[thick, ->]  (-1.95,0.3) to[bend left]  (-1,0.3) ;
\draw[thick, red] (-1.7,0.3)  -- (-1.3,0.6);
\draw[thick, red] (-1.7,0.6)  -- (-1.3,0.3);
\draw[thick, ->] node{} (-2.05,0.3) to[bend right] node{} (-3,0.3);
\draw[thick, red] (-2.7,0.3)  -- (-2.3,0.6);
\draw[thick, red] (-2.7,0.6)  -- (-2.3,0.3);
\draw[thick, ->] node{} (2,0.3) to[bend left] node[above right]{\footnotesize{$\phi^R_{q, \nu}(2\vert 3)$}} (4,0.3);
\draw[thick, ->]  (1.9,0.3) to[bend right] node[above left]{\footnotesize{$\phi^L_{q, \nu}(1\vert 2)$}} (1,0.3) ;
\draw (2,-0.5) node {\footnotesize{$ x_n(t)$}};
\draw (6.2,-0.5) node {\footnotesize{$ x_{n-1}(t)$}};
\draw (-0.8,-0.5) node {\footnotesize{$ x_{n+1}(t)$}};
\end{tikzpicture}
\end{center}
\caption{Two admissible jumps for the n$^{th}$ particle in the $q$-Hahn asymmetric exclusion process. }
\label{fig:intro}
\end{figure}
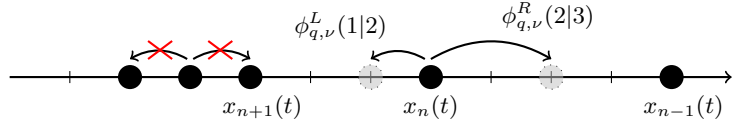
The rates $\phi^R_{q, \nu}(j\vert m)$ and $\phi^L_{q, \nu}(j\vert m)$, defined for all integers $1\leqslant j \leqslant m$, are not arbitrary. To ensure the exact solvability of the process, we fix
\begin{align*}
\phi^R_{q, \nu}(j\vert m) &:= R\  \frac{\nu^{j-1}}{[j]_q}  \frac{(\nu; q)_{m-j}}{(\nu; q)_{m}}\frac{(q; q)_{m}}{(q; q)_{m-j}}, \\ \phi^L_{q, \nu}(j\vert m) &:=L\  \frac{1}{[j]_q}\frac{(\nu; q)_{m-j}}{(\nu; q)_{m}}\frac{(q; q)_{m}}{(q; q)_{m-j}}.
\end{align*}
The $q$-Pochhammer symbol $(a ; q)_n$ is defined in Section \ref{sec:preliminaries}. The superscript $R$ (resp. $L$) on $\phi_{q, \nu}^R$ (resp. $\phi_{q, \nu}^L$) is not an exponent, but rather labels right (resp. left) jump rates.
The reader is referred to Section \ref{sec:continoustime} for a further discussion on the definition of the $q$-Hahn AEP.

The exact solvability of the $q$-Hahn AEP follows along the lines of the method developed in \cite{borodin2012duality} to study $q$-TASEP and ASEP. This method was later used in \cite{corwin2014q} to solve the discrete-time $q$-Hahn TASEP. The key Markov duality relation we use in step (1) below generalizes (though in continuous time) that of \cite{corwin2014q}. The steps in our analysis are as follows:
\begin{enumerate}
\item Via an exclusion/zero-range transformation applied to the $q$-Hahn AEP, we introduce (see Section \ref{sec:continoustime}) the \emph{$q$-Hahn asymmetric zero-range process} ($q$-Hahn AZRP) on $\mathbb{Z}$ with a finite number of particles. Owing to a particular symmetry of the $q$-Hahn distribution, we prove a Markov duality between the $q$-Hahn AEP and the $q$-Hahn AZRP (Proposition \ref{prop:duality}). This implies that $\EE\big[  \prod_{i=1}^k q^{x_{n_i}(t)+n_i} \big]$ solves the Kolmogorov backward equations for the $q$-Hahn AZRP with $k$ particles.
\item The generator of the $q$-Hahn AZRP is diagonalisable via Bethe ansatz, extending results from \cite{povolotsky2013integrability, corwin2014q} to the partially asymmetric case. Indeed, the discrete-time $q$-Hahn TAZRP was introduced by Povolotsky in \cite{povolotsky2013integrability} as the most general parallel update discrete time totally asymmetric `chipping' model on a ring lattice with factorized invariant measures which is solvable via Bethe ansatz. Combined with duality, Bethe ansatz yields  exact integral formulas for all moments of the random variable $q^{x_n(t)}$ (Proposition \ref{prop:nestedcontours}).
\item Using techniques introduced in the context of Macdonald processes \cite{borodin2014macdonald}, we use the moment formulas for $q^{x_n(t)}$ to compute a formula for the $e_q$-Laplace transform of $q^{x_n(t)}$ as a Fredholm determinant (Theorem \ref{th:fredholmgeneral}). This characterizes the law of $x_n(t)$.
\item We provide a rigorous asymptotic analysis of this Fredholm determinant in the case $q=\nu$ (i.e. the MADM case).
This is stated as Theorem \ref{thm:fluctuationsrarefactionfan} and Theorem \ref{thm:fluctuationsparticle1} and proves  Theorem \ref{thm:fluctuationsintro2} and Theorem \ref{thm:fluctuationsintro1} (see Section \ref{sec:MADM}). The asymptotic analysis here is an instance of the Laplace (or saddle-point) method which has been implemented in similar contexts in \cite{borodin2012free, ferrari2013tracy, barraquand2014phase, veto2014tracy}.
\end{enumerate}

Though the general gameplan for solving $q$-Hahn AEP is similar to that used in earlier works, there are certain technical novelties that arise in the present paper which we highlight.

\begin{itemize}
\item Previous works have been for totally asymmetric processes with right-finite initial data (such as the step initial data). In that case the position of the $n$th particle only depends on positions of the first $n-1$ particles. This is no longer true for partially asymmetric processes. This has two consequences: the processes are not obviously well-defined, and
 unlike in \cite{borodin2012duality, corwin2014q} the Markov duality functional defined in \eqref{eq:Hdual} is an infinite product involving infinitely many particle locations.
\item The proof of Proposition \ref{prop:uniqueness} is more involved than in previous papers, and more complete than \cite[Appendix C]{borodin2012duality}. In the totally asymmetric cases, the systems of ODEs considered were triangular, which implies uniqueness straightforwardly.
\item We use two different series representations of the $q$-digamma function (see Lemma \ref{lem:equivalentseries}), in order to connect the formulas arising from KPZ scaling theory with those coming from the saddle-point analysis of  Fredholm determinants in Section \ref{sec:KPZscaling}.
\item In the asymptotic analysis, we use an interpolation between cases for which formulas are manageable (cases $L=0 = 1-R$ and $R=qL$), in order to cover the general $R,L$ case.
\end{itemize}

\subsection*{Acknowledgements}
G.B. is grateful to Sandrine P\'ech\'e and B\'alint Vet\H o for stimulating discussions. G.B. and I.C. are grateful to Sidney Redner for discussions regarding \cite{gabel2010facilitated}, as well as helpful comments from the editor of AAP regarding this text.

G.B. acknowledges support from the Laboratoire de Probabilit\'es et Mod\`eles Al\'eatoires (LPMA) in Universit\'e Paris-Diderot.  I.C. was partially supported by the NSF through DMS-1208998 as well as by Microsoft Research and MIT through the Schramm Memorial Fellowship, by the Clay Mathematics Institute through the Clay Research Fellowship, by the Institute Henri Poincar\'e through the Poincar\'e Chair, and by the Packard Foundation through the Packard Fellowship for Science and Engineering.

\subsection*{Outline of the paper} In Section \ref{sec:preliminaries}, we provide definitions and establish useful identities for some $q$-deformed special functions.
In Section \ref{sec:continoustime}, we introduce the $q$-Hahn AEP
and establish the Fredholm determinant identity characterizing the distribution of particles positions. In Section \ref{sec:KPZscaling}, we study this process from the point of view of the conjectural KPZ scaling theory, and we state the predicted limit theorems. We sketch an asymptotic analysis of the Fredholm determinant, leading to the predicted Tracy-Widom limit theorem. In Section \ref{sec:MADM}, we make a rigorous asymptotic analysis in the case $\nu=q$, which corresponds to the MADM, thus proving Theorems \ref{thm:fluctuationsintro1} and \ref{thm:fluctuationsintro2}.

\section{Preliminaries on the q-deformed gamma and digamma functions}
\label{sec:preliminaries}

Fix hence forth that $q\in (0,1)$. For $a\in\mathbb{C}$ and $n\in \mathbb{Z}_{\geqslant 0}$,
define the $q$-Pochhammer symbol
$$ (a; q)_n = \prod_{i=0}^{n-1} (1-aq^i) \qquad \text{and }\qquad (a; q)_{\infty} = \prod_{i=0}^{\infty} (1-aq^i).$$
For an integer $n$, the $q$-integer $[n]_q$ is
$$[n]_q = 1+q+ \dots + q^{n-1} = \frac{1-q^n}{1-q}.$$
The $q$-factorial is defined as
\begin{equation}
[n]_q! = [n]_q[n-1]_q \dots [1]_q = \frac{(q ; q)_n}{(1-q)^n}. \label{eq:defqfactorial}
\end{equation}
The $q$-binomial coefficients are
$$ \left[ \begin{matrix}n\\k\end{matrix}\right]_q = \frac{[n]_q!}{[n-k]_q! [k]_q!}= \frac{(q ; q)_n}{(q ; q)_k(q ; q)_{n-k}}.$$
 For $\vert z\vert <1$, the $q$-binomial theorem \cite[Theorem 10.2.1]{andrews1999special} implies that
\begin{equation}
\sum_{k=0}^{\infty}\frac{(a ; q)_k}{(q ; q)_k}z^k = \frac{(az ; q)_{\infty}}{(z ; q)_{\infty}}.
\label{eq:qbinomial}
\end{equation}

 The $q$-gamma function is defined by
 $$ \Gamma_q(z) = (1-q)^{1-z} \frac{(q;q)_{\infty}}{(q^z;q)_{\infty}},$$
and the $q$-digamma function is defined by
 $$\Psi_q(z)  = \frac{\partial}{\partial z} \log \Gamma_q(z).$$
From the definition of the  $q$-digamma function, we have a series representation for $\Psi_q$,
\begin{equation}
\Psi_q(z) = \frac{\mathrm{d}}{\mathrm{d}z} \log \Gamma_q(z) = -\log(1-q) + \log(q) \sum_{k=0}^{\infty} \frac{q^{k+z}}{1-q^{k+z}}.
\label{eq:seriespsiq}
\end{equation}
 Let us also define a closely-related series that will appear in Section \ref{sec:KPZscaling},
\begin{equation*}
G_q(z):= \sum_{i=1}^{\infty} \frac{z^i}{[i]_{q}}.
\end{equation*}

\begin{lemma}\label{lemmaGseries}
For  $z\in \C$ with positive real part,
\begin{equation}
G_{q}(q^z) = \frac{1-q}{\log q} \big(\Psi_q(z) + \log(1-q) \big). \label{eq:lambertseries1}
\end{equation}
For $z\in \C$ with real part greater than $-1$,
\begin{equation}
G_{q^{-1}}(q^z) = \frac{q^{-1}-1}{\log q} \big(\Psi_q(z+1) + \log(1-q) \big). \label{eq:lambertseries2}
\end{equation}
\label{lem:equivalentseries}
\end{lemma}
\begin{proof}
Assume $z\in \C$ with positive real part. Using the series representation (\ref{eq:seriespsiq}), we have that
$$\frac{1-q}{\log q} \big(\Psi_q(z) + \log(1-q) \big) = (1-q)\sum_{k=0}^{\infty} \frac{q^{k+z}}{1-q^{k+z}}.$$
Since $z$ has positive real part, we can write for all $k\geqslant 0$
$$\frac{q^{k+z}}{1-q^{k+z}} = \sum_{i= 1}^{\infty} q^{(k+z)i},$$
so that the right-hand-side in (\ref{eq:lambertseries1}) equals
$$ (1-q)\sum_{k=0}^{\infty} \sum_{i=1}^{\infty} q^{(k+z)i}.$$
Exchanging the summations yields
 $$\frac{1-q}{\log q} \big(\Psi_q(z) + \log(1-q) \big) =(1-q)\sum_{i=1}^{\infty} \sum_{k= 0}^{\infty}q^{(k+z)i} =  \sum_{i=1}^{\infty} \frac{(q^z)^i}{[i]_q}.$$
Equation (\ref{eq:lambertseries2}) can be deduced from (\ref{eq:lambertseries1}) by replacing $z$ by $z+1$.
\end{proof}

A consequence of Lemma \ref{lemmaGseries} is the following formula for the $k$-fold derivatives of the $q$-digamma function:
\begin{equation}
\Psi_q^{(k)}(z) = (\log q)^{k+1} \sum_{n=1}^{\infty} \frac{n^k q^{nz}}{1-q^n}.
\label{eq:digammaderivatives}
\end{equation}

\section{The $q$-Hahn AEP and AZRP}
\label{sec:continoustime}

Let us first recall the definition of the discrete-time $q$-Hahn-TASEP \cite{povolotsky2013integrability,corwin2014q}. Fix $q\in (0,1)$ and $0\leqslant \nu <\mu<1$. The $N$-particle $q$-Hahn TASEP is a discrete time Markov chain $\vec{x}(t) = \lbrace x_n(t) \rbrace_{n=0}^{N} \in \mathbb{X}^N$ with state space
$$\mathbb{X}^N  = \lbrace +\infty = x_0 > x_1 >\dots >x_N \ ;\  \forall n\geq 1, x_n\in \mathbb{Z} \rbrace.$$
At time $t+1$, each coordinate $x_n(t)$ is updated independently and in parallel to $x_n(t+1) = x_n(t)+j_n$ where $0\leqslant j_n \leqslant x_{n-1}(t)-x_n(t)-1$ is drawn according to the $q$-Hahn probability distribution. The $q$-Hahn probability distribution on $j\in \lbrace 0, 1, \dots ,m\rbrace$ is defined by
\begin{equation}
 \varphi_{q, \mu, \nu}(j\vert m) = \mu^j \frac{(\nu/\mu ; q)_{j}(\mu ; q)_{m-j}}{(\nu;q)_{m}} \left[\begin{matrix}m\\j\end{matrix}\right]_q.
 \label{eq:defqhahndistribution}
\end{equation}
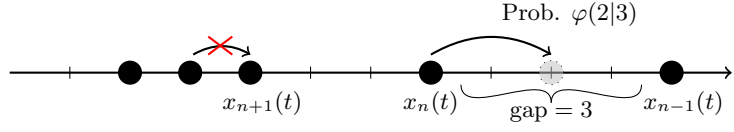
\begin{figure}
\begin{center}
\begin{tikzpicture}[scale=0.8]
\draw[thick, ->] (-5,0) -- (7,0);
\foreach \k in {-4, ..., 6}
{\draw (\k, -0.1) -- (\k,0.1);}
\fill (-3,0) circle(0.2);
\fill (-2,0) circle(0.2);
\fill (-1,0) circle(0.2);
\fill (2,0) circle(0.2);
\draw[dotted] (4,0) circle(0.2);
\fill[black!25, opacity=0.5] (4,0) circle(0.2);
\fill (6,0) circle(0.2);
\draw[thick, ->]  (-1.95,0.3) to[bend left]  (-1,0.3) ;
\draw[thick, red] (-1.7,0.3)  -- (-1.3,0.6);
\draw[thick, red] (-1.7,0.6)  -- (-1.3,0.3);

\draw[thick, ->] node{} (2,0.3) to[bend left] node[above right]{\footnotesize{Prob. $\varphi(2\vert 3) $}} (4,0.3);

\draw (2.5, -0.2) to[in=90, out=-90] (4,-0.5);
\draw (5.5, -0.2) to[in=90, out=-90] (4,-0.5);
\draw (4, -0.65) node{\footnotesize{$\rm{gap}=3$}};

\draw (2,-0.5) node {\footnotesize{$ x_n(t)$}};
\draw (6.2,-0.5) node {\footnotesize{$ x_{n-1}(t)$}};
\draw (-0.8,-0.5) node {\footnotesize{$ x_{n+1}(t)$}};
\end{tikzpicture}
\end{center}
\caption{Jumps probabilities in the (discrete-time) $q$-Hahn TASEP.}
\end{figure}
As we have explained in the Introduction, the exact solvability of this process is granted by a Markov duality with a zero-range process (the $q$-Hahn Boson process) and the Bethe ansatz solvability of the latter.

In this section, we introduce a generalization of a continuous time limit of the $q$-Hahn TASEP allowing jumps towards both directions. The key to this generalization is that the Markov duality is preserved under it.
Proposition 1.2 in \cite{corwin2014q}, shows that certain `$q$-moments' of the $q$-Hahn probability distribution enjoy a symmetry relation, which is ultimately responsible for an intertwining (and hence Markov duality) of the Markov generators of the $q$-Hahn Boson model and the $q$-Hahn TASEP. This relation is that for all positive integers $m$ and $y$,
\begin{equation}
\sum_{j=0}^m \varphi_{q, \mu, \nu}(j\vert m ) q^{jy} = \sum_{j=0}^y \varphi_{q, \mu, \nu}(j\vert y ) q^{jm}.
\label{eq:symmetry}
\end{equation}
The same identity replacing all variables by their inverse also holds:
\begin{equation}
\sum_{j=0}^m \varphi_{q^{-1}, \mu^{-1}, \nu^{-1}}(j\vert m ) q^{-jy} = \sum_{j=0}^y \varphi_{q^{-1}, \mu^{-1}, \nu^{-1}}(j\vert y ) q^{-jm}.
\label{eq:symmetryinverse}
\end{equation}
For $q, \mu, \nu$ as specified earlier, the weights $\varphi_{q, \mu, \nu}(j\vert m )$ and $\varphi_{q^{-1}, \mu^{-1}, \nu^{-1}}(j\vert m )$ are positive, and hence define probability distributions on $j\in\lbrace 0, 1, \ldots, m \rbrace$. Notice also that
\begin{equation*}
\varphi_{q^{-1}, \mu^{-1}, \nu^{-1}}(j\vert m ) = \left(\frac{\nu}{\mu} \right)^m \frac{1}{\nu^j} \varphi_{q, \mu, \nu}(j\vert m ).
\end{equation*}
One can extend the $q$-Hahn weights by continuity when $\nu$ goes to zero, so that
\begin{equation}
\varphi_{q, \mu, 0}(j\vert m ) = \mu^j (\mu ; q)_{m-j}\left[\begin{matrix}m\\j\end{matrix}\right]_q\  \text{ and } \ \varphi_{q^{-1}, \mu^{-1}, \infty}(j\vert m ) = \mathds{1}_{\lbrace j=m\rbrace}.
\label{eq:qhahnweightsnuzero}
\end{equation}

These observations motivate the introduction of a two-sided $q$-Hahn process where jumps to the left are distributed according to a $q$-Hahn distribution with parameters $q^{-1}, \mu^{-1}, \nu^{-1}$, and those to the right with parameters $q, \mu, \nu$ as before. We will define this sort of two-sided process, but only in continuous time to simplify possible obstacles that arise in discrete time. Let us first observe how the right and left jump distribution turns into continuous time rates for exponentially distributed jump waiting times.   Fix $q,\nu\in (0,1)$ and set $\mu = \nu+ (1-q)\epsilon$. Then for all $j\geqslant 1$, the jump probabilities of the $q$-Hahn distribution become jump rates given by the limits,
\begin{eqnarray}
\varphi_{q, \mu, \nu}(j\vert m )/\epsilon & \underset{\epsilon\to 0}{ \longrightarrow} & \frac{\nu^{j-1}}{[j]_q} \frac{(\nu; q)_{m-j}}{(\nu; q)_{m}}\frac{(q; q)_{m}}{(q; q)_{m-j}},\label{eq:limitprobaright}
\\
\varphi_{q^{-1}, \mu^{-1}, \nu^{-1}}(j\vert m )/\epsilon &\underset{\epsilon\to 0}{ \longrightarrow} &   \frac{\nu^{-1}}{[j]_q} \frac{(\nu; q)_{m-j}}{(\nu; q)_{m}}\frac{(q; q)_{m}}{(q; q)_{m-j}}.\label{eq:limitprobaleft}
\end{eqnarray}

Let us fix some notation and write these limiting rates as $\phi^R_{q, \nu}$ and $\phi^L_{q, \nu}$:
\begin{align*}
\phi^R_{q, \nu}(j\vert m) := R \frac{\nu^{j-1}}{[j]_q}  \frac{(\nu; q)_{m-j}}{(\nu; q)_{m}}\frac{(q; q)_{m}}{(q; q)_{m-j}}, \\ \phi^L_{q, \nu}(j\vert m) :=L \frac{1}{[j]_q}\frac{(\nu; q)_{m-j}}{(\nu; q)_{m}}\frac{(q; q)_{m}}{(q; q)_{m-j}}.
\end{align*}
The letters $R$ and $L$ stand for ``right'' and ``left'' as well as denote the values of the relative rates of jumps of particles in the process in those respective directions. Note that we deliberately removed the factor $\nu^{-1}$ (present in the $\e\to 0$ limit) from $\phi^L_{q, \nu}(j\vert m)$ to be consistent with models previously introduced in the particle system literature (see Section \ref{subsec:degenerations}). In this way, the rates are well-defined for $\nu=0$ and all results of this section hold for  $\nu=0$ as well. It is useful for later calculations to notice that
\begin{equation}
 R^{-1}\phi^R_{q^{-1}, \nu^{-1}}(j\vert m) = \frac{\nu}{q} L^{-1}\phi^L_{q, \nu}(j\vert m).
 \label{eq:inversionrates}
\end{equation}

\begin{definition}
We define the (continuous time) $q$-Hahn asymmetric zero-range process (abbreviated $q$-Hahn AZRP) as a Markov process $\vec{y}(t)\in \mathbb{Y}^{\infty}$ with state-space
$$\mathbb{Y}^{\infty} = \left\lbrace (y_0, y_1, \dots ) \ ; \forall i\in \Z_{\geqslant 0}, \ y_i \in \Z_{\geqslant 0}\text{ and } \sum_{i=0}^{\infty} y_i < \infty \right\rbrace.$$
and infinitesimal generator $B_{q, \nu}$ defined in (\ref{eqnctnsgen}). Before stating this generator, we must introduce some notation. For a vector $\vec{y} = (y_0, y_1, \dots)$, and any $j\leqslant y_i$ we denote
\begin{align*}
 \vec{y}^{j}_{i, i-1} = (y_0, \dots, y_{i-1}+j, y_i-j, y_{i+1}, \dots),\\
  \vec{y}^{j}_{i, i+1} = (y_0, \dots, y_{i-1}, y_i-j, y_{i+1}+j, \dots).
 \end{align*}
The operator $B_{q, \nu}$ is defined by its action on functions $\mathbb{Y}^{\infty}\rightarrow \R$ by
\begin{equation}\label{eqnctnsgen}
\big(B_{q, \nu} f\big)(\vec{y}) = \sum_{i=1}^{\infty} \left( \sum_{j=1}^{y_i} \phi^R_{q, \nu}(j\vert y_i) \left( f(\vec{y}^{j}_{i, i-1})- f(\vec{y}) \right) +
\sum_{j=1}^{y_i}  \phi^L_{q, \nu}(j\vert y_i) \left( f(\vec{y}^{j}_{i, i+1})- f(\vec{y}) \right)\right).
\end{equation}
\begin{figure}
\begin{center}
\begin{tikzpicture}[scale=0.5]
\draw  (0,0.5) -- (0,0) -- (8,0);
\draw[dashed] (8,0) -- (10,0);
\foreach \x in {1,2,...,8}
	\draw (\x,0) -- (\x,0.5);
\foreach \x in {0,1,...,6}
	\draw (\x+0.5,-0.4) node{\footnotesize{$y_{\x} $}};
\draw (8,-0.4) node{$\dots$};
\fill (1.5,0.5) circle(0.25);
\fill (2.5,0.5) circle(0.25);
\fill (3.5,0.5) circle(0.25);
\fill (4.5,0.5) circle(0.25);
\fill (5.5,0.5) circle(0.25);
\fill (6.5,0.5) circle(0.25);
\fill (2.5,1.5) circle(0.25);
\fill (4.5,1.5) circle(0.25);
\fill (6.5,1.5) circle(0.25);
\fill (4.5,2.5) circle(0.25);
\fill (4.5,3.5) circle(0.25);
\draw (4.3, 3.8) -- (4.1,3.8) -- (4.1,2.2) -- (4.3,2.2);
\draw[->, thick] (4.1,3) node{} to[bend right] (3.5,2)  node[above left]{\footnotesize{rate  $\phi^R(2\vert y_4)$}};
\draw (4.7, 3.8) -- (4.9,3.8) -- (4.9,1.2) -- (4.7,1.2);
\draw[->, thick] (4.9,2.5) node{} to[bend left] (5.5, 1.5)  node[above right]{\footnotesize{rate  $\phi^L(3\vert y_4)$}};
\end{tikzpicture}
\end{center}
\caption{Rates of two possible transitions in the $q$-Hahn asymmetric zero-range process.}
\label{fig:defBoson}
\end{figure}
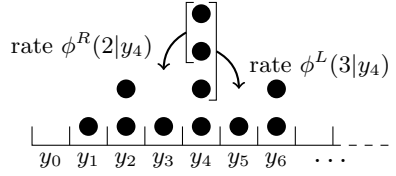
Informally, if the site $i$ is occupied by $y$ particles, $j$ particles move together to site $i-1$ with rate $\phi^R_{q, \nu}(j\vert y)$ whereas $j'$ particles move together to site $i+1$ with rate $\phi^L_{q, \nu}(j'\vert y)$, for all $1\leqslant j,j'\leqslant y$ (see Figure \ref{fig:defBoson}).

Similarly, we define the $q$-Hahn AEP as a Markov process $\vec{x}(t)\in \mathbb{X}^{\infty}$ where
the state space $\mathbb{X}^{\infty}$ is defined by
\begin{equation*}
\mathbb{X}^{\infty}= \left\lbrace  +\infty =  x_0 > x_1 > \dots >x_n>\dots \  \left|\  \begin{matrix*}[l]
 \forall n\geq 1,\  x_n\in \mathbb{Z} \\
 \exists N >0 ,  \forall n \geqslant N, x_n-x_{n+1}=1
\end{matrix*}\right. \right\rbrace .
\end{equation*}
In words, $\mathbb{X}^{\infty}$ is the space of particle configurations that have a right-most particle and a left-most empty site. This is the analogue of the state space $\mathbb{Y}^{\infty}$ by exclusion/zero-range transformation, that is if one maps the gaps between consecutive particles in the exclusion process to the number of particles on the sites of the zero-range process.

The $q$-Hahn AEP is defined by  the action of its infinitesimal generator $T_{q, \nu}$. Let us introduce some notations.
For a vector $\vec{x} = (x_0, x_1, \dots)$ we denote for any $j\in \mathbb{Z}$ and $i\geqslant 1$
$$ \vec{x}^{j}_i = (x_0, \dots, x_{i-1}, x_i+j, x_{i+1}, \dots).$$
The operator $T_{q, \nu}$ acts on  functions $\mathbb{X}^{\infty}\rightarrow \R$ by
\begin{multline} \label{eq:genAEP}
\big(T_{q, \nu} f\big)(\vec{x}) = \sum_{i=1}^{\infty} \left( \sum_{j=1}^{x_{i-1}-x_i-1}  \phi^R_{q, \nu}(j\vert x_{i-1}-x_i-1) \left( f(\vec{x}^{+j}_{i})- f(\vec{x}) \right) \right. +\\
\left. \sum_{j=1}^{x_{i}-x_{i+1}-1}  \phi^L_{q, \nu}(j\vert x_{i}-x_{i+1}-1) \left( f(\vec{x}^{-j}_{i})- f(\vec{x}) \right)\right).
\end{multline}
\end{definition}

\begin{remark}
It may be possible to define the $q$-Hahn AZRP (resp. $q$-Hahn AEP) on a larger state space including configurations with an infinite number of particles (resp. an infinite number of positive gaps between consecutive particles). Such a more general definition would add some complexity in several of the later statements. In the following, we study the zero-range processes only with a finite number of particles and the exclusion process starting only from the step-initial condition ($\forall n>0, x_n(0)=-n$), thus we prefer to restrict our definition to the state-spaces $\mathbb{X}^{\infty}$ and $\mathbb{Y}^{\infty}$.
\end{remark}

Before going further into the analysis of the $q$-Hahn AEP and AZRP, we must justify that  both processes are well defined.
\paragraph{\textbf{Existence of the $q$-Hahn AZRP}} Observe that the (finite) number of particles is conserved by the dynamics. Let $k$ be the number of particles in the initial condition. Then, each entry of the transition matrix of the process is bounded by
$$ k \cdot \max_{m\in\lbrace 1, \dots, k\rbrace} \sum_{j\leqslant m} \left(\phi^R_{q, \nu}(j\vert m)+ \phi^L_{q, \nu}(j\vert m) \right) <\infty.$$
The existence of a Markov process with the generator  (\ref{eqnctnsgen}) follows from the classical construction of Markov chains on a denumerable state space with bounded generator (see e.g. \cite[Chap. 4 Section 2]{ethier2009markov}).

\paragraph{\textbf{Existence of $q$-Hahn AEP}}
Although it should be possible to show that the generator (\ref{eq:genAEP}) defines uniquely a Markov semi-group (using e.g. \cite[Proposition 4.3]{borodin2012markov}), we prefer to give a probabilistic construction of the $q$-Hahn AEP that corresponds to the generator. Fix some $T>0$ and let us show that the processes is  well-defined on the time interval $[0,T]$. The construction will then extend to any time $t\in \R_+$ by the Markov property. We prove that the construction on $[0,T]$ is actually that of a continuous-time Markov chain on a finite (random) state space.
 Consider a (possibly random) initial condition in  $\mathbb{X}^{\infty}$. By the definition of the state space $\mathbb{X}^{\infty}$, there exists a (possibly random) integer $N$ such that for all $n>N$, $x_n(0)-x_{n+1}(0)=1$. Almost surely, there exists an integer $n>N$ such that  the particle labelled by $n$ does not move on the time interval $[0,T]$. Indeed, if this particle moves, then it has to move at least once to its right, since there is no room to its left. The rates at which a jump on the right occurs is bounded by
$$ M:= \sup_{m\geqslant 1} \sum_{j=1}^m\phi^R_{q, \nu}(j\vert m) <\infty.$$
Since all particles are equipped with independent Poisson clocks, there exists almost surely a particle that does not jump to the right. Finally, the $q$-Hahn AEP can be constructed on $[0,T]$ as a Markov chain on a finite state-space.

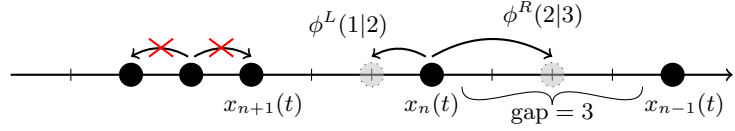
\begin{figure}
\begin{center}
\begin{tikzpicture}[scale=0.8]
\draw[thick, ->] (-5,0) -- (7,0);
\foreach \k in {-4, ..., 6}
{\draw (\k, -0.1) -- (\k,0.1);}
\fill (-3,0) circle(0.2);
\fill (-2,0) circle(0.2);
\fill (-1,0) circle(0.2);
\fill (2,0) circle(0.2);
\draw[dotted] (4,0) circle(0.2);
\fill[black!25, opacity=0.5] (4,0) circle(0.2);
\draw[dotted] (1,0) circle(0.2);
\fill[black!25, opacity=0.5] (1,0) circle(0.2);
\fill (6,0) circle(0.2);
\draw[thick, ->]  (-1.95,0.3) to[bend left]  (-1,0.3) ;
\draw[thick, red] (-1.7,0.3)  -- (-1.3,0.6);
\draw[thick, red] (-1.7,0.6)  -- (-1.3,0.3);
\draw[thick, ->] node{} (-2.05,0.3) to[bend right] node{} (-3,0.3);
\draw[thick, red] (-2.7,0.3)  -- (-2.3,0.6);
\draw[thick, red] (-2.7,0.6)  -- (-2.3,0.3);

\draw[thick, ->] node{} (2,0.3) to[bend left] node[above right]{\footnotesize{$\phi^R(2\vert 3) $}} (4,0.3);
\draw[thick, ->]  (1.9,0.3) to[bend right] node[above left]{\footnotesize{$\phi^L(1\vert 2) $}} (1,0.3) ;

\draw (2.5, -0.2) to[in=90, out=-90] (4,-0.5);
\draw (5.5, -0.2) to[in=90, out=-90] (4,-0.5);
\draw (4, -0.65) node{\footnotesize{$\rm{gap}=3$}};

\draw (2,-0.5) node {\footnotesize{$ x_n(t)$}};
\draw (6.2,-0.5) node {\footnotesize{$ x_{n-1}(t)$}};
\draw (-0.8,-0.5) node {\footnotesize{$ x_{n+1}(t)$}};
\end{tikzpicture}
\end{center}
\caption{Rates of two possible jumps in the $q$-Hahn asymmetric exclusion process.}
\end{figure}

\subsection{Markov duality}
We come now to the Markov duality between the $q$-Hahn AEP and the $q$-Hahn AZRP.
\begin{proposition}
Define $H : \mathbb{X}^{ \infty} \times \mathbb{Y}^{\infty} \to \R$ by
\begin{equation}\label{eq:Hdual}
H(\vec{x}, \vec{y}) := \prod_{i=0}^{\infty} q^{(x_i +i)y_i},
\end{equation}
with the convention that the product is $0$ when $y_0>0$.
 For any $(\vec{x}, \vec{y})$ in $\mathbb{X}^{\infty}\times \mathbb{Y}^{\infty}$, we have that
$$ B_{q, \nu}  H(\vec{x}, \vec{y}) = T_{q, \nu}  H(\vec{x}, \vec{y}),$$
where $B_{q, \nu}$ acts on the $\vec{y}$ variable,  $T_{q, \nu}$ acts on the $\vec{x}$ variable.
\label{prop:duality}
\end{proposition}
\begin{proof}
Under the scalings above and when $\epsilon$ goes to zero, identities (\ref{eq:symmetry}) and (\ref{eq:symmetryinverse}) degenerate to
\begin{equation}
\sum_{j=1}^m \phi^R_{q, \nu}(j\vert m) \left(q^{jy} - 1 \right) = \sum_{j=1}^y \phi^R_{q, \nu}(j\vert y) \left(q^{jm} - 1 \right),
\label{eq:symmetrydegenerate}
\end{equation}
and
 \begin{equation}
\sum_{j=1}^m \phi^L_{q, \nu}(j\vert m) \left(q^{-jy} - 1 \right) = \sum_{j=1}^y \phi^L_{q, \nu}(j\vert y) \left(q^{-jm} - 1 \right).
\label{eq:symmetryinversedegenerate}
\end{equation}
Let us explain how (\ref{eq:symmetrydegenerate}) is obtained. From the limit (\ref{eq:limitprobaright}), we know that for $j\geqslant 1$,
$$ \varphi_{q, \mu, \nu}(j\vert m) = \epsilon R^{-1}\phi^R_{q, \nu}(j\vert m) + o(\epsilon).$$
Since $\sum_{j=0}^m \varphi_{q, \mu, \nu}(j\vert m) = 1$, we know that
$$ \varphi_{q, \mu, \nu}(0\vert m) = 1 -  \sum_{j=1}^m \epsilon R^{-1}\phi^R_{q, \nu}(j\vert m) + o(\epsilon).$$
Finally, in terms of $\epsilon$, identity (\ref{eq:symmetry}) implies
\begin{multline*}
 1 -  \sum_{j=1}^m \epsilon R^{-1}\phi^R_{q, \nu}(j\vert m) + \sum_{j=1}^m \epsilon R^{-1}\phi^R_{q, \nu}(j\vert m) q^{jy} + o(\epsilon) = \\ 1 -  \sum_{j=1}^y\epsilon R^{-1}\phi^R_{q, \nu}(j\vert y) + \sum_{j=1}^y \epsilon R^{-1}\phi^R_{q, \nu}(j\vert y) q^{jm} + o(\epsilon).
 \end{multline*}
Substracting $1$ from both sides and identifying terms of order $\epsilon$, one gets identity (\ref{eq:symmetrydegenerate}). Identity (\ref{eq:symmetryinversedegenerate}) is obtained in a similar way.

Applying generators $B_{q, \nu}$ and $T_{q, \nu}$ to the function $H(\vec{x}, \vec{y})=\prod_{i=0}^{\infty} q^{(x_i+i)y_i}$ and using  (\ref{eq:symmetrydegenerate})  and (\ref{eq:symmetryinversedegenerate}) on each term of the sum, one gets that $B_{q, \nu} H = T_{q, \nu} H$. More precisely, we have that

\begin{multline*}
T_{q, \nu} H(\vec{x}, \vec{y}) =
 \prod_{i=1}^{\infty} \left(\sum_{j_i=0}^{x_{i-1}-x_i-1}  \phi^R_{q, \nu}(j_i\vert x_{i-1}-x_i-1 ) \big(q^{j_i y_i}-1\big)\ \  + \right.\\
\left. \sum_{k_i=0}^{x_{i}-x_{i+1}-1}  \phi^L_{q,  \nu}(k_i\vert x_{i}-x_{i+1}-1 ) \big(q^{-k_i y_i} -1\big)\right) \prod_{i=0}^{\infty} q^{(x_i+i)y_i}.
\end{multline*}
Applying (\ref{eq:symmetrydegenerate}) and (\ref{eq:symmetryinversedegenerate}) to the terms inside the parenthesis, we find that
\begin{eqnarray*}
T_{q, \nu} H(\vec{x}, \vec{y}) &=& \prod_{i=1}^{\infty} \left(  \sum_{s_i=0}^{y_i}  \phi^R_{q,  \nu}(s_i\vert y_i ) \big(q^{s_i (x_{i-1}-x_i-1)} - 1\big)  \right.\\
&&\left. \qquad + \sum_{t_i=0}^{y_i}  \phi^L_{q, \nu}(t_i\vert y_i ) \big(q^{-t_i (x_{i}-x_{i+1}-1)}-1 \big)\right) \prod_{i=0}^{\infty} q^{(x_i+i)y_i}\\
&=& B_{q, \nu} H(\vec{x}, \vec{y}).
\end{eqnarray*}
\end{proof}
\begin{remark}
One sees from the proof of Proposition \ref{prop:duality} that  our statement could be generalized to hold when the parameter $\nu$ is not the same for the jumps to the left and the jumps to the right, as well as when the
 parameter $\nu$ and the asymmetry parameters $R$ and $L$ vary over different  sites/particles provided that the parameters corresponding to the $i^{\rm th}$ particle in the exclusion process equal the parameters corresponding to the $i^{\rm th}$ site in the zero-range process.

It is not presently clear if beyond this duality, the integrability via Bethe ansatz of the $q$-Hahn AZRP (resp. $q$-Hahn AEP) process extends to the general time and site-dependent  (resp. particle-dependent) case (see \cite[Section 2.4]{corwin2014q} for a related discussion in the $q$-Hahn TASEP case).
\end{remark}

The $k$-particle $q$-Hahn AZRP process can be alternatively described in terms of ordered particle locations $\vec{n}(t) = \vec{n}(\vec{y}(t))$. The bijection between $\vec{n}$ coordinates and $\vec{y}$ coordinates is such that $n_i(t) = n$ if and only if $\sum_{j>n} y_j < i \leqslant \sum_{j\geqslant n} y_j$ and we impose that $\vec{n}\in \Weyl{k}$ where the Weyl chamber $\Weyl{k}$ is defined as
\begin{equation}\label{eqnWk}
\Weyl{k} = \left\lbrace  n_1 \geqslant n_2\geqslant \dots \geqslant  n_k \ ; \  n_i\in\Z_{\geqslant 0}, 1\leqslant i \leqslant k\right\rbrace.
\end{equation}
For a subset $I\subset\lbrace 1, \dots, k\rbrace $ and $\vec{n}\in \Weyl{k}$, we introduce the vector
$ \vec{n}_I^+  $ obtained from $\vec{n}$ by increasing by one all coordinates with index in $I$ ; and the vector
$ \vec{n}_I^-  $ obtained from $\vec{n}$ by decreasing by one all coordinates with index in $I$. As an example,
$$ \vec{n}_i^+  = (n_1,  \dots, n_{i-1}, n_i+1,n_{i+1}, \dots , n_k ).$$
With a slight abuse of notations, we will use the same symbol $B_{q, \nu}$ for the generator of the $q$-Hahn AZRP described in terms of variables in either $\mathbb{Y}^{\infty}$ or  $\Weyl{k}$.

\begin{definition}
We say that $h:\R_+\times \Weyl{k}$ solves the $k$-particle true evolution equation with initial data $h_0$ if it satisfies the conditions that:
\begin{enumerate}
\item for all $\vec{n}\in\Weyl{k}$ and $t\in\R_+$,
$$ \frac{\mathrm{d}}{\mathrm{d}t} h(t, \vec{n}) = B_{q, \nu} h(t, \vec{n}), $$
\item for all $\vec{n}\in\Weyl{k}$, $h(t,\vec{n}) \xrightarrow[t\to 0]{}  h_0(\vec{n})$,
\item for any $T>0$, there exists constants $c,C >0$ such that for all $\vec{n}\in \Weyl{k}$, $t\in [0,T]$,
$$ \vert h(t, \vec{n} ) \vert \leqslant Ce^{c \Vert \vec{n}\Vert},$$
and for all $\vec{n}, \vec{n}'\in \Weyl{k}$, $t\in [0,T],$
$$ \vert h(t, \vec{n}) - h(t, \vec{n}')\vert \leqslant C \vert e^{c\Vert \vec{n}\Vert } - e^{c\Vert \vec{n}'\Vert }\vert, $$
where we define the norm of a vector in $\Weyl{k}$ by
$ \Vert \vec{n} \Vert = \sum_{i=1}^k n_i$.
\end{enumerate}
\label{def:trueevol}
\end{definition}

\begin{proposition}
Consider any initial data $h_0$ such that there exists constants $c,C >0$ such that  for all $\vec{n}\in \Weyl{k}$,
$ \vert h_0( \vec{n} ) \vert \leqslant Ce^{c \Vert \vec{n}\Vert}$, and for all $\vec{n}, \vec{n}'\in \Weyl{k}$, $ \vert h_0( \vec{n}) - h_0( \vec{n}')\vert \leqslant C \vert e^{c\Vert \vec{n}\Vert} - e^{c\Vert \vec{n}'\Vert}\vert $. Then the solution of the true evolution equation is unique.
\label{prop:uniqueness}
\end{proposition}
\begin{proof}
We provide a probabilistic proof adapted from \cite[Appendix C]{borodin2012duality}. Given $\vec{n}(t)$, a $q$-Hahn AZRP started from initial condition $\vec{n}(0)=\vec{n}$, we use a representation of any solution to the true evolution equation as a functional of the $q$-Hahn AZRP.

Let $h^1$ and $h^2$ two solutions of the true evolution equation with initial data $h_0$. Then $g:=h^1-h^2$ solves the true evolution equation with zero initial data. Let $T>0$. Our aim is to prove that for any $\vec{n}\in \Weyl{k}$, $g(T, \vec{n})=0$.  The idea is the following: By \emph{formally} differentiating the function $t\mapsto  \EE^{\vec{n}}[g(t, \vec{n}(T-t))]$ we find a zero derivative. Thus we expect that this function is constant, and hence its value for $t=T$, which is $g(T, \vec{n})$, equals the limit when $t$ goes to zero, which is expected to be $0$. Of course, these formal manipulations need to be justified and we will see how condition (3) of the true evolution equation applies.

By condition (3) of the true evolution equation, there exist constants $c, C>0$  such that for $t\in [0, T]$,
\begin{equation}
 \vert g( t, \vec{n}) \vert \leqslant Ce^{c \Vert \vec{n}\Vert}.
 \label{eq:exponentialbound}
\end{equation}

Let us first prove that on $[0, T]$, $\Vert \vec{n}(t)\Vert - \Vert \vec{n} \Vert $ can be bounded by
 a Poisson random variable $N_{T}$. Indeed, we have that for any $0\leqslant t\leqslant T$,
 $$ \PP^{\vec{n}}\left( \Vert \vec{n}(t)\Vert - \Vert \vec{n} \Vert =N\right) \leqslant \PP\left(\text{ at least }\frac{N}{k}\text{ events on the right occurred on } [0, T] \right).$$
The rate of an event on the right is crudely bounded by $ k \lambda$ where $\lambda=\max_{j\leqslant m \leqslant k} \phi^L(j\vert m)<\infty$. Thus, $\Vert \vec{n}(t) \Vert - \Vert \vec{n} \Vert$ can be bounded by a Poisson random variable $N_{T}$ depending only on the horizon time $T$.

Consider the function $[0, T]\rightarrow \R$, $t\mapsto \EE^{\vec{n}}[g(t, \vec{n}(T-t))]$. Given the exponential bound (\ref{eq:exponentialbound}) and the inequality $\Vert \vec{n}(t)\Vert\leqslant \Vert \vec{n} \Vert + N_{T}$, this function is well-defined. Moreover, one can apply dominated convergence to show that it is continuous. Thus, the limit when $t$ goes to zero is zero (because of the initial condition for $g$).

Let us show that the function is constant. First, observe that for $t\in [0, T]$,
$$ B_{q, \nu}g(t, \vec{n}) \leqslant \sum_{ \vec{n}\to\vec{n}'} 2 k \lambda \vert g(t, \vec{n}') \vert \leqslant (2 k)^2 \lambda
C e^{c(\Vert \vec{n}\Vert +k)}.$$
Since $\vec{n}(T- t)$ can be bounded by $\Vert \vec{n} \Vert + N_{T}$, 
\begin{equation}
\left| B_{q, \nu}g(t, \vec{n}(T -t)) \right| \leqslant (2 k)^2 \lambda
C e^{c(\Vert \vec{n}\Vert +k+N_{T})}.
\label{eq:boundoperator}
\end{equation}
Consider the function $\phi : [0, T]^2 \rightarrow \R$ defined by $\phi(t, s) = \EE^{\vec{n}}[g(t, \vec{n}(s))]$. Since the right-hand-side of (\ref{eq:boundoperator}) is integrable, one can take the partial derivative of $\phi$ with respect to $t$ inside the expectation, and we get
\begin{equation*}
\frac{\partial \phi}{\partial t} (t, s) = \EE^{\vec{n}}[B_{q, \nu} g(t, \vec{n}(s))]
\end{equation*}
The equality comes from condition (1) of true evolution equation, using dominated convergence. By condition (3) of the true evolution equation, we also have that for $t\in [0,T]$,
\begin{equation}
\vert g(t, \vec{n}) - g(t, \vec{n}') \vert \leqslant C \vert e^{c\Vert \vec{n} \Vert} - e^{c\Vert \vec{n}' \Vert}\vert.
\end{equation}
Hence, for any fixed $t\in  [0,T]$, we have for $0<s<s'<T$
\begin{equation}
\Big\vert \frac{\phi(t,s')- \phi(t, s)}{s'-s} \Big\vert \leqslant C \EE^{\vec{n}} \left[ \frac{\vert e^{c\Vert\vec{n}(s') \Vert}- e^{c\Vert\vec{n}(s) \Vert}\vert}{s-s'}\right].
\label{eq:derivativebyhand}
\end{equation}
Since one can bound $\vert \Vert \vec{n}(s)\Vert - \Vert \vec{n}(s')\Vert \vert $ by a Poisson random variable with parameter proportional to $s'-s$, the right-hand-side of (\ref{eq:derivativebyhand}) has a limit when $s'$ goes to $s$. This means that for any $t\in [0,T]$, the function $\vec{n}\mapsto g(t, \vec{n})$ is in the domain of the semi-group (of the $q$-Hahn AZRP). Thus, applying Kolmogorov backward equation and using the commutativity of the generator with the semi-group, we have that
\begin{equation*}
\frac{\partial \phi}{\partial s} (t, s) = \EE^{\vec{n}}[B_{q, \nu} g(t, \vec{n}(s))].
\end{equation*}
Consequently the derivative of $t\mapsto \EE^{\vec{n}}[g(t, \vec{n}(T-t))]$ is zero. Hence the function is constant, and the value at $t=T$, $ g(T, \vec{n})$ equals the limit when $t\to 0$ which is zero.
\end{proof}

\begin{corollary} \label{cor:duality}
For any fixed $\vec{x}\in\mathbb{X}^{\infty}$,
the function $u : \R_+ \times \Weyl{k} \to \R$ defined by
$$u(t, \vec{n}) = \EE^{\vec{x}}[H(\vec{x}(t), \vec{n})]$$
 satisfies the true evolution equation with initial data $h_0(\vec{n}) = H(\vec{x}, \vec{n})$. As a consequence,
the $q$-Hahn AEP and the $k$-particle $q$-Hahn AZRP are dual with respect to the function $H$, that is for any
$\vec{x}\in \mathbb{X}^{\infty}$ and $\vec{n}\in \Weyl{k}$,
$$ \EE^{\vec{x}}[H(\vec{x}(t), \vec{n})]  = \EE^{\vec{n}}[H(\vec{x}, \vec{n}(t))].$$
\end{corollary}
\begin{proof}
By the Kolmogorov backward equation for the $q$-Hahn AZRP, it is clear that $(t, \vec{n})\mapsto \EE^{\vec{n}}[H(\vec{x}, \vec{n}(t))]$ satisfies the true evolution equation with initial data $  \EE[H(\vec{x}, \vec{n})]$ (the growth condition is clear). On the other hand, Kolmogorov backward equation for the $q$-Hahn AEP yields
$$ \frac{\rm d}{\mathrm{d}t} \EE^{\vec{x}}[H(\vec{x}(t), \vec{n})]  = T_{q, \nu}\EE^{\vec{x}}[H(\vec{x}(t), \vec{n})] = \EE^{\vec{x}}[ T_{q, \nu} H(\vec{x}(t), \vec{n})].$$
Proposition \ref{prop:duality} then implies
$$ \frac{\rm d}{\mathrm{d}t} u(t, \vec{n})  = \EE^{\vec{x}}[ B_{q, \nu} H(\vec{x}(t), \vec{n})] = B_{q, \nu}u(t, \vec{n}).$$
Since $u$ satisfies the growth condition and the initial condition, $u$ solves the true evolution equation. Hence, by Proposition \ref{prop:uniqueness}, we have that for all $\vec{x}\in \mathbb{X}^{\infty}$ and $\vec{n}\in \Weyl{k}$,
$$ \EE^{\vec{x}}[H(\vec{x}(t), \vec{n})]  = \EE^{\vec{n}}[H(\vec{x}, \vec{n}(t))].$$
\end{proof}

\subsection{Bethe ansatz solvability and moment formulas}
In light of Corollary \ref{cor:duality}, in order to compute $\mathbb{E}\left[\prod_{i=1}^{k}q^{x_{n_i}(t)+n_i} \right]$, we must solve the true evolution equation. Proposition \ref{prop:systemode} provides a rewriting of the $k$-particle true evolution equation as a $k$-particle free evolution equation with $k-1$ two-body boundary conditions.
\begin{proposition}
Let $\vec{x}(\cdot)$ denote the $q$-Hahn AEP.
If  $u:\R_+ \times \Z^k \to \C$ solves:
\begin{enumerate}
\item {\em ($k$-particle free evolution equation)} for all $\vec{n}\in \Z^k$ and $t\in \R_{+}$,
$$
\frac{\mathrm{d}}{\mathrm{d}t}  u(t;\vec{n}) = \frac{1-q}{1-\nu} \sum_{i=1}^{k} \Big[R \left( u(t;\vec{n}_{i}^{-}) - u(t;\vec{n}) \right) + L \left( u(t;\vec{n}_{i}^{+}) - u(t;\vec{n}) \right)\Big];
$$
\item {\em ($k-1$ two-body boundary conditions)} for all $\vec{n}\in \Z^k$ and $t\in \R_{+}$ if $n_i= n_{i+1}$  for some $i\in \lbrace 1,\ldots, k-1\rbrace$ then
$$
\alpha u(t; \vec{n}_{i,i+1}^{-}) +\beta u(t;\vec{n}_{i+1}^{-}) + \gamma u(t;\vec{n}) -u(t;\vec{n}_i^{-}) = 0
$$
where the parameters $\alpha,\beta,\gamma$ are defined in terms of $q$ and $\nu$ as
$$
\alpha = \frac{\nu(1-q)}{1-q\nu},\qquad \beta= \frac{q-\nu}{1-q\nu},\qquad \gamma = \frac{1-q}{1-q\nu};
$$
\item {\em (initial data)} for all $\vec{n}\in W_k$, $u(0;\vec{n}) = \mathbb{E}\left[\prod_{i=1}^{k}q^{x_{n_i}(0)+n_i} \right]$;
\item for any $T>0$, there exists constants $c, C>0$ such that for all $\vec{n}\in W_k$, $t\in [0,T]$,
$$ \vert u(t; \vec{n})\vert \leqslant  Ce^{c\Vert \vec{n}\Vert}, $$
and for all $\vec{n}, \vec{n}'\in \Weyl{k}$, $t\in [0,T]$,
$$ \vert h(t, \vec{n}) - h(t, \vec{n}')\vert \leqslant C \vert e^{c\Vert \vec{n}\Vert } - e^{c\Vert \vec{n}'\Vert }\vert ;$$
\end{enumerate}
then for all $\vec{n}\in W_k $ and all $t\in\R_+$, $u(t;\vec{n}) = \mathbb{E}\left[\prod_{i=1}^{k}q^{x_{n_i}(t)+n_i} \right]$.
\label{prop:systemode}
\end{proposition}
\begin{proof}
In the totally asymmetric case, that is when $R=1$ and $L=0$, this result can be seen as a degeneration of Proposition 1.7 in \cite{corwin2014q}.

First we show that conditions (1) and (2) imply that $u$ satisfies condition (1) of the true evolution equation in Definition \ref{def:trueevol}.
Condition (2) in Proposition \ref{prop:systemode} says that for all $\vec{n}$ such that $n_i=n_{i+1}$,
\begin{equation}
 \frac{\nu(1-q)}{1-q\nu} u(t; \vec{n}_{i,i+1}^{-}) +\frac{q-\nu}{1-q\nu} u(t;\vec{n}_{i+1}^{-}) + \frac{1-q}{1-q\nu} u(t;\vec{n}) -u(t;\vec{n}_i^{-}) = 0.
 \label{eq:BC1}
\end{equation}
This is equivalent to saying that for all $\vec{n}$ such that $n_i=n_{i+1}$,
\begin{equation}
 \frac{\nu^{-1}(1-q^{-1})}{1-q^{-1}\nu^{-1}} u(t; \vec{n}_{i,i+1}^{+}) +\frac{q^{-1}-\nu^{-1}}{1-q^{-1}\nu^{-1}} u(t;\vec{n}_{i}^{+}) + \frac{1-q^{-1}}{1-q^{-1}\nu^{-1}} u(t;\vec{n}) -u(t;\vec{n}_{i+1}^{+}) = 0.
 \label{eq:BC2}
\end{equation}
Indeed, if we set $\vec{m}:= \vec{n}_{i,i+1}^{-}$ in (\ref{eq:BC1}),  we have that $\vec{n}_{i+1}^{-} = \vec{m}_{i}^{+}, \vec{n} = \vec{m}_{i, i+1}^{+}$ and $\vec{n}^-_{i} = \vec{m}^+_{i+1} $. Dividing the numerator and the denominator of each coefficient in (\ref{eq:BC1}) by $-q\nu$, we have
\begin{align*}
\frac{\nu(1-q)}{1-q\nu} u(t; \vec{n}_{i,i+1}^{-})=\frac{1-q^{-1}}{1-q^{-1}\nu^{-1}} u(t;\vec{m}),  \\
 \frac{q-\nu}{1-q\nu} u(t;\vec{n}_{i+1}^{-}) =  \frac{q^{-1}-\nu^{-1}}{1-q^{-1}\nu^{-1}} u(t;\vec{m}_{i}^{+}), \\
 \frac{1-q}{1-q\nu} u(t;\vec{n}) =  \frac{\nu^{-1}(1-q^{-1})}{1-q^{-1}\nu^{-1}} u(t; \vec{m}_{i,i+1}^{+}).
\end{align*}
Finally we get exactly (\ref{eq:BC2}) with $\vec{n}$ replaced by $\vec{m}$.

The next lemma explains the effect of the boundary condition.
\begin{lemma}
Suppose that a function $f:\Z^m \to \R$ satisfies the boundary conditions that for all $\vec{n}$ such that $n_i=n_{i+1}$ for some $i\in \{1,\ldots, k-1\}$,
\begin{equation*}
\alpha f(\vec{n}_{i,i+1}^{-}) +\beta f(\vec{n}_{i+1}^{-}) + \gamma f(\vec{n}) -f(\vec{n}_i^{-}) = 0.
\end{equation*}
Then for $\vec{n} = (n, \dots, n)$, the function $f$ also satisfies
\begin{equation}
\sum_{i=1}^{m} R\frac{1-q}{1-\nu}\left(f(\vec{n}_{i}^{-}) - f(\vec{n})\right) = \sum_{j=1}^{m} \phi^R_{q,\nu}(j|m) f(\underbrace{n,\ldots,n}_{m-j},\underbrace{n-1,\ldots, n-1}_{j}),
\label{eq:mbodyright}
\end{equation}
and
\begin{equation}
\sum_{i=1}^{m} L\frac{1-q}{1-\nu}\left(f(\vec{n}_{i}^{+}) - f(\vec{n})\right) = \sum_{j=1}^{m} \phi^L_{q,\nu}(j|m) f(\underbrace{n+1,\ldots, n+1}_{j},\underbrace{n,\ldots, n}_{m-j}).
\label{eq:mbodyleft}
\end{equation}
\label{lem:mbody}
\end{lemma}
\begin{proof}
Equation (\ref{eq:mbodyright}) is exactly the conclusion of Lemma 2.4 in \cite{corwin2014q} with $\mu=\nu+(1-q)\epsilon$ and keeping only the terms of order $\epsilon$. For completeness, we will give a direct proof as well. Theorem 1 in \cite{povolotsky2013integrability} states that an associative algebra generated by $A,B$ obeying the quadratic homogeneous relation
\begin{equation}
 BA= \alpha AA + \beta AB + \gamma BB,
 \label{eq:quadraticrelation}
\end{equation}
enjoys the following non-commutative analogue of Newton binomial expansion
$$ \left(\frac{\mu-\nu}{1-\nu} A + \frac{1-\mu }{1-\nu} B \right)^m  = \sum_{j=0}^m \varphi_{q, \mu, \nu}(j\vert m)A^j B^{m-j}.$$
Let $\mu=\nu  + (1-q)\epsilon $ and consider only the terms of order $\epsilon$ as $\epsilon \to 0$ in the above expression. By identification of $O(\epsilon)$ terms, we have
\begin{equation}
 \sum_{i=1}^m \frac{1-q}{1-\nu} B^{i-1}  A B^{m-i} = \sum_{j=1}^{m} R^{-1}\phi^R_{q,\nu}(j|m) A^{j} B^{m-j}.
 \label{eq:binomialdegenerated}
\end{equation}
Interpreting each monomial of the form $X_1 X_2\dots X_m$ with $X_i\in\lbrace A,B \rbrace$ as $f(n_1, \dots, n_m)$ where $n_i=n$ if $X_i=B$ and $n_i=n-1$ if $X_i=A$, the boundary condition in the statement of the Lemma corresponds algebraically to the quadratic relation (\ref{eq:quadraticrelation}). Thus we find that for $\vec{n} = (n, \dots, n)$, $f$ satisfies
$$
\sum_{i=1}^{m} R\frac{1-q}{1-\nu}\left(f(\vec{n}_{i}^{-}) - f(\vec{n})\right) = \sum_{j=1}^{m} \phi^R_{q,\nu}(j|m) f(\underbrace{n,\ldots,n}_{m-j},\underbrace{n-1,\ldots, n-1}_{j}).
$$
Since (\ref{eq:binomialdegenerated}) is true as an identity in an algebra over the field of rational fractions in $q$ and $\nu$,  we can certainly replace $q$ and $\nu$ by their inverses. Keeping in mind (\ref{eq:inversionrates}), we find that
\begin{equation}
 \sum_{i=1}^m \frac{1-q^{-1}}{1-\nu^{-1}} B^{i-1}  A B^{m-i} = \frac{\nu}{q}\sum_{j=1}^{m} L^{-1}\phi^L_{q,\nu}(j|m) A^{j} B^{m-j}.
 \label{eq:binomialdegeneratedinversed}
\end{equation}
Interpreting the monomials as $f(n_1, \dots, n_m)$ with $n_i=n$ or $n+1$, we get that
$$
\sum_{i=1}^{m} L\frac{1-q}{1-\nu}\left(f(\vec{n}_{i}^{+}) - f(\vec{n})\right) = \sum_{j=1}^{m} \phi^L_{q,\nu}(j|m) f(\underbrace{n+1,\ldots, n+1}_{j},\underbrace{n,\ldots, n}_{m-j}).
$$
\end{proof}
The application of Lemma \ref{lem:mbody} for each cluster  of equal elements in $\vec{n}$  shows that under conditions (1) and (2), $u(t; \vec{n})$ satisfies condition (1) of Definition \ref{def:trueevol}
$$  \frac{\mathrm{d}}{\mathrm{d}t} h(t, \vec{n}) = B_{q, \nu} h(t, \vec{n}).$$
The growth condition (3) of the true evolution equation is exactly the same as  condition (4) of the Proposition with the same constants $c, C$, and
the initial data are the same.
Hence, if $u$ satisfies the conditions of the Proposition, it solves the true evolution equation with initial data $h_0(\vec{y})=H(\vec{x}, \vec{y})$, and by Proposition \ref{prop:uniqueness},  $u(t;\vec{n}) = \mathbb{E}\left[\prod_{i=1}^{\infty}q^{x_{n_i}(t)+n_i} \right]$.
\end{proof}
\begin{remark}
In the case $\nu=q$, the system of  ODEs with two-body boundary conditions in Proposition \ref{prop:systemode} was already known, see (10) and (12) in \cite{sasamoto1998one}.
\end{remark}

Proposition \ref{prop:nestedcontours} provides an exact contour integral formula for the observables  $\mathbb{E}\left[\prod_{i=1}^{k}q^{x_{n_i}(t)+n_i} \right]$. We simply check that the formula is a solution to the true evolution equation, using Proposition \ref{prop:systemode}. This type of formula arises in the theory of Macdonald processes  \cite{borodin2014macdonald} (though the $q$-Hahn processes do not fit in that framework) and in Bethe ansatz \cite{borodin2012duality, borodin2014spectral}.
\begin{proposition} Fix $q\in (0,1)$, $0\leqslant  \nu <1$, and an integer $k$. Consider the $q$-Hahn AEP started from step initial data (i.e. $x_n(0)=-n$ for $n\geq 1$). Then for any $\vec{n}\in \Weyl{k}$,
\begin{multline}
\EE\left[\prod_{i=1}^{k} q^{x_{n_i}(t)+n_i}\right] = \frac{(-1)^k q^{\frac{k(k-1)}{2}}}{(2\pi i)^k} \oint_{\gamma_1} \cdots \oint_{\gamma_k} \prod_{1\leq A<B\leq k} \frac{z_A-z_B}{z_A-q z_B} \\
 \prod_{j=1}^{k} \left(\frac{1-\nu z_j}{1-z_j}\right)^{n_j} \exp\left((q-1)t \left(\frac{Rz_j}{1-\nu z_j} - \frac{Lz_j}{1-z_j} \right) \right) \frac{dz_j}{z_j(1-\nu z_j)}.
\label{eq:qmomentsgeneral}
\end{multline}
where the integration contours $\gamma_1,\ldots, \gamma_k$ are chosen so that they all contain $1$, $\gamma_A$ contains $q\gamma_B$ for $B>A$ and all contours exclude $0$ and $1/\nu$.
\label{prop:nestedcontours}
\end{proposition}
\begin{proof}
We prove that the right-hand-side of (\ref{eq:qmomentsgeneral}) verifies the conditions of Proposition \ref{prop:systemode}. Note that (\ref{eq:qmomentsgeneral}) is very similar with the result of Theorem 1.9 in \cite{corwin2014q} for the discrete-time $q$-Hahn TASEP, the only difference being that the factor $((1-\mu z_j)/(1-\nu z_j))^t$ is replaced by
$$\exp\left((q-1)t \left(\frac{Rz_j}{1-\nu z_j} - \frac{Lz_j}{1-z_j} \right) \right).  $$
Let us explain briefly why conditions (2) and (3) are verified: As it is explained in the proof of Theorem 1.9 in \cite{corwin2014q}, the application of the boundary condition to the integrand brings out an additional factor
$$ \frac{(1-\nu)^2}{(1-q\nu)(1-\nu z_i)(1-\nu z_{i+1})}(z_i - qz_{i+1}).$$
The factor $(z_i - qz_{i+1})$ cancels out the pole separating the contours for the variables $z_i$ and $z_{i+1}$. We may then take the same contour and use antisymmetry to prove that the integral is zero. To check the initial data, one may observe by residue calculus that the integral is zero when $n_k\leq 0$ since there is no pole at $1$ for the $z_k$ integral; and one verifies that the integral equals $1$ in the alternative case by sending the contours to infinity (this is the same calculation as in \cite{corwin2014q}).

Let us check the free evolution equation. The generator of the free evolution equation can be written as a sum $\sum_{i=1}^{k} \mathcal{L}_i$  where $\mathcal{L}_i$ acts by
$$\mathcal{L}_i f = \frac{1-q}{1-\nu} \left[R\left(f(\vec{n}^{-}_i)-f(\vec{n}) \right) +L \left(f(\vec{n}^{-}_i)-f(\vec{n}) \right)\right].$$
 Applying $ \mathcal{L}_i $ to the R.H.S of (\ref{eq:qmomentsgeneral}) brings inside the integration a factor
\begin{equation*}
\frac{1-q}{1-\nu} \left( R \left(\frac{1-z_i}{1-\nu z_i} -1\right) + L \left(\frac{1- \nu z_i}{1- z_i} -1\right) \right)
\end{equation*}
which is readily shown to equal the argument of the exponential.

Finally, let us check the growth condition. Let us denote by $\tilde{u}(t, \vec{n})$ the right-hand-side of (\ref{eq:qmomentsgeneral}). One can choose the contours $\gamma_1, \dots, \gamma_k$ such that for all $1\leqslant A<B\leqslant k $ and $1\leqslant j \leqslant k$, $\vert z_A- q z_B\vert$, $\vert 1-z_j\vert $, $\vert 1-\nu z_j\vert $ and $\vert z_j\vert $ are uniformly bounded away from zero. Since the contours are finite, one can find constants $c_1$, $c_2$ and $c_3$,  such that for any $t$ smaller that some arbitrary but fixed constant $T$,
\begin{equation*}
\left| \tilde{u}(t, \vec{n})\right| \leqslant c_1 \prod_{j=1}^k \left( c_2^{n_j} \exp((1-q)t c_3)\right),
\end{equation*}
and
\begin{equation*}
\vert \tilde{u}(t, \vec{n})- \tilde{u}(t, \vec{n}')\vert \leqslant c_1 \exp(k(1-q)t c_3)\vert c_2^{\Vert\vec{n} \Vert} - c_2^{\Vert\vec{n}' \Vert}\vert ,
\end{equation*}
where $c_1$, $c_2$ and $c_3$ depend only on the parameters $q, \nu$, the choice of contours and the horizon time $T$.
\end{proof}

\subsection{Fredholm determinant formulas}
Proposition \ref{prop:nestedcontours} provides a formula for all integer moments of the random variable $q^{x_n(t)+n}$ when the $q$-Hahn AEP is started from step initial condition. Since $q\in (0,1)$ and $x_n(t)+n\geq 0$, this completely characterizes the law of $x_n(t)$. In order to extract information out of these expressions, we give a Fredholm determinant formula for the $e_q$-Laplace transform of $q^{x_n(t)+n}$, following an approach designed initially for the study of Macdonald processes \cite{borodin2014macdonald}. The reader is referred to  \cite[Section 3.22]{borodin2014macdonald} for some background on Fredholm determinants.  In the totally asymmetric case ($L=0$), Theorem \ref{th:fredholmgeneral} can also be seen as a degeneration when $\epsilon$ goes to zero of Theorem 1.10 in \cite{corwin2014q}.

\begin{theorem}
Fix $q\in (0,1)$ and $0\leqslant  \nu <1$. Consider step initial data. Then for all $\zeta\in\C\setminus \R_+$, we have the ``Mellin-Barnes-type'' Fredholm determinant formula
\begin{equation}\label{eq:MellinBarnes}
\EE \left[\frac{1}{\left(\zeta q^{x_{n}(t)+n};q\right)_{\infty}}\right] = \det\left(I + K_{\zeta}\right)
\end{equation}
where $\det\left(I + K_{\zeta}\right)$ is the Fredholm determinant of $K_\zeta: L^2(C_1)\to L^2(C_1)$ for $C_1$ a positively oriented circle containing $1$ with small enough radius so as to not contain $0$, $1/q$ and $1/\nu$. The operator $K_\zeta$ is defined in terms of its integral kernel
\begin{equation*}
K_{\zeta}(w,w') = \frac{1}{2\pi i} \int_{-i \infty + 1/2}^{i\infty +1/2} \frac{\pi}{\sin(-\pi s)} (-\zeta)^s \frac{g(w)}{g(q^s w)} \frac{1}{q^s w - w'} ds
\end{equation*}
with
\begin{equation*}
g(w) = \left(\frac{(\nu w;q)_{\infty}}{(w;q)_{\infty}}\right)^{n}
\exp\left( (q-1)t \sum_{k=0}^{\infty} \frac{R}{\nu}\frac{\nu w q^k}{1-\nu w q^k} -L\frac{w q^k}{1- w q^k}\right)
 \frac{1}{(\nu w;q)_{\infty}}.
\end{equation*}

The following ``Cauchy-type'' formula also holds:
\begin{equation}
\EE \left[\frac{1}{\left(\zeta q^{x_{n}(t)+n};q\right)_{\infty}}\right] = \frac{\det\left(I+\zeta \tilde{K}\right)}{(\zeta ; q)_{\infty}},
\label{eq:Cauchytype}
\end{equation}
where $\det\left(I+\zeta \tilde{K}\right)$ is the Fredholm determinant of $\zeta$ times the operator $\tilde{K}:\mathbb{L}^2(C_{0,1})\to \mathbb{L}^2(C_{0,1})$ for $C_{0,1}$ a positively oriented circle containing $0$ and $1$ but not $1/\nu$, and the operator $\tilde{K}$ is defined by its integral kernel
$$ \tilde{K}(w,w')  = \frac{g(w)/g(qw)}{qw'-w}.$$
\label{th:fredholmgeneral}
\end{theorem}
\begin{proof}
We will sketch the main deductions which occur in the proof of the Mellin-Barnes type formula (\ref{eq:MellinBarnes}). Similar derivations (with all details given) of such Fredholm determinants from moment formulas can be found in \cite[Theorem 3.18]{borodin2014macdonald}, \cite[Theorem 1.1]{borodin2012duality} or more recently \cite[Theorem 1.10]{corwin2014q} and the proofs always follow the same general scheme (cf. \cite[Section 3.1]{borodin2012duality}). Propositions 3.2 to 3.6 in \cite{borodin2012duality} show that for $\vert \zeta \vert $ small enough and $C_1$ a positively oriented circle containing $1$ with small enough radius,
\begin{equation}
\sum_{k=0}^{\infty} \EE\left[ q^{k(x_n(t)+n)}\right] \frac{\zeta^k}{\left[k\right]_q!} = \det\left(I + K_{\zeta}\right),
\end{equation}
with $\left[k\right]_q!$ as in (\ref{eq:defqfactorial}). The only technical condition to verify is that 
$$ \sup\left\lbrace  \vert g(w)/g(wq^s)\vert  :\   w\in C_1, \ k\in \mathbb{Z}_{> 0}, \  s\in D_{R, d, k}\right\rbrace  <\infty.$$
 Here, $D_{R, d, k}$ is the contour depicted in \cite[Figure 3]{borodin2012duality}. Note that here $R$ is not the asymmetry parameter of the process but the radius of the circular part of the contour $D_{R, d, k}$. If one chooses $R$ large enough, $d$ small enough, and the radius of $C_1$ small enough, then $q^s w$ stay in a neighbourhood of the segment $[0,\sqrt{d}]$. The function $g$ has singularities at $q^{-n}$ and $\nu^{-1}q^{-n}$ for all $n\in \Z_{\geqslant 0}$. Hence for $w \in C_1$ a small but fixed circle around $1$, one can choose $R$ and $d$ such that  $q^s w$ stay in a compact region of the complex plane away from all singularities, and thus the ratio $\vert g(w)/g(wq^s)\vert$ remains bounded.

 By an application of the $q$-binomial theorem (\ref{eq:qbinomial}), for $\vert \zeta \vert <1 $ we also have that
$$ \sum_{k=0}^{\infty} \EE\left[ q^{k(x_n(t)+n)}\right] \frac{\zeta^k}{\left[k\right]_q!} = \EE\left[ \frac{1}{(\zeta q^{x_n(t)+n} ; q)_{\infty}}\right],$$
proving that (\ref{eq:MellinBarnes}) holds for $\vert \zeta \vert$ sufficiently small. Both sides of (\ref{eq:MellinBarnes}) can be seen to be analytic over $\mathbb{C}\setminus \R_+$. The left-hand side equals
$$ \sum_{k=0}^{\infty} \frac{\PP\big(x_n(t)+n = k\big)}{(1-\zeta q^k)(1-\zeta q^{k+1})\cdots}.$$
For any $\zeta\in\mathbb{C}\setminus \R_+$ the infinite products are uniformly convergent and bounded away from zero on a neighbourhood of $\zeta$, which implies that the series is analytic. The  right-hand side of (\ref{eq:MellinBarnes}) is
$$ \det\left(I + K_{\zeta}\right)= 1 = \sum_{n=1}^{\infty} \frac{1}{n!} \int_{C_1}\mathrm{d}w_1 \dots \int_{C_1}\mathrm{d}w_n \det\left(K_{\zeta}(w_i, w_j) \right)_{i,j=1}^n.$$
Due to exponential decay in $\vert s\vert$ in the integrand of $K_{\zeta}$, $\det\left(K_{\zeta}(w_i, w_j) \right)_{i,j=1}^n$ is analytic in $\zeta$ for all $w_1, \dots, w_n \in C_1$. Analyticity of the Fredholm expansion proceeds from absolute and uniform convergence of the series on a neighbourhood of $\zeta\not\in \R_+$. This can be shown using that $\vert g(w)/g(wq^s)\vert < const$ for $w\in C_1$ and $s\in 1/2+i\R$ and Hadamard's bound to control the determinant.

We do not prove explicitly the Cauchy-type Fredholm determinants but refer to the Section 3.2 in \cite{borodin2012duality} where a general scheme is explained to prove such formulas.
\end{proof}

\subsection{Some degenerations of the $q$-Hahn AEP}
\label{subsec:degenerations}

\subsubsection{Partially asymmetric generalizations of the $q$-TASEP} \label{subsubsec:qTASEP}

The limits of the $q$-Hahn weights when $\nu$ goes to zero and when $\epsilon = (\mu-\nu)/(1-q)$ goes to zero do not commute, thus several choices are possible in order to build continuous time, partially asymmetric versions of  the $q$-TASEP and the $q$-Boson process (see, e.g. \cite{borodin2012duality}). We investigate here
the case  when we first take $\epsilon$ to zero. This  corresponds to taking $\nu= 0$ in the rates $\phi^R_{q, \nu}$ and $\phi^L_{q, \nu}$. We have
\begin{equation}
\phi^R_{q, 0}(j\vert m) = R\left(1-q^m \right)\mathds{1}_{\lbrace j=1\rbrace} \qquad\text{ and }\qquad \phi^L_{q, 0}(j\vert m) = \frac{L}{[j]_q} \frac{(q;q)_m}{(q;q)_{m-j}}.
\end{equation}
 In the associated exclusion process, independently for each $n\geqslant 1$, the particle at location $x_n(t)$ jumps to $x_n(t)+1$ at rate $R(1-q^{\rm gap})$ (the gap being here $x_{n-1}(t)-x_{n}(t) -1$), and jumps to the location $x_{n}-j$ at rate $\big(L/[j]_q\big)\big( (q;q)_{\rm gap}/(q;q)_{{\rm gap}-j}\big)$, for all $j\in \lbrace 1, \dots , x_{n}-x_{n+1}-1\rbrace$ (the gap being here $x_{n}(t)-x_{n+1}(t) -1$). All the result in Section \ref{sec:continoustime} apply for the case $\nu=0$, and one could study this system in more details by analyzing the Fredholm determinant formula of Theorem \ref{th:fredholmgeneral}. A motivation for studying this process is that as $q$ goes to $1$,
\begin{equation}
\phi^R(j\vert m)  \approx R (1-q) m \mathds{1}_{\lbrace j=1 \rbrace} \qquad\text{ and }\qquad
\phi^L(j \vert m)  \approx L(1-q)m \mathds{1}_{\lbrace j=1 \rbrace}. \label{eq:rateslimitqto1}
\end{equation}
Thus, the rates on the left and on the right have the same expression at the first order in $1-q$, and the limit of this process when $q\to 1$ may be interesting.

\begin{remark}
There are other partially asymmetric generalizations of the $q$-TASEP which preserve its duality. One possibility is to send first  $\nu$ to zero in the expressions for $\varphi_{q, \mu, \nu}(j\vert m )$ and $\varphi_{q^{-1}, \mu^{-1}, \nu^{-1}}(j\vert m )$, and then take a continuous time limit.
%
Another generalization preserving duality has jumps to the right at rate $(1-q)[\rm{gap}]_q$ and to the left at rate  $(q^{-1}-1)[\rm{gap}]_{q^{-1}}$.
It is not clear if these processes are Bethe ansatz solvable, so we do not discuss them further here.
\end{remark}


\subsubsection{Totally asymmetric case}

When $R=1$ and $L=0$, we are in the totally asymmetric case. This case was studied  by Takeyama in \cite{takeyama2014deformation}. Indeed, the particle system defined in \cite{takeyama2014deformation} is a zero-range   process defined on $\Z$ controlled by two parameters $s$ and $q$. Particles move from site $i$ to $i-1$ independently for each $i\in\Z$, and the rate at which $j$ particles move to the left from a site occupied by $m$ particles is given by
\begin{equation*}
\frac{s^{j-1}}{[j]_q}\prod_{i=0}^{j-1} \frac{[m-i]_q}{1+s[m-1-i]_q}.
\end{equation*}
Setting $s=(1-q)\frac{\nu}{1-\nu} $, we find that
\begin{equation*}
\frac{s^{j-1}}{[j]_q}\prod_{i=0}^{j-1} \frac{[m-i]_q}{1+s[m-1-i]_q}\  = \ \phi^R_{q, \nu}(j\vert m).
\end{equation*}
\begin{remark}
The totally asymmetric version of the $q$-Hahn AEP\footnote{Here we mean the degeneration of the $q$-Hahn AEP when $L=0$, which is a continuous time Markov process, hence different from the discrete-time $q$-Hahn TASEP.}, is also the natural continuous time limit of the (discrete-time) $q$-Hahn TASEP, and it was already noticed in \cite{povolotsky2013integrability} that letting $\mu\to \nu$ and rescaling time was the right way of defining such a continuous time limit.
\end{remark}

\subsubsection{Multiparticle asymmetric diffusion model}
\label{subsec:MADM}

When $\nu=q$, the jump rates of the $q$-Hahn AZRP and AEP no longer depend on the gap between consecutive particles (or the number of particles on each site in the zero-range formulation). The rates are now given by $R/[j]_{q^{-1}}$ and $L/[j]_q$. The zero-range model with $N$ particles is exactly  the ``multi-particle asymmetric diffusion model'' introduced by Sasamoto and Wadati\footnote{\cite{sasamoto1998one} defined the model with the restriction that $R/L=q$.} in \cite{sasamoto1998one} and further studied by Lee \cite{lee2012current} (see also \cite{alimohammadi1999two, alimohammadi1998exact}).
For the corresponding exclusion process, we prove (by an asymptotic analysis of the Fredholm determinant in (\ref{eq:MellinBarnes})) in Section \ref{sec:MADM} that the rescaled positions of particles converge to the Tracy-Widom GUE distribution (Theorem \ref{thm:fluctuationsrarefactionfan}). The same results even holds for the first particle (Theorem \ref{thm:fluctuationsparticle1}).

\subsubsection{Push-ASEP}

Consider the $q$-Hahn AEP when $\nu=0$ (see Section \ref{subsubsec:qTASEP}), and let further $q=0$.  The process obtained after particle-hole inversion is known. Indeed, when $\nu=q=0$, $\phi^R(j\vert m) = \mathds{1}_{j=1}$ and $\phi^L(j\vert m) = 1$ for all $m\geqslant 1$. This corresponds to the Push-ASEP introduced in \cite{borodin2008large}, wherein convergence to the Airy process is proved.

\section{Predictions from the KPZ scaling theory}
\label{sec:KPZscaling}

In this section, we explain how asymptotics of our Fredholm determinant formula (Theorem \ref{th:fredholmgeneral}) confirms the universality predictions from the physics literature KPZ scaling theory \cite{krug1992amplitude, spohn2012kpz}. Although the original paper \cite{krug1992amplitude} on the KPZ scaling theory deals only with so-called single step models and directed random  polymers, the predictions can be straightforwardly adapted to any exclusion process. In particular, we compute the non-universal constants arising in one-point limit theorems for the $q$-Hahn AEP. In Section \ref{sec:MADM}, we provide a rigorous confirmation in the particular case corresponding to the MADM exclusion process.

Following \cite{spohn2012kpz}, we present the predictions of KPZ scaling theory in the context of exclusion processes. Assume that the translation invariant stationary measures for an exclusion process on $\Z$ with local dynamics are precisely labelled by the density of particles $\rho$, where
$$ \rho = \lim_{n\to\infty} \frac{1}{2n+1}\#\lbrace \text{particles between } -n\text{ and }n\rbrace. $$
We define the average steady-state current $j(\rho)$ as the expected number of particles going from site $0$ to $1$ per unit time, for a system distributed according to the stationary measure indexed by $\rho$. We also define the integrated covariance $A(\rho)$ by
$$ A(\rho) = \sum_{j\in \mathbb{Z}} Cov(\eta_0, \eta_j), $$
where $\eta_0,  \eta_j\in \lbrace 0,1 \rbrace$ are the occupation variables of the exclusion system at sites $0$ and $j$, and the covariance is taken under the $\rho$-indexed stationary measure. One expects that the rescaled particle density $\varrho(x, \tau)$, given heuristically by
\begin{equation}
\varrho(x, \tau) = \lim_{\tau\to\infty} \mathbb{P}(\text{There is a particle at }\lfloor x t \rfloor \text{ at time }t\tau )
\label{eq:defvarrho}
\end{equation}
 satisfies the conservation equation (subject to being a weak solution that satisfies the entropy condition)
\begin{equation}
\frac{\partial }{\partial \tau}\varrho(x, \tau) + \frac{\partial}{\partial x}j(\varrho(x, \tau)) = 0,
\label{eq:conservationPDE}
\end{equation}
with some initial condition which is $\varrho(x, 0) = \mathds{1}_{x<0}$ for the step initial condition.

This hydrodynamic behaviour can also be phrased in terms of a law of large numbers for the position of particles. For $\kappa\geqslant 0$, if $n$ and $t$ go to infinity with $n=\lfloor \kappa t\rfloor$, there is a constant $\pi$ such that\footnote{Note that we do not prove this law of large numbers in terms of  almost-sure limit, our results only imply the convergence in probability}
\begin{equation}
\frac{x_n(t)}{t}\xrightarrow[t\to\infty]{} \pi.
\label{eq:LLNgeneral}
\end{equation}
It turns out that instead of expressing $\pi$ as a function of $\kappa$, it is more convenient to parametrize it by the local density $\rho$ around the macroscopic position $\pi$. Under the assumption that such a parametrization exists (it is the case starting from step initial condition), the definition of $\varrho$ in \eqref{eq:defvarrho} implies that  $\pi(\rho)$ is determined by $\rho=\varrho(\pi(\rho), 1)$. We parametrize $\kappa$ by $\rho$ as well and define $\kappa(\rho)$ such that the law of large numbers \eqref{eq:LLNgeneral} holds: Otherwise said, $\kappa(\rho)$ is the rescaled integrated current at the macroscopic position $\pi(\rho)$ (i.e. the limit as $t\to\infty $ of the number of particles sitting on the right of position $\pi(\rho)t$, divided by $t$).

\begin{KPZclass}
Let $\lambda(\rho) = -j''(\rho)$. For $\rho$ such that $\lambda(\rho) \neq 0$, the KPZ class conjecture states that starting from the step initial condition
\begin{equation}
\lim_{t\to\infty} \PP\left(\frac{x_{\lfloor \kappa(\rho)t \rfloor}(t)-t\pi(\rho)}{\sigma(\rho)t^{1/3}}\geqslant x\right) \xrightarrow[t\to\infty]{} F_{GUE}(-x),
\label{eq:KPZclassconjecture}
\end{equation}
where
\begin{eqnarray}
\pi(\rho) &=& \frac{\partial j(\rho)}{\partial \rho},\label{eq:generalpi}\\
\kappa(\rho) &=&  j(\rho) - \rho\pi(\rho),  \label{eq:generalkappa}\\
\sigma(\rho ) &=& \left(\frac{\lambda(\rho) \big(A(\rho)\big)^2}{2\rho^3} \right)^{1/3}.
\end{eqnarray}
The precise definition of $F_{GUE}$ is given in Definition \ref{def:TW}.
 \end{KPZclass}
The conjecture that fluctuations occur in the scale $t^{1/3}$ dates back to \cite{kardar1986dynamic}. The expression for the magnitude of fluctuations $\sigma(\rho)$ was derived in \cite{krug1992amplitude}, and the limiting distribution was first discovered in the work of Johannsson on TASEP \cite{johansson2000shape}.

Let us explain how the expressions for $\pi(\rho)$ and $\kappa(\rho)$ are heuristically  derived.
The existence of the limit \eqref{eq:defvarrho} implies that with the above definition of $\pi(\rho)$, we have
\begin{equation}
\forall t >0,\ \  \varrho(\pi(\rho)t, t)=\rho.
\label{eq:determinepi}
\end{equation}
Differentiating \eqref{eq:determinepi} with respect to $t$ and using \eqref{eq:conservationPDE}, we find that
$\pi(\rho) = \frac{\partial j(\rho)}{\partial \rho}$.

The rescaled current $\kappa(\rho)$ is the rescaled number of particles between the first particle and the position $\pi(\rho)t$.  Since integrating the density over space counts the number of particles, we can write that
\begin{equation}
\kappa(\rho)t = \int_{\pi(\rho)t}^{\pi(\rho_0)t} \varrho(x, t)\mathrm{d}x,
\label{eq:determinekappa1}
\end{equation}
where $\rho_0$ is the density around the first particle. As we have already explained in Section \ref{subsec:motivations}, $\rho_0$ may not be $0$ (we will see that $\rho_0>0$ for the $q$-Hahn AEP when $L>0$). Making the change of variables $x=\pi(\rho)t$ in \eqref{eq:determinekappa1}, we get that
$$  \kappa(\rho) = \int_{\rho}^{\rho_0} \rho\mathrm{d}\pi(\rho).$$
Integrating by parts and using \eqref{eq:generalpi} yields
\begin{equation}
\kappa(\rho) = \rho_0 \pi(\rho_0) - \rho \pi(\rho) +j(\rho) - j(\rho_0).
\label{eq:generalkappa1}
\end{equation}
Equation \eqref{eq:generalkappa1} is not satisfactory since $\rho_0$ is unknown. However, we can also determine $\kappa(\rho)$ by observing that we know $\kappa(\rho)$ at the left edge of the rarefaction fan. Since we start from step initial condition,
for any fixed $t$,  $x_N(t)= -N$ for $N$ large enough.  Hence, assuming that the density is continuous and equal $1$ at the left edge of the rarefaction fan, one has $\kappa(1)= -\pi(1)$.
\begin{equation}
\kappa(1)t - \kappa(\rho)t = \int_{\pi(1)t}^{\pi(\rho)t} \varrho(x, t)\mathrm{d}x.
\label{eq:determinekappa2}
\end{equation}
Again, by the change of variables $x=\pi(\rho)t$ and integrating by parts, in \eqref{eq:determinekappa2}, we find that
$$ \kappa(1)- \kappa(\rho) =\rho \pi(\rho)-  \pi(1) + j(1) - j(\rho),$$
and since $j(1)=0$ and $\kappa(1)= -\pi(1)$, we get that
\begin{equation}
\kappa(\rho)  = j(\rho) - \rho\pi(\rho)
 \label{eq:generalkappa2}
\end{equation}
as claimed in our statement of  the KPZ class Conjecture. Furthermore, by combining \eqref{eq:generalkappa1} and \eqref{eq:generalkappa2}, we get that $ j(\rho_0) = \rho_0 \pi(\rho_0)$. In other words,
\begin{equation}
 \frac{j(\rho_0)}{\rho_0}  = \frac{\partial j}{\partial \rho}(\rho_0),
 \label{eq:equationforrho_0}
\end{equation}
which means that $\rho_0$ is the argmax of the steady-state drift.

\begin{remark}
The magnitude of fluctuations in \cite[Equation (2.14)]{spohn2012kpz} slightly differs from our expression $\lambda(\rho) \left(A(\rho)\right)^2/(2\rho^3)$. This is because \cite{spohn2012kpz} considers fluctuations of the height function. The fluctuation of the height function is twice the fluctuations of the integrated current. And the fluctuations of the current are, on average, $\rho$ times the fluctuations of a tagged particle. Then, the quantities $j(\rho)$  and $A(\rho)$ defined in \cite{spohn2012kpz} differs from ours by a factor $2$ and $4$ respectively. Moreover, since we consider step-initial condition with particles on the left, it is more convenient to drop the minus sign. That is why the scale $\Big(-\frac{1}{2}\lambda(\rho)  A(\rho)^2 \Big)^{1/3}$ becomes $\Big(\lambda(\rho) A(\rho)^2/\big(2\rho^3\big)\Big)^{1/3}$.
\end{remark}
\subsection{Hydrodynamic limit}
\label{subsec:hydro}
In the case of the $q$-Hahn AEP, there exist translation invariant and stationary measures $\mu_{\alpha}$ indexed by a parameter $\alpha\in (0,1)$ such that the gaps between particles $(x_n-x_{n+1}-1)$ are independent  and identically distributed according to
\begin{equation}
\mu_{\alpha}({\rm gap}=m) =  \alpha^m \frac{(\nu ; q)_m}{(q ; q)_m} \frac{(\alpha ; q)_{\infty}}{(\alpha\nu ; q)_{\infty}}.
\label{eq:statmeasure}
\end{equation}
 Let us explain (without proof details) why these are stationary: It is known from a study of a more general family of zero-range processes on a ring \cite{evans2004factorized} that this measures are stationary for the (discrete time) $q$-Hahn TASEP on a ring (see also \cite{povolotsky2013integrability}). This implies that they are stationary as well in the infinite volume setting considered in \cite{corwin2014q}. By taking a limit of the transition matrix of the $q$-Hahn TASEP when $\mu$ goes to $\nu$, the measures $\mu_{\alpha}$ are also stationary for the totally asymmetric continuous-time case. Since the family of measures  $\mu_{\alpha}$ is stable by inversion of the parameters $q$ and $\nu$, they are also stationary in the two-sided case which is a linear combination of the one-sided ones.

Fix $q \in(0,1), \nu\in[0,1)$ and assume $L=1-R$, without loss of generality.
By the renewal theorem, the density $\rho$ under the stationary measure $\mu_{\alpha}$  is given by
\begin{equation}
\rho = \frac{1}{1+\EE[\rm gap]},
\end{equation}
where
\begin{eqnarray*}
\EE[\rm gap]&=& \sum_{m=0}^{\infty} m \alpha^m \frac{(\nu ; q)_m}{(q ; q)_m} \frac{(\alpha ; q)_{\infty}}{(\alpha\nu ; q)_{\infty}}, \\
&=&\alpha \frac{\rm d}{\rm{d}\alpha} \mathrm{log}\left(\frac{(\alpha\nu ; q)_{\infty}}{(\alpha ; q)_{\infty}}\right),\\
&=& \frac{1}{\log(q)} \left( \Psi_q(\theta) - \Psi_q(\theta + V)\right);
\end{eqnarray*}
with $\theta=\log_q(\alpha)$ and $V=\log_q(\nu)$.
Summarizing, for $\alpha=q^{\theta}$, we define the density parametrized by $\theta$ as
\begin{equation}
\rho(\theta) = \frac{\log(q)}{\log(q) + \Psi_q(\theta) - \Psi_q(\theta + V)}.
\label{eq:formularho}
\end{equation}
Let us compute the average steady-state current $j(\rho)$. By averaging the empirical current of particles over a large box under the stationary measure, we find that
$$ j(\rho) = \rho \ \cdot \EE[{\rm drift}],$$
where the drift is the average speed of a tagged particle.
 For $\rho$ corresponding to the parameter $\alpha$ (or
$\theta$) as above,  we have
\begin{eqnarray*}
j(\rho) &=&\rho\ \cdot \sum_{m=0}^{\infty} \alpha^m \frac{(\nu ; q)_m}{(q ; q)_m} \frac{(\alpha ; q)_{\infty}}{(\alpha\nu ; q)_{\infty}} \left(\sum_{j=1}^{m} j \phi^R_{q, \nu}(j\vert m) - \sum_{j'=1}^{m} j' \phi^L_{q, \nu}(j'\vert m)  \right), \\
&=& \rho \alpha \frac{\rm d}{\rm{d}\alpha}\left( \frac R \nu \ G_q(\alpha \nu) - L \ G_q(\alpha)\right),\ \ \ (\text{see Section \ref{sec:preliminaries} for the Def. of }G_q)\\
&=& \rho \frac{1-q}{\log(q)^2} \left(\frac R \nu \ \Psi_q'(\theta+V) - L \ \Psi_q'(\theta) \right) ;
\end{eqnarray*}
where we have used the $q$-binomial theorem (\ref{eq:qbinomial}) to sum over $m$ in the second equality and we have used Lemma \ref{lem:equivalentseries} for the third equality.
The functions $\pi$, $\kappa$ and $\sigma$ that arise in limit theorems (\ref{eq:LLNgeneral}) and (\ref{eq:KPZclassconjecture})
are written as  functions of the density $\rho$, but given the formula (\ref{eq:formularho}), one can express all quantities as functions of the $\theta$ variable. In the following, we compute the exact expression of these quantities for the $q$-Hahn AEP.  Since the dynamics depend on parameters $q$, $\nu$ and $R$ (we have assumed that $L=1-R$), the quantities $\pi$, $\kappa$ and $\sigma$ will be denoted $\pi_{q, \nu, R}(\theta)$, $\kappa_{q, \nu, R}(\theta)$ and $\sigma_{q, \nu, R}(\theta)$.

\subsubsection{Computation of $\pi_{q, \nu, R}(\theta)$.}\label{subsubsecpi}
Equation \eqref{eq:generalpi} from our statement of the KPZ class conjecture implies that
\begin{equation*}
\pi_{q, \nu, R}(\theta) =
\frac{\partial j(\rho(\theta) )}{\partial \theta} \Big\slash \frac{\partial \rho(\theta) }{\partial \theta},
\end{equation*}
which yields the formula
\begin{multline}
\pi_{q, \nu, R}(\theta) = \frac{1-q}{\log(q)^2}
\left[\frac R \nu \left(\Psi_q'(\theta+V) + \Psi_q''(\theta+V)\frac{\log q + \Psi_q(\theta)-\Psi_q(\theta+V)}{\Psi_q'(\theta+V)-\Psi_q'(\theta)} \right)\right.\\ \left. - L \left(\Psi_q'(\theta) + \Psi_q''(\theta)\frac{\log q + \Psi_q(\theta)-\Psi_q(\theta+V)}{\Psi_q'(\theta+V)-\Psi_q'(\theta)} \right) \right].
\label{eq:expressionforpi}
\end{multline}

\subsubsection{Computation of $\kappa_{q, \nu, R}(\theta)$}\label{subsubsec:kappa}   Equation \eqref{eq:generalkappa} from our statement of the KPZ class conjecture implies that
\begin{equation*}
\kappa(\theta)  = - \rho(\theta)\  \pi(\theta) + j(\rho)(\theta).
\end{equation*}
This yields the formula
\begin{equation}
\kappa_{q, \nu, R}(\theta) = \frac{1-q}{\log(q)}  \frac{\frac{R}{\nu}\  \Psi_q''(\theta+V) - L\ \Psi_q''(\theta)}{\Psi_q'(\theta)-\Psi_q'(\theta+V)}.
\label{eq:expressionforkappa}
\end{equation}
In order to make sense physically, the quantity $\kappa_{q, \nu, R}(\theta) $ must be positive, at least for
$\theta$ belonging to some interval $(\tilde{\theta}, +\infty)$.
 Since $\kappa_{q, \nu, R}(\theta) $ tends to $R-L$ when $\theta$ tends to infinity (equivalently $\alpha\to 0$), this requires that $R>L$ and suggests that the particles lie on a support of size $\mathcal{O}(\text{time})$ with high probability only if $R >L$.

Now assume that $R>L>0$. Then $\kappa_{q, \nu, R}(\theta) $ tends to $-\infty$ when $\theta$ tends to $0$. The local behaviour of particles around the first particles is described by the stationary measure $\mu_{\alpha_0}$,  where  $\alpha_0=q^{\theta_0}$ is such that  $\kappa_{q, \nu, R}(\theta_0)= 0$. If $R>L>0$, then $0<\theta_0< \infty$, which means that the density of particles $\rho_0$ is strictly positive around the first particle. In other words, the steady-state drift $j(\rho)/\rho$ is not decreasing and admits a maximum for some $\rho_0 >0$. Hence the density profile exhibit a discontinuity at the first particle, see Figure \ref{fig:densityprofileintro}. (Note that the curved section in Figure \ref{fig:densityprofileintro} is the parametric curve $(\pi_{q, \nu}(\theta), \rho(\theta))$ for $\theta\in(\theta_0, +\infty)$ where $\theta_0$ is such that $\kappa_{q, \nu}(\theta_0)=0$. This density profile is proved as a consequence of Theorem \ref{thm:fluctuationsrarefactionfan} in the case $q=\nu$.) Figure \ref{fig:simulation2} provides an additional confirmation using simulation data.

The macroscopic position of the first particle is then given by
\begin{equation*}
\pi(\theta_0) = \frac{1-q}{(\log q)^2} \left(\frac R \nu  \Psi_q'(\theta_0+V) - L\ \Psi_q'(\theta_0) \right),
\end{equation*}
where $\theta_0 = \log_q(\alpha_0)$. Not surprisingly, it is also the drift of a tagged particle in an environment given by $\mu_{\alpha_0}$. This gives another explanation of why the density $\rho_0$ around the first particle is such that $\frac{\partial j(\rho_0)}{\partial \rho}  = \pi(\rho_0) = \frac{j(\rho_0)}{/rho_0}$, which implies that $\rho_0$ maximizes the drift.



\subsection{Magnitude of fluctuations}
\label{subsec:sigma}
One first needs to compute $\lambda = -j''(\rho)$. We have expressions for $j(\rho(\theta))$ and $\rho(\theta)$ but we take the second derivative of the function $j$ with respect to the variable $\rho$. We have that
\begin{multline*}
j''(\rho(\theta)) = \frac{1-q}{(\log q)^3} \frac{(\log q + \Psi_q(\theta) -\Psi_q(\theta+V))^3}{(\Psi_q'(\theta)-\Psi_q'(\theta+V))^2}\times   \\
\left( \frac{R}{\nu} \Psi_q'''(\theta+V) - L\Psi_q'''(\theta) - \left( \frac{R}{\nu} \Psi_q''(\theta+V) - L\Psi_q''(\theta)\right) \frac{\Psi_q''(\theta)-\Psi_q''(\theta+V)}{\Psi_q'(\theta)-\Psi_q'(\theta+V)}\right).
\end{multline*}
Note that the Lemma \ref{lem:thirdderivativegeneral} (proved in Section \ref{subsec:technicalproofs}), implies that $j''(\rho)\neq 0$ so that the main assumption of the KPZ class conjecture is satisfied.

In order to compute $A(\rho)$, we follow \cite{spohn2012kpz} and define
\begin{equation}
Z(\alpha) = \frac{(\alpha\nu ; q)_{\infty}}{(\alpha ; q)_{\infty}},
\label{eq:defZ}
\end{equation}
  the normalization constant in the definition of (\ref{eq:statmeasure}), and $G(\alpha) = \log(Z(\alpha))$. Then
  $$A = \frac{\alpha (\alpha G')'}{(1+\alpha G')^3},$$
  where all derivatives are taken with respect to the variable $\alpha$. (The formula differs by a factor $4$ with \cite{spohn2012kpz} because we take occupation variables $\eta_i\in\lbrace 0,1 \rbrace$ instead of $\lbrace -1,1 \rbrace$.) With  $Z$ as in (\ref{eq:defZ}), we have
  $$G'(\alpha) = \frac{1}{\alpha \log q}\left( \Psi_q(\theta) - \Psi_q(\theta+V)\right),$$
  and
\begin{equation}
A(\theta)=\log q \frac{\Psi_q'(\theta) - \Psi_q'(\theta+V)}{(\log q + \Psi_q(\theta) - \Psi_q(\theta+V))^3}.
\end{equation}
Finally, $\sigma_{q, \nu}(\theta) = \left( \frac{\lambda A^2}{2\rho^3}\right)^{1/3}$ with
\begin{equation}
\frac{\lambda A^2}{2\rho^3} = \frac{q-1}{4(\log q)^4} \left(  \frac{R}{\nu} \Psi_q'''(\theta+V) - L\Psi_q'''(\theta) - \left( \frac{R}{\nu} \Psi_q''(\theta+V) - L\Psi_q''(\theta)\right) \frac{\Psi_q''(\theta)-\Psi_q''(\theta+V)}{\Psi_q'(\theta)-\Psi_q'(\theta+V)}\right).
\label{eq:expressionforsigma}
\end{equation}
One should note that we have always $\sigma_{q, \nu}(\theta)>0$ (see Lemma \ref{lem:thirdderivativegeneral} for a proof of this claim).

\subsection{Critical point Fredholm determinant asymptotics}
\label{subsec:heuristicasymptotics}

We sketch an asymptotic analysis of the Mellin-Barnes Fredholm determinant formula of Theorem \ref{th:fredholmgeneral} that confirms the KPZ class conjecture for the $q$-Hahn AEP. In particular, we recover independently the functions $\pi_{q, \nu}(\theta)$, $\kappa_{q, \nu}(\theta)$ and $\sigma_{q, \nu}(\theta)$ from (\ref{eq:expressionforpi}), (\ref{eq:expressionforkappa}) and (\ref{eq:expressionforsigma}). We do not provide all necessary justifications to make this rigorous. However, in Section \ref{sec:MADM}, we do provide such rigorous justifications for the $\nu=q$ case under certain ranges of parameters.

The  function $x \mapsto 1/(-q^{x} ; q)_{\infty}$ converges to $1$ as $ x\to +\infty$ and $0$ as $x\to-\infty$. Hence the sequence of functions $\left(x \mapsto 1/(-q^{t^{1/3}x} ; q)_{\infty}\right)_{t>0}$ converges to a step function when $t\to\infty$. On account of this,  if we set
\begin{equation*}
\zeta = -q^{-\kappa t - \pi t - t^{1/3}\sigma x},
\end{equation*}
then it follows that for $\sigma >0$,
\begin{equation*}
\lim_{t\to\infty} \EE \left[\frac{1}{\left(\zeta q^{x_{n}(t)+n};q\right)_{\infty}}\right] = \lim_{t\to \infty} \mathbb{P}\left(\frac{x_n(t) - \pi t}{\sigma t^{1/3}} \geqslant x\right) ,
\end{equation*}
with $n=\lfloor \kappa t \rfloor$. Of course, we have omitted here to justify the exchange of limit, and we refer to Section \ref{sec:MADM} where a complete argument is provided.

For the moment, let the constants $\kappa, \pi$ and $\sigma$ remain undetermined.
$\EE \left[\frac{1}{\left(\zeta q^{x_{n}(t)+n};q\right)_{\infty}}\right]$ is given by $\det\left(I + K_{\zeta}\right)$ as in (\ref{eq:MellinBarnes}).
Assume for the moment that the contour $C_1$ for the variable $w$ is a very small circle around $1$. Let us make the change of variables 
$$ w=q^W, \ \ w'=q^{W'},\ \  s+W=Z.$$
Then the Fredholm determinant can be written with the new variables as $\det\left(I + K_{x}\right)$ where $K_x$ is an operator acting on $\mathbb{L}^2(C_0)$ where $C_0$ is the image of $C_1$ under the mapping $w\mapsto \log_q(w)$, defined by its kernel
\begin{multline}
K_{x}(W,W') =\frac{q^W \log q}{2\pi i} \\
\int_{\mathcal{D}_W} \frac{\pi}{\sin(-\pi(Z-W))} \exp\left(t\big(f_0(Z) - f_0(W)\big) - t^{1/3} \sigma x \log(q)(Z-W) \right) \frac{1}{q^Z - q^{W'}} \frac{(\nu q^Z ; q)_{\infty}}{(\nu q^W ; q)_{\infty}} dZ,
\label{eq:kernelexponentialformgeneral}
\end{multline}
where the new contour $\mathcal{D}_W$ as the straight line $W+1/2+i\R$, and the function $f_0$ is defined by
\begin{equation}
f_0(Z) = \kappa \log\left(\frac{(q^Z ; q)_{\infty}}{(\nu q^Z ; q)_{\infty}}\right) + \frac{1-q}{\log(q)} \left(\frac{R}{\nu}\Psi_q(Z+V)    - L \left(\Psi_q(Z) \right)\right) -Z \log(q) \left(\kappa + \pi\right).
\label{eq:deff0general}
\end{equation}
Since $C_1$ was any small enough circle around $1$, $C_0$ can be deformed to be a small circle around $0$, and we can also deform the contour for $Z$ to be simply $1/2+i\R$ without crossing any singularities.

The idea of Laplace's method is to deform the integration contours so that they go across a critical point of $f_0$, and then make a Taylor approximation around the critical point. Actually, we know that the Airy kernel would occur in the limit if this critical point is a double critical point, so we determine our unknown parameters $(\kappa, \pi, \sigma)$ so as to have a double critical point. We have that
\begin{equation}
f_0'(Z) = \kappa \left( \Psi_q(Z+V) -\Psi_q(Z)\right)+\frac{1-q}{\log(q)} \left(\frac{R}{\nu}\Psi_q'(Z+V)    - L \left(\Psi_q'(Z) \right)\right) - \log(q) \left(\kappa + \pi\right),
\end{equation}
and
\begin{equation}
f_0''(Z) = \kappa  \left( \Psi_q'(Z+V) -\Psi_q'(Z)\right)+\frac{1-q}{\log(q)} \left(\frac{R}{\nu}\Psi_q''(Z+V)    - L \left(\Psi_q''(Z) \right)\right).
\end{equation}
We see that if $\pi = \pi_{q, \nu}(\theta)$ and  $\kappa=\kappa_{q, \nu}(\theta)$ as in (\ref{eq:expressionforpi}) and (\ref{eq:expressionforkappa}), then $f_0'(\theta) =f_0''(\theta)=0$. Hence, up to higher order terms in $(Z-\theta)$,
$$ f_0(Z) - f_0(W) \approx \frac{f_0'''(\theta)}{6} \left((Z-\theta)^3 - (W-\theta)^3\right).$$
The next lemma, about the sign of $f_0'''$, is proved in Section \ref{subsec:technicalproofs}.
\begin{lemma}
For any $q\in (0,1)$, $\nu\in [0,1)$, and any $R,L\geqslant 0$ such that $R+L=1$, we have that for all $\theta>0$, $f_0'''(\theta)>0$.
\label{lem:thirdderivativegeneral}
\end{lemma}
Using Lemma \ref{lem:thirdderivativegeneral} we know the behaviour of $\Real[f_0]$ in the neighbourhood of $\theta$. To make Laplace's method rigorous, we must control the real part of $f_0$ along the contours for $Z$ and $W$, and prove that only the integration in the neighbourhood of $\theta$ has a contribution to the limit. We do not prove that here, and the rest of the asymptotic analysis presented in this section would require some additional effort to be completely rigorous.

Assume that one is able to deform the contours for $Z$ and $W$ passing through $\theta$ so that
\begin{itemize}
\item The contour for $Z$ departs $\theta$ with angles $\xi$ and $-\xi$ where $\xi\in (\pi/6, \pi/2)$, and $\Real[ f_0]$ attains  its maximum uniquely at $\theta$,
\item The contour for $W$ departs $\theta$ with angles $\omega$ and $-\omega$ where $\omega\in (\pi/2, 5\pi/6)$, and $\Real[f_0]$ attains its minimum uniquely at  $\theta$.
\end{itemize}
Then, modulo some estimates that we do not state explicitly here, the Fredholm determinant can be approximated by the following. We make the change of variables $Z-\theta= zt^{-1/3}$ and likewise for $W$ and $W'$. Taking into account the Jacobian of the $W$ and $W'$ change of variables, we get that the kernel has rescaled to
\begin{equation}
\tilde{K}_x(w,w') = \frac{1}{2i\pi} \int \frac{1}{w-z} \frac{1}{z-w'}\exp\left(\frac{f_0'''(\theta)}{2}\left(z^3/3 - w^3/3\right) - \sigma x \log(q)(z-w) \right)\mathrm{d}z.
\end{equation}
Finally, if we set $\sigma = \left( \frac{-f_0'''(\theta)}{2 (\log q)^3}\right)^{1/3}$, and we make the change of variables replacing  $  - z\sigma \log(q)$ by $z$ and likewise for $w$ and $w'$,  we get the kernel
\begin{equation}
\tilde{K}_x(w,w') = \frac{1}{2i\pi} \int_{\infty e^{-i\pi/3}}^{\infty e^{i\pi/3}} \frac{1}{w-z} \frac{1}{z-w'}e^{z^3/3 - w^3/3 + x (z-w)}\mathrm{d}z,
\end{equation}
acting on a contour coming from $\infty e^{-2i\pi/3}$ to $\infty e^{2i\pi/3}$ which does not intersect the contour for $z$. Let us call $  \mathcal{G}$ this contour. Using the ``$\det(I-AB) =  \det(I-BA)$ trick'' to reformulate Fredholm determinants, see e.g. Lemma 8.6 in \cite{borodin2012free}, one has that
$$\det(I+ \tilde{K}_x)_{\mathbb{L}^2(\mathcal{G})} = \det(I-K_{\rm Ai})_{\mathbb{L}^2(-x, +\infty)},$$
where $K_{\rm Ai}$  is the Airy kernel defined in \ref{def:TW}. Since $F_{\rm GUE}(x) = \det(I-K_{\rm Ai})_{\mathbb{L}^2(-x, +\infty)}$, we have that
$$\lim_{t\to\infty} \PP\left(\frac{x_n(t)-t\pi(\theta)}{\sigma(\theta)t^{1/3}}\geqslant x\right) \xrightarrow[t\to\infty]{} F_{GUE}(-x)$$
as claimed in (\ref{eq:KPZclassconjecture}).

The expression for $\sigma_{q, \nu}(\theta)$ in (\ref{eq:expressionforsigma})  is the same as $\sigma = \left( \frac{-f_0'''(\theta)}{2 (\log q)^3}\right)^{1/3}$. Indeed, we have that
\begin{equation}
f_0'''(Z) = \frac{1-q}{\log q} \left(  \frac{R}{\nu} \Psi_q'''(Z+V) - L\Psi_q'''(Z) - \left( \frac{R}{\nu} \Psi_q''(\theta+V) - L\Psi_q''(\theta)\right) \frac{\Psi_q''(Z)-\Psi_q''(Z+V)}{\Psi_q'(\theta)-\Psi_q'(\theta+V)}\right),
\label{eq:thirdderivativegeneral}
\end{equation}
so that $(\sigma_{q, \nu}(\theta))^3 = \frac{- f_0'''(\theta)}{2 (\log q)^3}.$

\section{Asymptotic analysis}
\label{sec:MADM}

In this section, we make the arguments of Section \ref{subsec:heuristicasymptotics} rigorous in the case $\nu=q$, which, in light of Section \ref{subsec:MADM} corresponds with the MADM. Consequently, we also provide a proof of Theorems \ref{thm:fluctuationsintro2} and \ref{thm:fluctuationsintro1} from the Introduction.  In order to simplify the notations we set $\pi(\theta) = \pi_{q, q, R}(\theta)$, $\kappa(\theta) = \kappa_{q, q, R}(\theta)$, and $\sigma(\theta) = \sigma_{q, q, R}(\theta)$, without writing explicitly the dependency on the parameters $q$ and $R$.

\begin{definition}
\label{def:TW}
The distribution function $F_{\rm GUE}(x)$ of the GUE Tracy-Widom distribution is defined  by $F_{\rm GUE}(x)=\det(I-K_{\rm Ai})_{\mathbb{L}^2(x,+\infty )}$ where $K_{\rm Ai}$ is the Airy kernel,
\begin{equation*}
K_{\rm Ai} (u, v) = \frac{1}{(2i\pi)^2} \int_{e^{-2i\pi/3}\infty}^{e^{2i\pi/3}\infty} \mathrm{d}w \int_{e^{-i\pi/3}\infty}^{e^{i\pi/3}\infty} \mathrm{d}z \frac{e^{z^3/3-zu}}{e^{w^3/3-wv}}\frac{1}{z-w},
\end{equation*}
where the contours for $z$ comes from infinity with an angle $-\pi/3$ and go to infinity with an angle $\pi/3$ ; the contour for  $w$ comes from infinity with an angle $-2\pi/3$ and go to infinity with an angle $2\pi/3$, and both contours do not intersect.
\end{definition}

\begin{theorem}
\label{thm:fluctuationsrarefactionfan}
Fix $q\in(0,1)$, $\nu=q$ and $R>L\geqslant 0$ with $R+L=1$. Let $\theta >0$ such that $\kappa\left(\theta\right) \geqslant 0$. Suppose additionally that $q^{\theta}>2q/(1+q)$. Then, for $n=\lfloor \kappa(\theta ) t\rfloor$, we have
$$ \lim_{t\to \infty} \mathbb{P}\left(\frac{x_n(t) - \pi(\theta)t}{\sigma(\theta) t^{1/3}} \geqslant x\right) = F_{\mathrm{GUE}}(-x).$$
\end{theorem}
\begin{remark}
In Figures \ref{fig:simulation1} and \ref{fig:simulation2}, one can see that the simulated curve is above the limiting curve predicted from KPZ scaling theory. This is coherent with the positive sign of $\sigma(\theta)$ (This is a consequence of Lemma \ref{lem:thirdderivativegeneral}, proved in Section \ref{subsec:technicalproofs}) and the fact that the Tracy-Widom distribution has negative mean.
\end{remark}

\begin{theorem}
\label{thm:fluctuationsparticle1}
Fix $q\in (0,1)$, $\nu=q$ and let
$$ R_{min}(q) = \frac{q\Psi_q''\left(\log_q\left(\frac{2q}{1+q} \right)\right)}{\Psi_q''\left(\log_q\left(\frac{2q}{1+q} \right)\right)+q\Psi_q''\left(\log_q\left(\frac{2q^2}{1+q} \right)\right)}\in \left(\frac{1}{2}, 1\right).$$
Then for $R_{min}(q)<R<1$ and $L=1-R$, there exists a real number $\theta_0>0$ such that $\kappa_{q, q,R}(\theta_0)=0$, and we have
$$ \lim_{t\to \infty} \mathbb{P}\left(\frac{x_1(t) - \pi(\theta_0)t}{\sigma(\theta_0) t^{1/3}} \geqslant x\right) = F_{\mathrm{GUE}}(-x).$$
\end{theorem}
\begin{remark}
We expect the same kind of result for the fluctuations of the position of the first particle in any $q$-Hahn AEP with positive asymmetry, when the parameter $\nu$ is such that $0<\nu<1$.
\end{remark}

\begin{remark}
The condition $q^{\theta}>2q/(1+q)$ in Theorem \ref{thm:fluctuationsrarefactionfan} is probably just technical. It ensures that we do not cross any residues when deforming the integration contour in the definition of the kernel $K_{\zeta}$ in Theorem \ref{th:fredholmgeneral} (see Remark \ref{rem:restrictivecondition}). The condition $ R_{min}(q)<R$ in Theorem \ref{thm:fluctuationsparticle1}  is equivalent to $q^{\theta}>2q/(1+q)$ in the particular setting of Theorem \ref{thm:fluctuationsparticle1}.

However, the condition $R<1$ is really meaningful, since in the totally asymmetric case ($R=1$), the first particle has Gaussian fluctuations.
\end{remark}

\begin{figure}[ht]
\begin{center}
\includegraphics[scale=0.4]{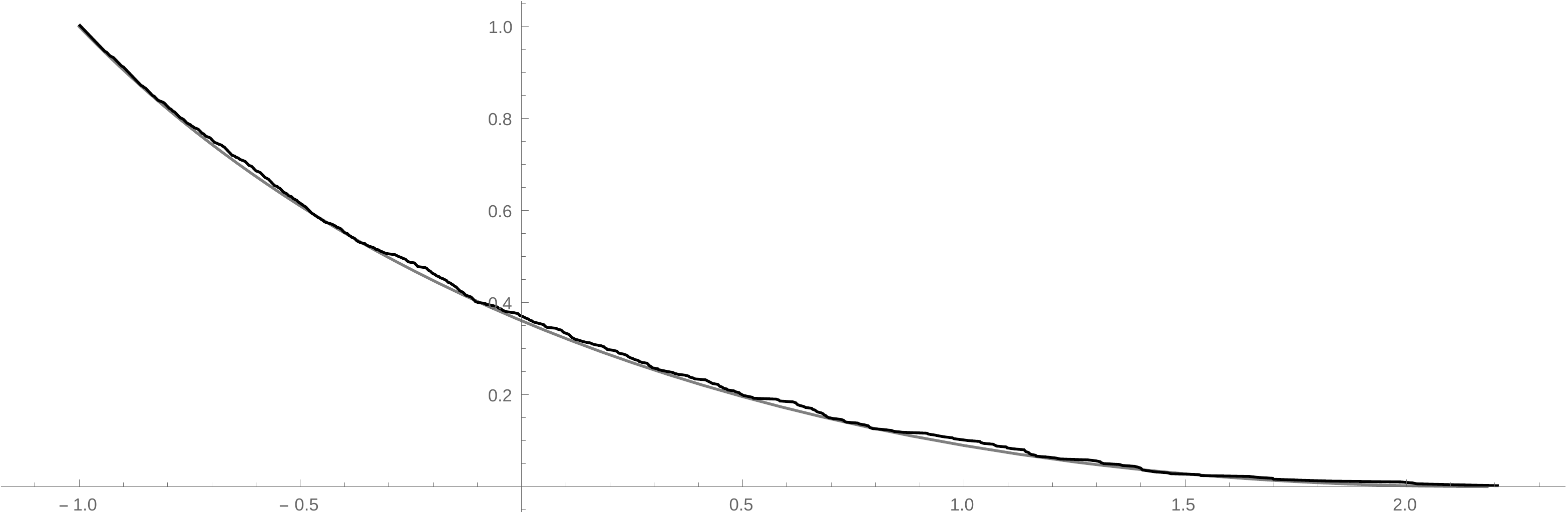}
\end{center}
\caption{Comparison between simulated numerical data and predicted hydrodynamic limit. The black curve is $(x_N(t)/t, N/t)_N$ for $N$ ranging from 1 to $t=500$ (which is fixed) in the totally asymmetric case ($R=1, L=0$), with $\nu=q=0.4$. This is also the graph of the function $x\mapsto N_{t x}(t)/t$, where by definition $N_x(t)$ is the number of particles right to $x$ at time $t$. The gray curve is the parametric curve $(\pi(\theta), \kappa(\theta))_{\theta\in(0,+\infty)}$ with $\pi(\theta)$ and $\kappa(\theta)$ as in (\ref{eq:expressionforpi}) and (\ref{eq:expressionforkappa}). }
\label{fig:simulation1}
\end{figure}

\begin{figure}[ht]
\begin{center}
\includegraphics[scale=0.4]{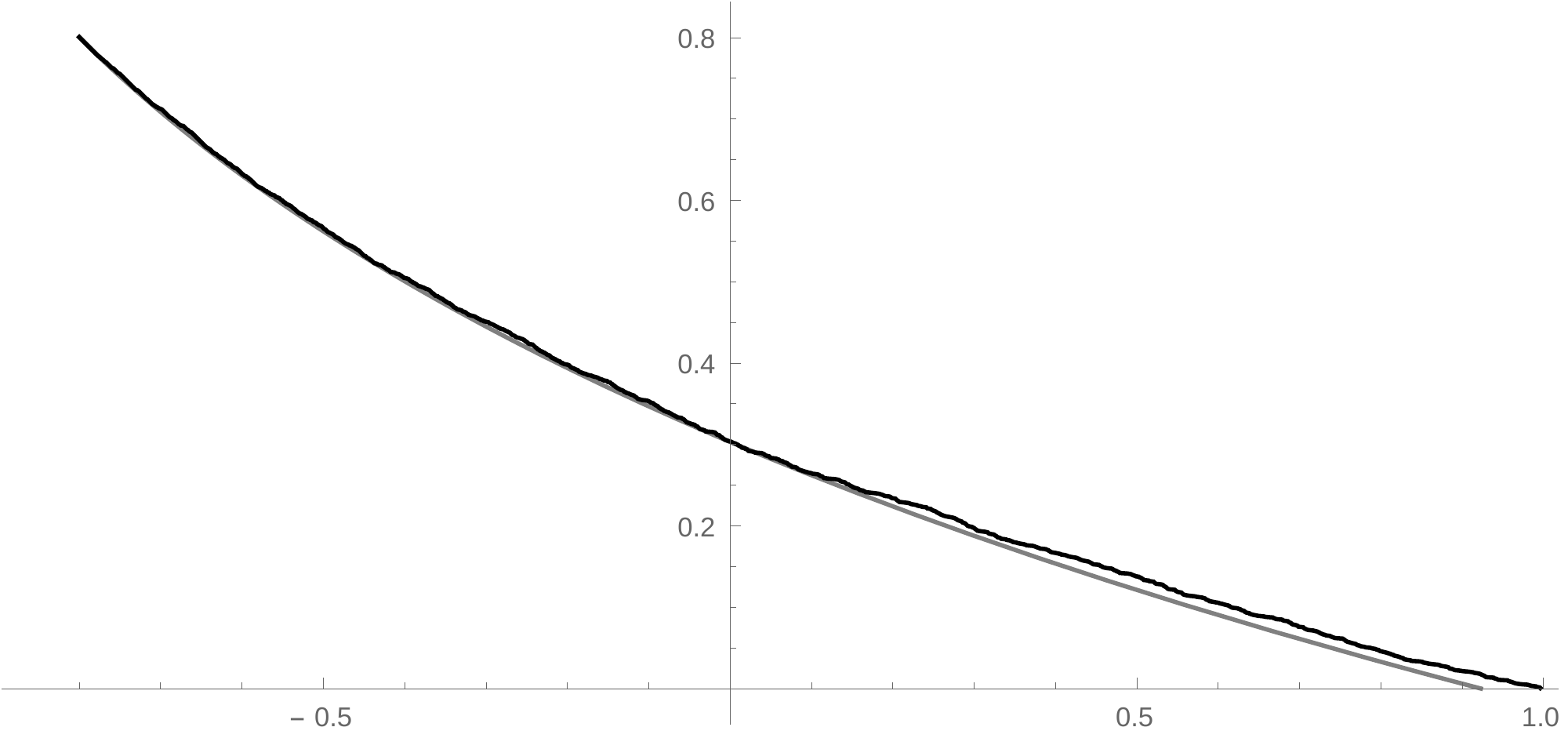}
\end{center}
\caption{The black curve is a simulation of $(x_N(t)/t, N/t)_N$ for $N$ ranging from 1 to $(R-L)t$, with $t=1500$ fixed, $R=0.9, L=.1$ and $\nu=q=0.6$. The gray curve is the parametric curve $(\pi(\theta), \kappa(\theta))_{\theta\in(\theta_0, +\infty)}$ where $\theta_0$ is such that $\kappa(\theta_0)$ as in Section \ref{subsubsec:kappa}. It goes from the point $(L-R,R-L)$ to the point $(\pi(\theta_0),0)$.
Since the slope of the curve $x\mapsto N_{t x}(t)/t$ (or equivalently $(x_N(t), N/t)_N$) is the macroscopic density $\rho(x, 1)$, this simulationally confirms the discontinuity of density at the point $\pi(\theta_0)$ (see Figure \ref{fig:densityprofileintro}).}
\label{fig:simulation2}
\end{figure}

\subsection{Proof of Theorem \ref{thm:fluctuationsrarefactionfan}}
The proof uses Laplace's method and follows the style of \cite{ferrari2013tracy} (similar proofs can be found in \cite{barraquand2014phase} for $q$-TASEP with slow particles, in \cite{borodin2012free} for the semi-discrete directed polymer, and in \cite{veto2014tracy} for the $q$-Hahn TASEP).

Fix $q\in(0,1)$, $\nu=q$, $R>L\geqslant 0$ with $R+L=1$ and $\theta >0$ such that $\kappa(\theta) \geqslant 0$. In the particular case $q=\nu$, Theorem \ref{th:fredholmgeneral} states that for all $\zeta\in\C\setminus \R_+$,
\begin{equation}
\EE \left[\frac{1}{\left(\zeta q^{x_{n}(t)+n};q\right)_{\infty}}\right] = \det\left(I + K_{\zeta}\right)
\end{equation}
where $\det\left(I + K_{\zeta}\right)$ is the Fredholm determinant of $K_\zeta: L^2(C_1)\to L^2(C_1)$ for $C_1$ a positively oriented circle containing $1$ with small enough radius so as to not contain $0$, $1/q$. The operator $K_\zeta$ is defined in terms of its integral kernel
\begin{equation}
K_{\zeta}(w,w') = \frac{1}{2\pi i} \int_{-i \infty + 1/2}^{i\infty +1/2} \frac{\pi}{\sin(-\pi s)} (-\zeta)^s \frac{g(w)}{g(q^s w)} \frac{1}{q^s w - w'} ds
\label{eq:kernelMADM}
\end{equation}
with
\begin{multline*}
g(w) = \\ \left(\frac{1}{1-w}\right)^{n}
 \exp\left(\frac{(q-1)t}{\log(q)} \left(\frac{R}{q}\big(\Psi_q(W+1)+\log(1-q) \big) - L \big(\Psi_q(W)+\log(1-q) \big)   \right)\right)
 \frac{1}{(q w;q)_{\infty}},
\end{multline*}
where $W=\log_q(w)$.
\begin{remark}\label{rem:ratiosimpler}
One notices that the argument of the exponential simplifies to $t\frac{(1-q)}{1+q} \frac{w}{1-w} $ when $R/L=q$.  This yields a simpler analysis, though we work with the general $R,L$ case here.
\end{remark}

In order to compute the probability distribution of $\frac{x_n(t) - \pi(\theta)t}{\sigma(\theta) t^{1/3}}$ from our $e_q$-Laplace transform formula, we use
\begin{lemma}[Lemma 4.1.39 \cite{borodin2014macdonald} ]\label{lem:Macdonald4.1.39}
Consider a sequence of functions $\{f_t\}_{t\geqslant 1}$ mapping $\R\to [0,1]$ such that for each $n$, $f_t(x)$ is strictly decreasing in $x$ with a limit of $1$ as $x\to-\infty$ and $0$ as $x\to +\infty$, and for each $\delta>0$, on $\R\setminus[-\delta,\delta]$ $f_t$ converges uniformly to $\mathds{1}_{\lbrace x\leqslant 0 \rbrace}$ as $t\to\infty$. Define the $r$-shift of $f_t$ as $f^r_t(x) = f_t(x-r)$. Consider a sequence of random variables $X_t$ such that for each $r\in \R$,
$$\EE[f^r_t(X_n)] \to p(r) $$
and assume that $p(r)$ is a continuous probability distribution function. Then $X_n$ converges weakly in distribution to a random variable $X$ which is distributed according to $\PP(X\leqslant  r) = p(r)$.
\end{lemma}
The sequence of functions $\left(f_t(x): x \mapsto 1/(-q^{-t^{1/3}x} ; q)_{\infty}\right)_{t>0}$
 satisfies the hypotheses of lemma \ref{lem:Macdonald4.1.39}. Hence, if we set
\begin{equation*}
\zeta = -q^{-\kappa(\theta) t - \pi(\theta) t -t^{1/3}\sigma(\theta) x},
\end{equation*}
and prove that $\EE \left[\frac{1}{\left(\zeta q^{x_{n}(t)+n};q\right)_{\infty}}\right]$ converges to the Tracy-Widom distribution (which is continuous), then it will imply that
\begin{equation*}
\lim_{t\to\infty} \EE \left[\frac{1}{\left(\zeta q^{x_{n}(t)+n};q\right)_{\infty}}\right] = \lim_{t\to \infty} \mathbb{P}\left(\frac{x_n(t) - \pi(\theta)t}{\sigma(\theta) t^{1/3}} \geqslant x\right)  = F_{GUE}(-x),
\end{equation*}
with $n=\lfloor \kappa(\theta) t\rfloor$.

Following the path described in Section \ref{subsec:heuristicasymptotics}, we make the change of variables:
$$ w=q^W, \ \ w'=q^{W'},\ \  s+W=Z.$$
The Fredholm determinant $\det\left(I + K_{\zeta}\right)$ equals $\det\left(I + K_{x}\right)$ where $K_x$ is an operator acting on $\mathbb{L}^2(C_0)$ where $C_0$ is a small circle around $0$, defined by its kernel
\begin{multline}
K_{x}(W,W') =\frac{q^W \log q}{2\pi i}
\int_{\mathcal{D}} \frac{\pi}{\sin(-\pi(Z-W))} \\ \times \exp\left(t\big(f_0(Z) - f_0(W)\big) - t^{1/3} \sigma(\theta) \log(q) x (Z-W) \right) \frac{1}{q^Z - q^{W'}} \frac{(q^{Z+1} ; q)_{\infty}}{(q^{W+1} ; q)_{\infty}} dZ,
\label{eq:kernelexponentialform}
\end{multline}
where the new contour $\mathcal{D}$ is the straight line $1/2+i\R $, and the function $f_0$ is defined by
\begin{equation}
f_0(Z) = \kappa(\theta) \log(1-q^Z) + \frac{1-q}{\log(q)} \left( \frac{R}{q} \Psi_q(Z+1) - L \Psi_q(Z) \right)
 -Z \log(q) \Big(\kappa(\theta ) + \pi(\theta)\Big).
\label{eq:deff0}
\end{equation}
Using the expressions (\ref{eq:expressionforkappa}) and (\ref{eq:expressionforpi})  for $\kappa(\theta)$ and $\pi(\theta)$ in terms of the $q$-digamma function, we have
\begin{multline*}
f_0(Z) = \frac{1-q}{\log(q)} \bigg(  \frac{R}{q}\Big[  \Psi_q(Z+1)+\log(1-q) - Z\Psi_q'(\theta+1)  \\
 + \frac{\Psi_q''(\theta+1)}{\log q} \left( \frac{(1-\alpha)^2}{\alpha}\frac{\log(1-q^Z)}{\log(q)} + Z (1-\alpha)\right) \Big] \\
 -L \Big[   \Psi_q(Z)+\log(1-q) - Z\Psi_q'(\theta) + \frac{\Psi_q''(\theta)}{\log q} \left( \frac{(1-\alpha)^2}{\alpha}\frac{\log(1-q^Z)}{\log(q)} + Z (1-\alpha)\right) \Big]  \bigg),
\end{multline*}
with $\alpha=q^{\theta}$. For the derivatives, we have
\begin{multline}
f_0'(Z) =  \frac{1-q}{\log(q)} \frac{R}{q} \left[  \Psi_q'(Z+1)-\Psi_q'(\theta+1) + \frac{\Psi_q''(\theta+1)}{\log(q)} \left(  (1-\alpha)-\frac{(1-\alpha)^2}{\alpha}\frac{q^Z}{1-q^Z} \right)\right]\\
- \frac{1-q}{\log(q)} L  \left[  \Psi_q'(Z)-\Psi_q'(\theta) + \frac{\Psi_q''(\theta)}{\log(q)} \left(  (1-\alpha)-\frac{(1-\alpha)^2}{\alpha}\frac{q^Z}{1-q^Z} \right)\right],
\label{eq:f0primegeneral}
\end{multline}
\begin{multline*}
f_0''(Z) = \frac{1-q}{\log(q)} \frac{R}{q} \left[  \Psi_q''(Z+1) -\frac{q^Z}{(1-q^Z)^2} \frac{(1-\alpha)^2}{\alpha} \Psi_q''(\theta+1) \right]\\
- \frac{1-q}{\log(q)} L \left[  \Psi_q''(Z) -\frac{q^Z}{(1-q^Z)^2} \frac{(1-\alpha)^2}{\alpha} \Psi_q''(\theta) \right].
\end{multline*}
Notice that the formulas become much simpler in the special case of Remark \ref{rem:ratiosimpler}. Using the fact that $\Psi_q'(Z) - \Psi_q'(Z+1) = \log(q)^2 \frac{q^Z}{(1-q^Z)^2}$, one has
\begin{equation}
f_0'(Z) = \frac{(1-q) \log(q)}{(1+q)(1-\alpha)^2} \left(\frac{q^Z}{1-q^Z} \left( 1-\alpha^2 - \frac{(1-\alpha)^2}{1-q^Z}\right) -\alpha^2\right).
\label{eq:simplecasef0prime}
\end{equation}

One readily verifies that $f_0'(\theta)=f_0''(\theta)=0$.
Since the saddle-point is at $\theta$, we need to deform the integration contours for the variables $Z$ and $W$ so that they pass through $\theta$ and control the real part of $f_0$ along these contours. Let $\mathcal{C}_{\alpha}$ be the positively oriented contour enclosing $0$ defined by its parametrization
\begin{equation}
W(u) :=\log_q( 1-(1-\alpha)e^{iu})
\label{eq:defparametrization}
\end{equation}
 for $u\in(-\pi, \pi)$. Hence $q^{W(u)}$ ranges in a circle of radius $(1-\alpha)$ centered at $1$ (see Figure \ref{fig:contoursmallvariables}). In order to use $\mathcal{C}_{\alpha}$ as the contour for $W$ in the definition of the Fredholm determinant $\det(I+K_{x})$, one should not encounter any singularities of the kernel when deforming the contour. Hence $\mathcal{C}_{\alpha}$ should not enclose $-1$ (this is the equivalent with the fact that the contour $C_1$ in Theorem \ref{th:fredholmgeneral} must not enclose $1/q$.) For the rest of this section, we impose the condition
\begin{equation}\label{eq:twoalphaq}
2-\alpha < 1/q,
\end{equation}
 so that our contour deformation is valid.

\begin{figure}
\begin{center}
\begin{tikzpicture}[scale=2]
\draw[->, thick] (-1, 0) -- (3,0);
\draw[->, thick] (0,-1) -- (0,1);
\draw[thick] (0,0) circle(0.6);
\draw[thick] (1,0) circle(0.4);
\draw (1,-0.05) -- (1,0.05) node[above] {$1$};
\draw (0.5,-0.1) node{$\alpha$};
\draw (2,-0.05) -- (2,0.05) node[above] {$1/q$};
\fill (120:0.6) circle(0.03) node[above left]{$z=q^Z$};
\fill (1,0.4) circle(0.03) node[above]{$w=q^W $};
\draw[dashed] (0,0) -- (1,0.4);
\fill[gray] (0.5,0.2) circle(0.03) node[left]{$qw$};
\end{tikzpicture}
\end{center}
\caption{Images of the contours $\mathcal{C}_{\alpha}$ and $\mathcal{D}_{\alpha}$ by the map $Z\mapsto q^Z$. The condition $\alpha >2q/(1+q)$ is such that $qw$ is always inside the image of $\mathcal{D}_{\alpha}$, which is the case in the figure. }
\label{fig:contoursmallvariables}
\end{figure}
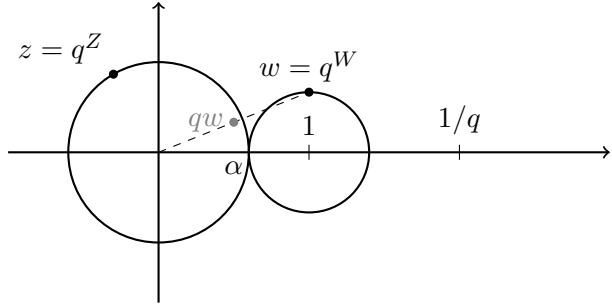

When deforming the contour for the variable $W$, one also have to deform the contour for the variable $Z$, since in the original definition of $K_{\zeta}$ in Equation (\ref{eq:kernelMADM}), the only singularities of the integrand for the variable $s$ are for $s\in\Z$. This means that the singularities at $W+1, W+2, \dots$ for the variable $Z$ must be on the right of the contour for $Z$.
Let us choose the contour $\mathcal{D}_{\alpha}$ being the straight line parametrized by $Z(u) := \theta+iu$ for $u$ in $\mathbb{R}$.
To ensure that $\Real[W+1] >\theta$, or equivalently that $\vert qw \vert <\alpha $ (see Figure \ref{fig:contoursmallvariables}), we impose the condition that
\begin{equation}
\alpha >\frac{2q}{1+q}.
\label{eq:restrictivecondition}
\end{equation}
Condition (\ref{eq:restrictivecondition}) implies in particular the previous condition $2-\alpha < 1/q$.
\begin{remark}
Condition (\ref{eq:restrictivecondition}) is  the same as condition (2.15) in \cite{veto2014tracy}. To get rid of this condition, one would need to add small circles around each pole in $W+1, W+2, \dots$ in the definition of the contour $\mathcal{D}$, as in \cite{ferrari2013tracy}. The rest of the asymptotic analysis would remain almost unchanged provided one is able to prove that for any $W\in \mathcal{C}_{\alpha}$ and $k\geqslant 1$ such that $\vert q^{W+k}\vert>\alpha$, $\Real [ f_0(W)- f_0(W+k) ] >0$.
In our case, it appears that the analysis of $\Real [ f_0(W)- f_0(W+k) ]$ is computationally difficult and we do not pursue that here.
\label{rem:restrictivecondition}
\end{remark}

One notices that $\Real[f_0]$ is periodic with a period $i\frac{2\pi}{\log q}$. Moreover, $f_0(\overline{Z}) = \overline{f_0(Z)}$ so that  $\Real[f_0]$ is determined by its restriction on the domain $\R+i[0, -\pi/\log q]$.
The following results about the behaviour of $\Real[f_0]$ along the contours are proved in Section \ref{subsec:technicalproofs}.
\begin{lemma}
For any $R >L\geqslant 0$ with $R+L=1$, we have $f_0'''(\theta)>0$.
\label{lem:thirdderivative}
\end{lemma}
\begin{proof}
This is a particular case ($\nu=q$) of Lemma \ref{lem:thirdderivativegeneral}, which we prove in Section \ref{sec:sixtwo}.
\end{proof}

\begin{proposition}
Assume that (\ref{eq:twoalphaq}) holds. For any $R >L\geqslant 0$ with $R+L=1$, the contour $\mathcal{C}_{\alpha}$ is steep-descent for the function $-\Real [f_0]$ in the following sense: the function $u\mapsto \Real [f_0(W(u))]$ is increasing for $u\in[0, \pi]$ and decreasing for $u\in [-\pi, 0]$.
\label{prop:steepdescentC}
\end{proposition}

\begin{proposition}
Assume that (\ref{eq:twoalphaq}) holds. For any $R >L\geqslant 0$ with $R+L=1$, the contour $\mathcal{D}_{\alpha}$ is steep-descent for the function $\Real [f_0]$ in the following sense: the function $t\mapsto \Real [f_0(Z(u))]$ is decreasing for $u\in[0, -\pi/ \log q]$ and increasing for $u\in [\pi/\log q, 0]$.
\label{prop:steepdescentD}
\end{proposition}

We are now able to prove that asymptotically, the contribution to the Fredholm determinant of the contours are negligible outside a neighbourhood of $\theta$.

\begin{proposition}
For any fixed $\delta >0$ and $\epsilon>0$, there exists a real $t_0$ such that for all $t>t_0$
$$ \big\vert \det(I+K_x)_{\mathbb{L}^2(\mathcal{C}_{\alpha}) } - \det(I+K_{x,\delta})_{\mathbb{L}^2(\mathcal{C}_{\alpha, \delta})}\big\vert <\epsilon$$
\label{prop:kernellocalization}
where $\mathcal{C}_{\alpha, \delta}$ is the intersection of $\mathcal{C}_{\alpha}$ with the ball $B(\theta,\delta)$ of radius $\delta$ around $\theta$, and
\begin{multline*}
K_{x, \delta}(W,W') =
\frac{q^W \log q}{2\pi i}\int_{\mathcal{D}_{\delta}} \frac{\pi}{\sin(-\pi(Z-W))} \\ \times \exp\left(t\big(f_0(Z) - f_0(W)\big) - t^{1/3} \sigma(\theta) \log(q) x (Z-W) \right) \frac{1}{q^Z - q^{W'}} \frac{(q^{Z+1} ; q)_{\infty}}{(q^{W+1} ; q)_{\infty}}dZ,
\end{multline*}
where $\mathcal{D}_{\delta} = \mathcal{D}\cap B(\theta,\delta)$.
\end{proposition}
\begin{proof}
We have the Fredholm determinant expansion
\begin{equation}
\det(I+K_x)_{\mathbb{L}^2(\mathcal{C}_{\alpha})} =\sum_{k=0}^{\infty} \frac{1}{k!} \int_{-\pi}^{\pi}\mathrm{d}s_1 \dots \int_{-\pi}^{\pi} \mathrm{d}s_k \det\big( K_x(W(s_i), W(s_j))\big)_{i,j=1}^k\frac{\mathrm{d}W(s_i)}{\mathrm{d} s_i},
\label{eq:fredholmexpansion}
\end{equation}
with $W(s)$ as in (\ref{eq:defparametrization}).
Let us denote by $s_{\delta}$ the positive real number such that $\vert W(s_{\delta}) - \theta \vert =\delta$. We need to prove that one can replace all the integrations on $[-\pi, \pi]$ by integrations on $[-s_{\delta}, s_{\delta}]$ , making a negligible error. By Propositions \ref{prop:steepdescentC} and \ref{prop:steepdescentD}, we can find a constant $c_{\delta}>0$ such that for $\vert s\vert >s_{\delta}$ and for any $Z\in \mathcal{D}_{\alpha}$,
$$\Real\big[f_0(Z) - f_0(W(s))\big]<-c_{\delta}.$$
The integral in (\ref{eq:kernelexponentialform}) is absolutely integrable due to the exponential decay of the sine in the denominator. Thus, one can find a constant $C_{\delta}$ such that for $\vert s\vert >s_{\delta}$, any $W'\in \mathcal{C}_{\alpha}$  and $t$ large enough,
$$\big\vert K(W(s), W')\big\vert< C_{\delta} \exp(-t c_{\delta}/2).$$
By dominated convergence the error (that is the expansion (\ref{eq:fredholmexpansion}) with integration on $\big[-\pi , \pi\big]^k\setminus \big[-s_{\delta}, s_{\delta}\big]^k $) goes to zero for $t$ going to infinity.

We also have to prove that one can localize the $Z$ integrals as well. Recall that $\Real[f_0]$ is periodic on the contour $\mathcal{D}_{\alpha}$.  By the steep-descent property of Proposition \ref{prop:steepdescentD} and the same kind of dominated convergence arguments, one can localize the integrations on
$$\bigcup_{k\in \Z} I_k, \text{ where } I_k=\big[\theta - i\delta +i2k\pi/\log q, \theta + i\delta +i2k\pi/\log q\big],$$
making a negligible error. Since $f_0(Z) - f_0(\theta) \approx \frac{f_0'''(\theta)}{6} (Z-\theta)^3$, by making the change of variables $Z=\theta+ i2\pi k/\log q + zt^{-1/3}$, we see that only the integral for $Z\in [\theta-i\delta, \theta+i\delta]$ contributes to the limit. Indeed, for $k\neq 0$, and $Z\in I_k$
$$ \frac{\mathrm{d}Z}{\sin(\pi(Z-W))} \approx t^{-1/3} \exp\left(-\vert2\pi^2 k/\log(q) \vert \right).$$
Hence the sum of  contributions of integrals over $I_k$ for $k\neq 0$ is $\mathcal{O}(t^{-1/3})$ and one can finally integrate over $\mathcal{D}_{W,\delta} $ making an error going to $0$ as $t\to\infty$. It is not enough to show that the error made on the kernel goes to zero as $t$ goes to infinity, but one can justify that the error on the Fredholm determinant goes to zero as well by a dominated convergence argument on the expansion (\ref{eq:fredholmexpansion}).
\end{proof}
By the Cauchy theorem, one can replace the contours $\mathcal{D}_{\delta}$ and $\mathcal{C}_{\alpha, \delta}$ by wedge-shaped contours $\hat{D}_{\varphi, \delta}:=\lbrace \theta+\delta e^{i \varphi sgn(y)} \vert y\vert ;y\in[-1,1] \rbrace$ and $\hat{C}_{\psi, \delta}:=\lbrace \theta+\delta e^{i (\pi - \psi) sgn(y)} \vert y\vert ;y\in[-1,1] \rbrace$, where the angles $\varphi, \psi\in(\pi/6, \pi/2)$ are chosen so that the endpoints of the contours do not change.

Let us make the change of variables
$$ Z=\theta+\tilde{z}t^{-1/3}, \ W=\theta+\tilde{w}t^{-1/3}, \ W' = \theta+\tilde{w}'t^{-1/3}.$$
We define the corresponding rescaled contours
$$\mathcal{D}_{\varphi}^{L}:=\lbrace L e^{i \varphi sgn(y)} \vert y\vert ;y\in[-1,1] \rbrace,$$
$$\mathcal{C}_{\psi}^{ L}:=\lbrace L e^{i (\pi - \psi) sgn(y)} \vert y\vert ;y\in[-1,1] \rbrace.$$

\begin{proposition}
We have the convergence
$$ \lim_{t\to\infty} \det(I+K_x)_{\mathbb{L}^2(\mathcal{C}_{\alpha})}  = \det(I+K'_{x,\infty})_{\mathbb{L}^2(\mathcal{C}_{\psi}^{ \infty})},$$
where for $L\in \R_+ \cup \lbrace \infty\rbrace$,

$$ K'_{x,L} = \frac{1}{2i\pi} \int_{\mathcal{D}_{\varphi}^{L}} \frac{\mathrm{d}\tilde{z}}{(\tilde{z}-\tilde{w}')(\tilde{w}-\tilde{z})} \frac{\exp\big((-\tilde{z} \sigma(\theta) \log{q})^3/3 -x\tilde{z} \sigma(\theta) \log{q} \big)}{\exp\big( (-\tilde{w} \sigma(\theta) \log{q})^3/3 -x \tilde{w} \sigma(\theta) \log{q} \big)}.$$
\end{proposition}
\begin{proof}
By the change of variables and the discussion about contours above,
$$\det(I+K_{x,\delta})_{\mathbb{L}^2(\mathcal{C}_{\alpha, \delta})} = \det(I+K_{x,\delta}^t)_{\mathbb{L}^2(\mathcal{C}_{\psi}^{\delta t^{1/3}})} $$
where $K_{x,\delta}^t$ is the rescaled kernel
$$ K_{x, \delta}^t(\tilde{w},\tilde{w}' ) = t^{-1/3} K_{x, \delta }(\theta +\tilde{w}t^{-1/3}, \theta+\tilde{w}'t^{-1/3}),$$
where we use the contours $ \mathcal{D}_{\varphi}^{ \delta t^{1/3}}$ for the integration with respect to the variable $\tilde{z}$.

Let us estimate the error that we make by replacing $f_0$ by its Taylor approximation. We recall that with our definition of $\sigma(\theta)$ in (\ref{eq:expressionforsigma}),
$$ f_0'''(\theta) = -2 \left(\sigma(\theta)\log(q) \right)^3.$$
Using Taylor expansion, there exists $C_{f_0}$ such that
$$\vert f_0(Z) -f_0(\theta) +\left(\sigma(\theta)\log(q) (Z-\theta) \right)^3/3 \vert <  C_{f_0} \vert Z-\theta \vert ^4,$$
for $Z$ in a fixed neighbourhood of $\theta$ (say e.g. $\vert Z-\theta\vert <\theta$).
Hence for $Z=\theta+\tilde{z}t^{-1/3}, \ W=\theta+\tilde{w}t^{-1/3}$,
\begin{multline}
\Big\vert t\big(f_0(Z ) - f_0(W)\big) - \big( (-\sigma(\theta) \log(q) \tilde{z})^3/3 -  (-\sigma(\theta) \log(q) \tilde{w})^3/3 \big)\Big\vert < \\ t^{-1/3 }C_{f_0} \left(\vert \tilde{z} \vert^4 + \vert \tilde{w}\vert^4\right) \leqslant \delta \left( \vert \tilde{z} \vert^3 + \vert \tilde{w}\vert^3 \right).
\label{eq:estimatedifff0}
\end{multline}
To control the other factors in the integrand, let
$$ F(Z,W,W'):= \frac{t^{-1/3}q^W \log(q)}{q^Z-q^{W'}} \frac{\pi t^{-1/3}}{\sin(\pi(Z-W))} \frac{(q^{Z+1} ; q)_{\infty}}{(q^{W+1} ; q)_{\infty}}.$$
we have that
$$ F(Z,W,W') \xrightarrow[t\to\infty]{} F^{lim}(\tilde{z},\tilde{w},\tilde{w}') := \frac{1}{\tilde{z}-\tilde{w}'}\frac{1}{\tilde{z}-\tilde{w}}.$$
\begin{lemma}
For $\tilde{z}\in \mathcal{D}_{\varphi}^{\delta t^{1/3}}$, and $\tilde{w}, \tilde{w}' \in \mathcal{C}_{\psi}^{\delta t^{1/3}}$, with $Z=\theta+\tilde{z}t^{-1/3},  W=\theta+\tilde{w}t^{-1/3}$ and $W' = \theta+\tilde{w}'t^{-1/3}$, we have that
$$ \vert  F(Z,W,W') -F^{lim}(\tilde{z},\tilde{w},\tilde{w}')\vert < C t^{-1/3} P(\vert\tilde{z} \vert, \vert\tilde{w}\vert  ,  \vert \tilde{w}'\vert  )F^{lim}(\tilde{z},\tilde{w},\tilde{w}'), $$
where and  $P$ is a polynomial and $C$ is a constant independent of $t$ and $\delta$, as soon as $\delta$ belongs to some fixed  neighbourhood of $0$.
\end{lemma}
\begin{proof}
Since $\vert Z-\theta\vert<\delta , \vert W-\theta\vert <\delta$ and $\vert W'-\theta\vert<\delta$, there exist constants $C_1, C_2$ and $C_3$ such that
$$ \left|\frac{q^W \log(q) (Z-W')}{q^Z-q^{W'}} -1 \right| \leqslant C_1(\vert Z-\theta \vert + \vert W'-\theta \vert),$$
$$ \left|\frac{\pi (Z-W)}{\sin(\pi(Z-W))} -1 \right| \leqslant C_2 (\vert Z-\theta \vert + \vert W-\theta \vert) , $$
$$ \left|\frac{(q^{Z+1} ; q)_{\infty}}{(q^{W+1} ; q)_{\infty}} -1 \right| \leqslant C_3 (\vert Z-\theta \vert + \vert W-\theta \vert).$$
Hence there exists a constant $C$ and a polynomial $P$ of degree $3$ such that
 $$ \left| \frac{F(Z,W,W')}{F^{lim}(\tilde{z},\tilde{w},\tilde{w}')} -1 \right| \leqslant C t^{-1/3} P(\vert\tilde{z} \vert , \vert\tilde{w}\vert , \vert \tilde{w}'\vert  ), $$
 and the result follows.
\end{proof}
Now we estimate the difference between the kernels $K_{x,\delta}^t$ and $K'_{x,\delta t^{1/3}}$. Let
$$f(Z,W,W') = t\big(f_0(Z ) - f_0(W)\big) - t^{1/3 } \sigma(\theta)\log(q)  x(Z-W)$$
 and
$$f^{lim}(\tilde{z},\tilde{w},\tilde{w}')=  \left((-\tilde{z} \sigma(\theta) \log{q})^3/3 -x\tilde{z} \sigma(\theta) \log{q} \right) - \left( (-\tilde{w} \sigma(\theta) \log{q})^3/3 -x \tilde{w} \sigma(\theta) \log{q} \right). $$
The difference between the kernels is estimated by
\begin{multline}
\big\vert K_{x,\delta}^t(\tilde{w},\tilde{w}') - K'_{x,\delta t^{1/3}}(\tilde{w},\tilde{w}')\big\vert < \int_{\mathcal{D}_{\varphi}^{ \delta t^{1/3}}} \mathrm{d}\tilde{z} \exp(f^{lim}) \vert F \vert \, \cdot\, \vert \exp(f-f^{lim})-1\vert \\
+ \int_{\mathcal{D}_{\varphi}^{ \delta t^{1/3}}} \mathrm{d}\tilde{z} \exp(f^{lim}) \vert F-F^{lim } \vert,
\label{eq:estimatediffkernels}
\end{multline}
where we have omitted the arguments of the functions $f(Z,W,W')$, $f^{lim}(\tilde{z},\tilde{w},\tilde{w}')$, $F(Z,W,W')$ and $F^{lim}(\tilde{z},\tilde{w},\tilde{w}') $.

Using the inequality $ \vert \exp(x)-1\vert <\vert x\vert \exp(\vert x\vert) $ and (\ref{eq:estimatedifff0}), we have
$$ \big\vert \exp(f-f^{lim})-1\big\vert < t^{-1/3} C_{f_0} \left( \vert \tilde{z}\vert^4 + \vert\tilde{w} \vert^4\right) \exp\left( \delta \left( \vert \tilde{z} \vert^3 + \vert \tilde{w}\vert^3 \right)\right).$$
 Hence, for $\delta $ small enough, the first integral in the right-hand-side of (\ref{eq:estimatediffkernels}) have cubic exponential decay in $\vert \tilde{z}\vert$, and the limit when $t\to\infty$ is zero by dominated convergence.  The second integral goes to zero as well by the same argument. We have shown pointwise convergence of the kernels. In order to show that the Fredholm determinants also converge, we give a dominated convergence argument.
The estimate (\ref{eq:estimatedifff0}) also shows that for $\delta $ small enough, one can bound the kernel $K_{x, \delta}^t$ by
$$ \vert K_{x, \delta}^t (\tilde{w},\tilde{w}') \vert < C \exp\left(\Real[(\sigma(\theta)\log(q) \tilde{w}^3)]/6\right)$$
for some constant $C$.
Then, Hadamard's bound yields
$$ \det\left( K_{x, \delta}^t(\tilde{w}_i, \tilde{w}_j)\right)_{i,j=1}^n \leqslant n^{n/2}C^n  \prod_{i=1}^n \exp\left(\Real[\sigma(\theta)\log(q) \tilde{w}_i^3]/6\right).$$
It follows that the Fredholm determinant expansion
$$ \det(I+K_{x, \delta }^t)_{\mathbb{L}^2(\mathcal{C}_{\psi}^{ \delta t^{1/3}})} = \sum_{n=0}^{\infty} \frac{1}{n!} \int_{\mathcal{C}_{\psi}^{ \delta t^{1/3}}}\mathrm{d}\tilde{w}_1 \dots \int_{\mathcal{C}_{\psi}^{ \delta t^{1/3}}} \mathrm{d}\tilde{w}_n  \det\left( K_{x, \delta}^t(\tilde{w}_i, \tilde{w}_j)\right)_{i,j=1}^n,$$
is absolutely integrable and summable. Thus, by dominated convergence
\begin{align*}
 \lim_{t\to\infty} \det(I+K_x)_{\mathbb{L}^2(\mathcal{C}_{\alpha})} =& \lim_{t\to\infty} \det(I+K'_{x,\delta t^{1/3}})_{\mathbb{L}^2(\mathcal{C}_{\psi}^{ \delta t^{1/3}})}\\
 =&\det(I+K'_{x,\infty})_{\mathbb{L}^2(\mathcal{C}_{\psi}^{ \infty})}.
\end{align*}
\end{proof}
Finally, using a reformulation of the Airy kernel as in Section \ref{subsec:heuristicasymptotics}, and a new change of variables $\tilde{z}\leftarrow  - z \sigma(\theta) \log{q}$, and likewise for $\tilde{w}$ and $\tilde{w}'$, one gets
$$\det(I+K'_{x,\infty})  = \det(I-K_{\mathrm{Ai}})_{\mathbb{L}^2(-x, +\infty)},$$
which finishes the proof of Theorem \ref{thm:fluctuationsrarefactionfan}.

\subsection{Proof of Theorem \ref{thm:fluctuationsparticle1}}\label{sec:sixtwo}
The condition $R<1$ ensures that there exists a solution $\theta_0 >0$ to the equation
$$ \kappa_{q, q, R}(\theta) =0.$$
The condition $R>R_{min}(q)$ ensures that the solution $\theta_0$ is such that $q^{\theta_0} >\frac{2q}{1+q}$. Indeed, given the definition of $\kappa_{q, \nu, R}(\theta)$ in (\ref{eq:expressionforkappa}), $\theta_0$ satisfies
$$ \frac{\Psi_q''(\theta_0+1)}{q\Psi_q''(\theta_0)}  = \frac{1-R}{R}.$$
If we set $\theta_{max} = \log_q(2q/(1+q))$, then
$$  \frac{\Psi_q''(\theta_{max}+1)}{q\Psi_q''(\theta_{max})}  = \frac{1-R_{min}(q)}{R_{min}(q)}. $$
Since the function $\theta \mapsto \Psi_q''(\theta+1)/\Psi_q''(\theta)$ is increasing on $\R_+$, the condition $R>R_{min}(q)$ implies that $\theta_0<\theta_{max}$ and equivalently $q^{\theta_0} >\frac{2q}{1+q}$.

If we set $\zeta = -q^{-\pi(\theta_0)t -t^{1/3}\sigma(\theta_0) x}$, then
\begin{equation*}
\lim_{t\to\infty} \EE \left[\frac{1}{\left(\zeta q^{x_{1}(t)+1};q\right)_{\infty}}\right] = \lim_{t\to \infty} \mathbb{P}\left(\frac{x_1(t) - \pi(\theta_0)t}{\sigma(\theta_0) t^{1/3}} \leqslant x\right).
\end{equation*}
The $e_q$-Laplace transform $\EE \left[\frac{1}{\left(\zeta q^{x_{1}(t)+1};q\right)_{\infty}}\right] $ is the Fredholm determinant of a kernel written in terms of $f_0$ exactly as in  (\ref{eq:kernelexponentialform}) with the only modification that the integrand should be multiplied by
$$ \left(\frac{(\nu q^W ; q)_{\infty}}{( q^W ; q)_{\infty}} \right)\Big/ \left( \frac{(\nu q^Z ; q)_{\infty}}{( q^Z ; q)_{\infty}} \right).$$
This additional factor does not perturb the rest of the asymptotic analysis, and disappears in the limit when we rescale the variables around $\theta$. Since the condition $q^{\theta_0}>2q/(1+q)$ is satisfied, Theorem \ref{thm:fluctuationsparticle1} follows from the proof of Theorem \ref{thm:fluctuationsrarefactionfan}.

\subsection{Proofs of Lemmas about properties of $f_0$}
\label{subsec:technicalproofs}

\begin{proof}[Proof of Lemma \ref{lem:thirdderivativegeneral}]
With $R+L=1$, the expression for $f_0'''(\theta)$ in Equation (\ref{eq:thirdderivativegeneral}) is linear in $R$. Hence we may prove the positivity only for the extremal values, i.e. $R=1$ and $R=0$.

We first prove that the function
$$\theta \in \R_{>0} \mapsto \frac{\Psi_q'''(\theta)}{\Psi_q''(\theta)}$$
is strictly increasing. We show that the derivative is positive, that is for any $\theta>0$,
$$\Psi_q''''(\theta)\Psi_q''(\theta) > \left(\Psi_q'''(\theta) \right)^2.$$
Using the series representation for the derivatives of the $q$-digamma function (\ref{eq:digammaderivatives}), this is equivalent to
\begin{equation}
\sum_{n,m\geqslant 1} \frac{n^4 \alpha^n}{1-q^n} \frac{m^2 \alpha^m}{1-q^m} > \sum_{n,m\geqslant 1}\frac{n^3 \alpha^n}{1-q^n} \frac{m^3 \alpha^m}{1-q^m},
\label{eq:ineqdoublesums}
\end{equation}
for $\alpha\in(0,1)$.
Each side of (\ref{eq:ineqdoublesums}) is a power series in $\alpha$, and we claim that the inequality holds for each coefficient. Indeed, keeping only the coefficient of $\alpha^k$, we have to prove that
\begin{equation}
\sum_{n=1}^{k-1} \frac{n^4 (k-n)^2}{(1-q^n)(1-q^{k-n})} \geqslant \sum_{n=1}^{k-1} \frac{n^3 (k-n)^3}{(1-q^n)(1-q^{k-n})},
\label{eq:coefficient}
\end{equation}
with strict inequality for at least one coefficient.
Symmetrizing the left-hand-side, the inequality is equivalent to
$$ \sum_{n=1}^{k-1} \frac{n^2 (k-n)^2}{(1-q^n)(1-q^{k-n})} \frac{n^2 + (k-n)^2}{2} \geqslant \sum_{n=1}^{k-1} \frac{n^2 (k-n)^2}{(1-q^n)(1-q^{k-n})} n(k-n),$$
which clearly holds, with strict inequality  for $k\geqslant 3$.

\textbf{Case $R=1$.} In that case, we have to prove that
$$ \Psi_q'''(\theta+V) - \Psi_q''(\theta+V) \frac{\Psi_q''(\theta)- \Psi_q''(\theta+V)}{\Psi_q'(\theta)- \Psi_q'(\theta+V)}<0.$$
Using Cauchy mean value theorem, the ratio can be rewritten as
$$\frac{\Psi_q''(\theta)- \Psi_q''(\theta+V)}{\Psi_q'(\theta)- \Psi_q'(\theta+V)}  = \frac{ \Psi_q'''(\tilde{\theta})}{\Psi_q''(\tilde{\theta})}, $$
for some $\tilde{\theta}\in (\theta, \theta+V)$.
Since $\Psi_q''(x)<0$ for $x \in (0, +\infty)$, the inequality reduces to
$$ \frac{\Psi_q'''(\theta+V)}{\Psi_q''(\theta+V)} >\frac{\Psi_q'''(\tilde{\theta})}{\Psi_q''(\tilde{\theta})},$$
which is true by the first part of the proof.

\textbf{Case $R=0$.} In that case, we have to prove that
$$ \Psi_q'''(\theta) - \Psi_q''(\theta) \frac{\Psi_q''(\theta)- \Psi_q''(\theta+V)}{\Psi_q'(\theta)- \Psi_q'(\theta+V)}>0.$$
Using the same argument, one is left with proving
$$ \frac{\Psi_q'''(\theta)}{\Psi_q''(\theta)} <\frac{\Psi_q'''(\tilde{\theta})}{\Psi_q''(\tilde{\theta})},$$
which is already done as well.

The proof also applies to the $\nu=0$ case, since the $\nu$ in the denominator in Equation (\ref{eq:thirdderivativegeneral}) can be cancelled by a factor $\nu$ coming out from the $q$-digamma function.
\end{proof}

\begin{proof}[Proof of Proposition \ref{prop:steepdescentC}]
It suffices to prove that for $u\in (0, \pi)$,
$$\frac{\rm{d}}{\rm{d}u} \Real\big[f_0(W(u))\big]>0.$$
We have
$$ \frac{\rm{d}}{\rm{d}u} \Real[f_0(W(u))] = \Real\left[ \frac{\rm{d} W(u)}{\rm{d}u} f_0'(W(u))\right] =  \Imag\left[\frac{1}{\log q} \frac{(1-\alpha)e^{iu}}{1-(1-\alpha) e^{iu}} f_0'(W(u))\right]. $$
We use the linear dependence of $f_0$ on $R$ as in the proof of Lemma \ref{lem:thirdderivativegeneral}.

\textbf{Case $R=1$.}  Using (\ref{eq:f0primegeneral}), one needs to prove that
$$ \Imag\left[ \frac{\Psi_q'(W(u)+1)}{(\log q)^2}  \frac{1-q^{W(u)}}{q^{W(u)}} - \frac{\Psi_q'(\theta +1)}{(\log q)^2}  \frac{1-q^{W(u)}}{q^{W(u)}}+ \frac{\Psi_q''(\theta+1)}{(\log q)^3} (1-\alpha)\frac{1-q^{W(u)}}{q^{W(u)}}  \right] >0. $$
Using the series representation of the $q$-digamma function (\ref{eq:seriespsiq}), the last inequality can be written as
$$ \Imag\left[
\sum_{k=1}^{\infty} \frac{(1-\alpha)e^{iu}}{1-(1-\alpha)e^{iu}} \left(
\frac{(1-(1-\alpha) e^{iu})q^k}{(1-(1-(1-\alpha) e^{iu})q^k)^2} - \frac{\alpha q^k}{(1-\alpha q^k)^2 }+ \frac{\alpha q^k (1+\alpha q^k)(1-\alpha)}{(1-\alpha q^k)^3}
\right)
 \right] >0$$
 A computation -- painful by hand, but easy for Mathematica -- shows that the left-hand-side can be rewritten as
\begin{equation}
\sum_{k=1}^{\infty}  \frac{4 \sin(u) \sin^2(u/2) (1-\alpha)^2 \alpha q^{k}(1-(2-\alpha)q^k) h(\alpha,q^k,u)}{\vert 1-(1-\alpha)e^{iu} \vert^2 \vert 1-(1-(1-\alpha) e^{iu})q^k\vert^4 (1-\alpha q^k)^3 },
\label{eq:steepdescentsamedenom}
\end{equation}
where
\begin{equation*}
 h(\alpha,q,u) = 1-\alpha q \Big(4- \alpha \big( 2+ 2q (1-\alpha) + q^2(2 - q) (1+(1-\alpha)^2) \big) \Big)
  + 2(1-\alpha)\alpha^2 q^2 (1-q)^2 \cos(u).
\end{equation*}
For any $u\in(0,\pi)$, $\cos(u)\geqslant -1$, hence
 $$
 h(\alpha,q,u) \geqslant 1-\alpha q (2-\alpha) \left( 2 - \alpha q^2 (2-\alpha)(2-q)\right)
  $$
and for any $\alpha\in(0,1), q\in (0,1)$, $1-\alpha q (2-\alpha) \left( 2 - \alpha q^2 (2-\alpha)(2-q)\right)\geqslant 0$.
Thus, if $(2-\alpha)q<1$, each term in (\ref{eq:steepdescentsamedenom}) is positive.

\textbf{Case $R=qL$.} Since $R+L=1$, this case corresponds to $R=q/(1+q)$ and $L=1/(1+q)$. As we have noticed in Remark \ref{rem:ratiosimpler}, we have the simpler expression (\ref{eq:simplecasef0prime}) for $f_0'$ when $R=qL$. Hence it is enough to show that
\begin{equation*}
 \Imag\left[\frac{1-q}{(1+q)(1-\alpha)^2} \left(1-\alpha^2 -  \frac{(1-\alpha)^2}{1-q^{W(u)}} - \alpha^2  \frac{1-q^{W(u)}}{q^{W(u)}}\right)\right] >0
\end{equation*}
or equivalently, that
\begin{equation*}
 \frac{1-q}{(1+q)(1-\alpha)^2}  (1-\alpha)\sin(u) \left(1 - \frac{\alpha^2}{\vert q^{W(u)}\vert^2} \right)>0
\end{equation*}
which is true since $\vert q^{W(u)}\vert \leqslant \alpha$ by assumption.

To conclude, since $f_0$ is linear in $R$, the result is also proved for any value $R\in[q/(1+q), 1]$.
\end{proof}

\begin{proof}[Proof of Proposition \ref{prop:steepdescentD}]
It suffices to show that for $u\in(0,\pi)$,
$$ 0>\frac{\rm{d}}{\rm{d}u} \Real[f_0(Z(u))] = \frac{-1}{\log q} \Imag[f_0'(Z(u))], $$
where
$$Z(u)=\theta+i u /\log(q),\ \ (u\in \R). $$
We use the linear dependence of $f_0$ on $R$ as in the proof of Lemma \ref{lem:thirdderivativegeneral} and Proposition \ref{prop:steepdescentC}.

\textbf{Case $R=1$.} Using (\ref{eq:f0primegeneral}), one has to show that
$$ \Imag\left[ \frac{\Psi_q'(Z(u)+1)}{(\log q)^2} - \frac{\Psi_q''(\theta+1)}{(\log q)^3} \frac{(1-\alpha)^2}{\alpha}\frac{q^{Z(u)}}{1-q^{Z(u)}} \right] >0. $$
Using the series representation of the $q$-digamma function (\ref{eq:seriespsiq}), the last inequality can be written
$$ \Imag\left[ \sum_{k=1}^{\infty} \frac{\alpha e^{iu}q^k}{(1-\alpha e^{iu}q^k)^2} - \frac{\alpha q^k (1+\alpha q^k)}{(1-\alpha q^k)^3}\frac{(1-\alpha)^2 e^{iu}}{1-\alpha e^{iu}} \right] >0. $$
The left-hand-side equals
\begin{equation}
\sum_{k=1}^{\infty} \dfrac{\sin(u) \alpha (1-\alpha q^k)(2-\alpha-\alpha^2 q^k)(1+(\alpha-2)q^k)}{\left|1-\alpha e^{iu}q^k\right|^4(1-\alpha q^k)^3 \vert  1-\alpha e^{iu}\vert^2}.
\label{eq:sumarranged}
\end{equation}
If $(2-\alpha)q < 1$, then for all $k\geqslant 1$, $1+(\alpha-2)q^k\geqslant 0$, and each term in (\ref{eq:sumarranged}) is positive.

\textbf{Case $R=qL$.} Using (\ref{eq:simplecasef0prime}), it is enough to show that
$$ \Imag\left[ \frac{q^{Z(u)}}{1-q^{Z(u)}}\left( 1-\alpha^2 -\frac{(1-\alpha)^2}{1-q^{Z(u)}}\right)-\alpha\right]  >0, $$
which is true since the left-hand-side equals
$$ \frac{2\sin(u) \alpha^2(1-\alpha^2)(1-\cos(u))}{\vert 1-\alpha e^{iu}\vert^2}. $$

To conclude, since $f_0$ is linear in $R$, the result is also proved for any value $R\in[q/(1+q), 1]$.
\end{proof}

\bibliographystyle{amsalpha}
\bibliography{twosidedqhahn.bib}

\providecommand{\bysame}{\leavevmode\hbox to3em{\hrulefill}\thinspace}
\providecommand{\MR}{\relax\ifhmode\unskip\space\fi MR }
\providecommand{\MRhref}[2]{%
  \href{http://www.ams.org/mathscinet-getitem?mr=#1}{#2}
}
\providecommand{\href}[2]{#2}
\begin{thebibliography}{KMHH92}

\bibitem[AAR99]{andrews1999special}
G.~E. Andrews, R.~Askey, and R.~Roy, \emph{Special functions}, Cambridge
  University Press, Cambridge, 1999.

\bibitem[AKK98]{alimohammadi1998exact}
M.~Alimohammadi, V.~Karimipour, and M.~Khorrami, \emph{Exact solution of a
  one-parameter family of asymmetric exclusion processes}, Phys. Rev. E
  \textbf{57} (1998), no.~6, 6370.

\bibitem[AKK99]{alimohammadi1999two}
\bysame, \emph{A two-parametric family of asymmetric exclusion processes and
  its exact solution}, J. Stat. Phys. \textbf{97} (1999), no.~1-2, 373--394.

\bibitem[Bar15]{barraquand2014phase}
G.~Barraquand, \emph{A phase transition for {$q$}-{TASEP} with a few slower
  particles}, Stochastic Process. Appl. \textbf{125} (2015), no.~7, 2674--2699.
  \MR{3332851}

\bibitem[BBCS15]{barraquand2015preparation}
J.~Baik, G.~Barraquand, I~Corwin, and T.~Suidan, \emph{in preparation}.

\bibitem[BC14]{borodin2014macdonald}
A.~Borodin and I.~Corwin, \emph{Macdonald processes}, Probab. Theory and Rel.
  Fields \textbf{158} (2014), no.~1-2, 225--400.

\bibitem[BCF14]{borodin2012free}
A.~Borodin, I.~Corwin, and P.~Ferrari, \emph{Free energy fluctuations for
  directed polymers in random media in 1+ 1 dimension}, Commun. Pure Appl.
  Math. \textbf{67} (2014), no.~7, 1129--1214.

\bibitem[BCPS14]{borodin2014spectral}
A.~Borodin, I.~Corwin, L.~Petrov, and T.~Sasamoto, \emph{Spectral theory for
  interacting particle systems solvable by coordinate bethe ansatz}, arXiv
  preprint arXiv:1407.8534 (2014).

\bibitem[BCS14]{borodin2012duality}
A.~Borodin, I.~Corwin, and T.~Sasamoto, \emph{From duality to determinants for
  q-tasep and asep}, Ann. Probab. \textbf{42} (2014), no.~6, 2314--2382.

\bibitem[BF08]{borodin2008large}
A.~Borodin and P.~L. Ferrari, \emph{Large time asymptotics of growth models on
  space-like paths {I: PushASEP}}, Electr. J. Probab. \textbf{13} (2008),
  no.~50, 1380--1418.

\bibitem[BO12]{borodin2012markov}
A.~Borodin and G.~Olshanski, \emph{Markov processes on the path space of the
  {G}elfand--{T}setlin graph and on its boundary}, J. Funct. Anal. \textbf{263}
  (2012), no.~1, 248--303.

\bibitem[Cor12]{corwin2012kardar}
I.~Corwin, \emph{The {Kardar--Parisi--Zhang} equation and universality class},
  Random Matrices Theory Appl. \textbf{1} (2012), no.~01, 1130001.

\bibitem[Cor14]{corwin2014q}
\bysame, \emph{The q-{H}ahn {B}oson {P}rocess and q-{H}ahn {TASEP}}, Inter.
  Math. Res. Not. (2014), rnu094.

\bibitem[EK09]{ethier2009markov}
S.~N. Ethier and T.~G. Kurtz, \emph{Markov processes: characterization and
  convergence}, vol. 282, John Wiley \& Sons, 2009.

\bibitem[EMZ04]{evans2004factorized}
M.~R. Evans, S.~N. Majumdar, and R.~K.~P. Zia, \emph{Factorized steady states
  in mass transport models}, J. Phys. A \textbf{37} (2004), no.~25, L275.

\bibitem[FV13]{ferrari2013tracy}
P.~L. Ferrari and B.~Vet\H{o}, \emph{Tracy-{W}idom asymptotics for q-{TASEP}},
  to appear in Ann. Inst. H. Poincaré, arXiv:1310.2515 (2013).

\bibitem[GKR10]{gabel2010facilitated}
A.~Gabel, P.~L. Krapivsky, and S.~Redner, \emph{Facilitated asymmetric
  exclusion}, Phys. Rev. Lett. \textbf{105} (2010), no.~21, 210603.

\bibitem[Joh00]{johansson2000shape}
K.~Johansson, \emph{Shape fluctuations and random matrices}, Commun. Math.
  Phys. \textbf{209} (2000), no.~2, 437--476.

\bibitem[KMHH92]{krug1992amplitude}
J.~Krug, P.~Meakin, and T.~Halpin-Healy, \emph{Amplitude universality for
  driven interfaces and directed polymers in random media}, Phys. Rev. A
  \textbf{45} (1992), 638--653.

\bibitem[KPZ86]{kardar1986dynamic}
M.~Kardar, G.~Parisi, and Y.~Zhang, \emph{Dynamic scaling of growing
  interfaces}, Physical Review Letters \textbf{56} (1986), no.~9, 889.

\bibitem[Lee12]{lee2012current}
E.~Lee, \emph{The current distribution of the multiparticle hopping asymmetric
  diffusion model}, J. Stat. Phys. \textbf{149} (2012), no.~1, 50--72.

\bibitem[Lee14]{lee2014fredholm}
\bysame, \emph{Fredholm determinants in the multi-particle hopping asymmetric
  diffusion model}, arXiv preprint arXiv:1410.1447 (2014).

\bibitem[Pov13]{povolotsky2013integrability}
A.~M. Povolotsky, \emph{On the integrability of zero-range chipping models with
  factorized steady states}, J. Phys. A \textbf{46} (2013), no.~46, 465205.

\bibitem[Spo12]{spohn2012kpz}
H.~Spohn, \emph{{KPZ} scaling theory and the semi-discrete directed polymer
  model}, MSRI Proceedings, arXiv:1201.0645 (2012).

\bibitem[SW98]{sasamoto1998one}
T.~Sasamoto and M.~Wadati, \emph{One-dimensional asymmetric diffusion model
  without exclusion}, Phys. Rev. E \textbf{58} (1998), no.~4, 4181.

\bibitem[Tak14]{takeyama2014deformation}
Y.~Takeyama, \emph{A deformation of affine {H}ecke algebra and integrable
  stochastic particle system}, J. Phys. A \textbf{47} (2014), no.~46, 465203,
  19. \MR{3279982}

\bibitem[TW09]{tracy2009asymptotics}
C.~A. Tracy and H.~Widom, \emph{Asymptotics in {ASEP} with step initial
  condition}, Commun. Math. Phys. \textbf{290} (2009), no.~1, 129--154.

\bibitem[Vet15]{veto2014tracy}
B.~Vet\H{o}, \emph{{T}racy-{W}idom limit of q-{H}ahn {TASEP}}, Electron. J.
  Probab. \textbf{20} (2015), no. 102, 1--22.

\end{thebibliography}
\end{document}